\documentclass[12pt]{amsart}
\usepackage{amssymb, amstext, amscd, amsmath}
\usepackage{mathtools, xypic, color, dsfont}

\makeatletter
\def\@cite#1#2{{\m@th\upshape\bfseries%
[{#1\if@tempswa{\m@th\upshape\mdseries, #2}\fi}]}}
\makeatother

\let\secsymb=\S
\mathtoolsset{centercolon} 

\theoremstyle{plain}
\newtheorem{thm}{Theorem}[section]
\newtheorem{cor}[thm]{Corollary}
\newtheorem{prop}[thm]{Proposition}
\newtheorem{lem}[thm]{Lemma}

\theoremstyle{definition}
\newtheorem{rem}[thm]{Remark}

\newtheorem{defn}[thm]{Definition}

\newtheorem{eg}[thm]{Example}

\newcommand{\bB}{{\mathds{B}}}
\newcommand{\bC}{{\mathds{C}}}
\newcommand{\bD}{{\mathds{D}}}
\newcommand{\bF}{{\mathds{F}}}
\newcommand{\bN}{{\mathds{N}}}

\newcommand{\bT}{{\mathds{T}}}
\newcommand{\bZ}{{\mathds{Z}}}

  \newcommand{\A}{{\mathcal{A}}}
  \newcommand{\B}{{\mathcal{B}}}
  \newcommand{\C}{{\mathcal{C}}}
  \newcommand{\D}{{\mathcal{D}}}
  \newcommand{\E}{{\mathcal{E}}}

\renewcommand{\H}{{\mathcal{H}}}
  \newcommand{\I}{{\mathcal{I}}}
  \newcommand{\J}{{\mathcal{J}}}
  \newcommand{\K}{{\mathcal{K}}}
\renewcommand{\L}{{\mathcal{L}}}
  \newcommand{\M}{{\mathcal{M}}}
  \newcommand{\N}{{\mathcal{N}}}
\renewcommand{\O}{{\mathcal{O}}}
\renewcommand{\P}{{\mathcal{P}}}
  
  \newcommand{\R}{{\mathcal{R}}}
\renewcommand{\S}{{\mathcal{S}}}
  \newcommand{\T}{{\mathcal{T}}}
  \newcommand{\U}{{\mathcal{U}}}
  \newcommand{\V}{{\mathcal{V}}}
  
  \newcommand{\X}{{\mathcal{X}}}
  \newcommand{\Y}{{\mathcal{Y}}}
  \newcommand{\Z}{{\mathcal{Z}}}

\renewcommand{\phi}{\varphi}
\newcommand{\upchi}{{\raise.35ex\hbox{\ensuremath{\chi}}}}

\newcommand{\fA}{{\mathfrak{A}}}
\newcommand{\fB}{{\mathfrak{B}}}

\newcommand{\fD}{{\mathfrak{D}}}

\newcommand{\fJ}{{\mathfrak{J}}}
\newcommand{\fK}{{\mathfrak{K}}}
\newcommand{\fL}{{\mathfrak{L}}}
\newcommand{\fM}{{\mathfrak{M}}}
\newcommand{\fN}{{\mathfrak{N}}}
\newcommand{\fR}{{\mathfrak{R}}}

\newcommand{\fU}{{\mathfrak{U}}}

\newcommand{\fs}{{\mathfrak{s}}}
\newcommand{\ft}{{\mathfrak{t}}}
\newcommand{\fu}{{\mathfrak{u}}}

\newcommand{\rA}{{\mathrm{A}}}
\newcommand{\rC}{{\mathrm{C}}}

\newcommand{\qand}{\quad\text{and}\quad}

\newcommand{\qfor}{\quad\text{for}\ }
\newcommand{\qforal}{\quad\text{for all}\ }

\newcommand{\AND}{\text{ and }}
\newcommand{\FOR}{\text{ for }}

\newcommand{\IF}{\text{ if }}

\newcommand{\OR}{\text{ or }}

\newcommand{\ad}{\operatorname{ad}}
\newcommand{\Alg}{\operatorname{Alg}}

\newcommand{\Aut}{\operatorname{Aut}}

\newcommand{\End}{\operatorname{End}}

\newcommand{\id}{{\operatorname{id}}}

\newcommand{\ran}{\operatorname{Ran}}
\newcommand{\spn}{\operatorname{span}}

\newcommand{\AD}{\mathrm{A}(\mathds{D})}
\newcommand{\ADD}{\mathrm{A}(\mathds{D}^2)}
\newcommand{\CT}{{\mathrm{C}(\mathds{T})}}
\newcommand{\rCD}{{\mathrm{C}(\overline{\mathds{D}})}}

\newcommand{\ca}{\mathrm{C}^*}
\newcommand{\cenv}{\mathrm{C}^*_{\text{e}}}

\newcommand{\Fn}{\mathds{F}_n^+}
\newcommand{\Fock}{\ell^2(\Fn)}

\newcommand{\lip}{\langle}
\newcommand{\rip}{\rangle}
\newcommand{\ip}[1]{\langle #1 \rangle}
\newcommand{\bip}[1]{\big\lip #1 \big\rip}
\newcommand{\mt}{\varnothing}
\newcommand{\ol}{\overline}

\newcommand{\sot}{\textsc{sot}}
\newcommand{\wot}{\textsc{wot}}
\newcommand{\ltwo}{{\ell^2}}

\newenvironment{sbmatrix}{\left[\begin{smallmatrix}}{\end{smallmatrix}\right]}

\begin{document}
\strut\vspace{-4ex}

\title[Dilation theory]{Dilation theory, commutant lifting and semicrossed products}

\author[K.R.Davidson]{Kenneth R. Davidson}
\thanks{First author partially supported by an NSERC grant.}
\address{Pure Mathematics Department\\
University of Waterloo\\Waterloo, ON\; N2L--3G1\\CANADA}
\email{krdavids@uwaterloo.ca}

\author[E.G.Katsoulis]{Elias~G.~Katsoulis}
\thanks{Second author was partially supported by a grant from ECU}
\address{ Department of Mathematics\\
University of Athens\\ 
15784 Athens \\GREECE \vspace{-2ex}}
\address{\textit{Alternate address:} Department of Mathematics\\
East Carolina University\\ Greenville, NC 27858\\USA}
\email{katsoulise@ecu.edu}

\begin{abstract}
We take a new look at dilation theory for nonself-adjoint operator algebras.
Among the extremal (co)extensions of a representation,
there is a special property of being fully extremal.  This allows a refinement
of some of the classical notions which are important when one moves away
from standard examples.  We show that many algebras including graph algebras
and tensor algebras of C*-correspondences have the semi-Dirichlet property
which collapses these notions and explains why they have a better dilation theory.
This leads to variations of the notions of commutant lifting and Ando's theorem.
This is applied to the study of semicrossed products by automorphisms, and 
endomorphisms which lift to the C*-envelope. In particular, we obtain several
general theorems which allow one to conclude that semicrossed products of
an operator algebra naturally imbed completely isometrically into the 
semicrossed product of its C*-envelope, and the C*-envelopes of these
two algebras are the same.
\end{abstract}

\subjclass[2000]{47L55, 47L40, 46L05, 37B20}
\keywords{dilation, extremal coextension, semicrossed product, commutant lifting, 
Fuglede property, C*-envelopes, Ando's theorem}

\date{}
\maketitle

\begin{itemize}
\item[{1.}] {Introduction} \hfill{2}\qquad\strut
\item[{2.}] {A review of dilations} \hfill{7}\qquad\strut
\item[{3.}] {Fully Extremal Coextensions} \hfill{11}\qquad\strut
\item[{4.}] {Semi-Dirichlet algebras} \hfill{21}\qquad\strut
\item[{5.}] {Commutant Lifting} \hfill{28}\qquad\strut
\item[{6.}] {A $2 \times 2$ matrix example.} \hfill{39}\qquad\strut
\item[{7.}] {Relative Commutant Lifting and Ando's Theorem} \hfill{46}\qquad\strut
\item[{8.}] {Incidence Algebras} \hfill{56}\qquad\strut
\item[{9.}] {The Fuglede Property} \hfill{65}\qquad\strut
\item[{10.}] {Completely Isometric Endomorphisms} \hfill{69}\qquad\strut
\item[{11.}] {A Review of Semicrossed Products} \hfill{74}\qquad\strut
\item[{12.}] {Dilating Covariance Relations} \hfill{76}\qquad\strut
\item[{13.}] {Further Examples} \hfill{82}\qquad\strut
\item[{}] {References} \hfill{88}\qquad\strut

\end{itemize}

\section{Introduction} \label{S:intro}

This paper is a general study of dilation theory for arbitrary nonself-adjoint
operator algebras.  It began with an attempt to formalize those properties
need to obtain dilation theorems for covariant representations of an
operator algebra and an endomorphism, in order to understand the
semicrossed product and its C*-envelope.
In this paper, we discuss versions of commutant lifting and Ando's theorem
and consider when they allow us to determine the structure of a semicrossed product
and its C*-envelope.
This forced us to revisit basic notions in dilation theory, and to introduce
a notion stronger than that of extremal (co)extensions.
We feel that certain notions in dilation theory are too closely modelled
on what happens for the disk algebra. This algebra has been shown to have
many very strong properties, and they are often not perfectly reflected
in the general case. Certain refinements should be considered to clarify
the various dilation properties in a general context.

\subsection*{Dilation theory}
Dilation theory for a single operator has its roots in the seminal work of
Sz.Nagy \cite{SzN} which is developed in the now classical book that
he wrote with Foia\c{s} \cite{SF}. 
Dilation theory for more general operators was initiated by the deep work
of Arveson \cite{Arv1,Arv2}.
The ideas have evolved over the past six decades.
The basic ideas are nicely developed in Paulsen's book \cite{Paul}. 

In formulating general properties related to commutant lifting and Ando's theorem,
we were strongly motivated, in part, by the general module formulation
expounded by Douglas and Paulsen \cite{DougPaul} and the important study
by Muhly and Solel \cite{MS_hilbert}.  
The language used there is a module theoretic approach, while we will mostly
talk about representations instead.  But the general constructs can, of course,
be formulated in either language.
Douglas and Paulsen focus on Shilov modules as a primary building block.
Muhly and Solel adopt this view, but focus more on a somewhat stronger property
of orthoprojective modules.
They may have gone further, as we do, had they known what we do today. 
We will argue that these are more central to the theory.

Another important influence is the Dritschel--McCullough \cite{DritsMcC}
proof of the existence of Arveson's C*-envelope \cite{Arv1,Arv2}, first established
using different methods by Hamana \cite{Ham}.
They provide a proof strongly influenced by ideas of Agler \cite{Ag}.
What they show is: given a completely contractive 
representation of a unital operator algebra $\A$,
that among all dilations of this representation, there are always certain representations
which are maximal in the sense that any further dilation can only be obtained by
appending a direct sum of another representation.
These dilations always exist, as they show, and they are precisely those representations
which extend to a $*$-representation of the C*-envelope.
It is in this manner that they establish that the existence of the C*-envelope.

This fact was anticipated by Muhly and Solel in \cite{MS_bound} where they show, 
assuming Hamana's theorem, that every representation has a dilation which
is both orthoprojective and orthoinjective.
It is easy to see that this is a reformulation of the maximal dilation property.
Indeed, one can see that a representation $\rho$ is orthoprojective
if and only if it is maximal as a coextension (called an extremal coextension)---meaning 
that any coextension can be obtained only by appending a direct sum 
of another representation.
Dritschel and McCullough proved that these exist as well.
The dual version shows that orthoinjective representations coincide
with the extremal extensions.

An extremal (co)extension of a representation $\rho$ on $\H$ is called \textit{minimal}
provided that the whole space is the smallest reducing subspace containing $\H$.
This is a weaker notion than saying that $\H$ is cyclic.
However, there can be many extremal coextensions which are
minimal but $\H$ is not cyclic.
Among extremal (co)extensions, there are some preferred (co)extensions
which we call \textit{fully extremal} because they satisfy a stronger maximality property.
While in many classical cases, this notion reduces to the usual extremal property,
we argue that in general they are preferred.
The existence of fully extremal (co)extensions is established by an argument
similar to Arveson's proof \cite{Arv3} of the existence of maximal dilations.

\subsection*{Commutant lifting}
The classical commutant lifting theorem was established by
Sz.Nagy and Foia\c{s} \cite{SF_CLT}. Many variations on this theorem
have been established in various contexts for a variety of operator algebras.
Douglas and Paulsen \cite{DougPaul} formulate a version for arbitrary operator
algebras, and we propose a modification of their definition.

Shilov representations of an operator algebra $\A$ are those which are
obtained by taking a $*$-representation of the C*-envelope and restricting
it to an invariant subspace for the image of $\A$.
All extremal coextensions (orthoprojective representations) are Shilov.  
The converse holds in some of the classical situations, but is not valid in general.
As we will argue, the notions of commutant lifting are better expressed in terms
of fully extremal coextensions rather than Shilov coextensions.
Limiting the family of coextensions for which lifting occurs increases the
family of algebras with this property.
Indeed the \textit{strong} version of commutant lifting can only hold when
there is a unique minimal \textit{fully} extremal coextension (of $\rho$).

The Douglas-Paulsen formulation of commutant lifting starts with a
(completely contractive) representation $\rho$ of an operator algebra $\A$, an
operator $X$ in the commutant of $\rho(\A)$, and a Shilov coextension $\sigma$ of $\rho$;
and they ask for a coextension of $X$ to an operator $Y$ of the same norm
commuting with $\sigma(\A)$.  
As remarked in the previous paragraph, this only holds when the minimal
Shilov extension is unique.
We show that this holds when $\A$ is \textit{semi-Dirichlet}, meaning that 
\[ \A^*\A \subset\ol{\A+\A^*} , \]
such as the disk algebra, the non-commutative
disk algebras, and all tensor algebras of graphs and C*-correspondences.
The fact that this large class of popular algebras has this remarkable property
has perhaps kept us from looking further for an appropriate definition of 
commutant lifting in other contexts.

We were also influenced by a different approach of Paulsen and Power
\cite{PP_nest,PP_tensor} and subsequent work of theirs with the first
author \cite{DPP_tree, D_local}.
In this version, the coextension $\sigma$ is not specified, and one looks for
common coextensions $\sigma$ and $Y$ which commute.
We will use extremal coextensions only, rather than arbitrary
Shilov coextensions, with the obvious parallel definitions.
The first version will be called \textit{strong commutant lifting}, and the latter
\textit{commutant lifting}.
A crucial point is that strong commutant lifting turns out to be equivalent
commutant lifting plus uniqueness of the minimal fully extremal coextension.

The intertwining version of commutant lifting proved to be challenging
in this context. The resolution of this problem was critical to obtaining
good dilation theorems for semicrossed products.

\subsection*{Ando's theorem}
Ando's Theorem \cite{Ando} states that if $A_1$ and $A_2$ are 
commuting contractions, then they
have coextensions $V_i$ which are commuting isometries.
For us, an Ando theorem for an operator algebra $\A$ will be formulated as
follows: given a (completely contractive) representation $\rho$ of an operator algebra $\A$
and a contraction $X$ in the commutant of $\rho(\A)$ , there is a fully 
extremal coextension $\sigma$ of $\rho$ and an isometric coextension $V$
of $X$ which commutes with it.
Even in the case of the disk algebra, our formulation is stronger than the
original, as it asks that one of the isometries, say $V_1$, should have the form
\[ V_1 \simeq V_{A_1} \oplus U \]
where $V_{A_1}$ is the minimal isometric coextension
of $A_1$ and $U$ is unitary (Corollary~\ref{C:ando}).

In the classical case of the disk algebra, the universal algebra of a contraction,
the generator of a representation, $A = \rho(z)$, plays a role parallel to the
operator $X$ which commutes with it.  
For this reason, commutant lifting can be applied recursively to
$A$ and $X$, alternating between them, in order to obtain Ando's Theorem.
So in this context, the Sz.Nagy-Foia\c{s} Commutant Lifting Theorem 
\cite{SF_CLT} is equivalent to Ando's Theorem.
But for other algebras, there are two distinct aspects, dilating $\rho$ to
an extremal coextension and simultaneously coextending $X$ to a commuting contraction,
and on the other hand coextending $X$ to an isometry and 
simultaneously coextending $\rho$ to a commuting representation. 

Paulsen and Power \cite{PP_tensor} formulate Ando's theorem as a dilation result for
$\A \otimes_{\text{min}} \AD$, or equivalently that 
\[ \A \otimes_{\text{min}} \AD = \A \otimes_{\text{max}} \AD . \]
Such a result holds for a wide class of CSL algebras \cite{PP_nest,DPP_tree, D_local}.
The stronger version of commutant lifting only holds in a restricted class \cite{MS_hilbert}.
See \cite[chapter 5]{MS_hilbert} for a discussion of the differences.
In our language, they start with a representation $\rho$ and a commuting
contraction $X$, and seek a maximal dilation $\pi$ and a simultaneous dilation of $X$
to a unitary $U$ commuting with $\pi(\A)$.
We show that this is equivalent to the weaker property of obtaining some
coextension $\sigma$ of $\rho$ and an isometric coextension $V$ of $X$
which commute.  
This is only `half' of Ando's theorem in our formulation.

Another property that we will consider is an analogue of the Fuglede theorem:
that the commutant of a normal operator is self-adjoint.  
We formulate this for an operator algebra $\A$ with C*-envelope $\cenv(\A)$
as saying that for any $*$-representation $\pi$ of $\cenv(\A)$, the commutant
of $\pi(\A)$ coincides with the commutant of $\pi(\cenv(\A))$.
We show that a number of operator algebras have this property including
all function algebras, the non-commutative disk algebras and more
generally the tensor algebras of all finite directed graphs.

\subsection*{Semicrossed products}
If $\A$ is a unital operator algebra and $\alpha$ is a completely isometric
endomorphism, then the semicrossed product \[\A \times_\alpha \bZ_+\]
is the operator algebra that encodes the system $(\A,\alpha)$ in the sense that
its (completely contractive) representations are in bijective correspondence 
with the covariant representations of the dynamical system.
Concrete versions of these algebras occur in work or Arveson
\cite{Arv_meas,ArvJ}.
When $\A$ is a C*-algebra, the abstract semicrossed product was 
defined by Peters \cite{Pet}.
The extension to arbitrary nonself-adjoint operator algebras is straightforward.

The structure of these semicrossed products can often be better
understood by showing that the C*-envelope is a full C*-algebra crossed product.
Peters \cite{Pet3} does this for the semicrossed product that encodes
a discrete dynamical system.  The operator algebras of multivariable
dynamical systems is developed in \cite{DavKatsMem}.  The C*-envelope
is further explained in \cite{DRoy}, extending Peter's analysis to this context.
More recently, the second author and Kakariadis \cite{KK} develop an
important generalization of these techniques to very general operator algebras.
They show that for nonself-adjoint operator algebras,
one first should try to imbed a general semicrossed product into a C*-semicrossed product.
They show how to accomplish this, and demonstrate that often the two
operator algebras have the same C*-envelope.

When $\alpha$ is the identity map, the semicrossed product is closely tied
to commutant lifting.  What we show here is that commutant lifting theorems
can be sufficient to understand other semicrossed products provided the
algebra has some other nice properties.
We concern ourselves only with endomorphisms that extend to $*$-endomorphisms
of the C*-envelope.  When $\A$ satisfies the Ando property, every
semicrossed product by a completely isometric automorphism is
isometrically isomorphic to a subalgebra of the semicrossed product of $\cenv(\A)$.
These general techniques recover various results in the literature about the 
structure of crossed products, especially of the non-commutative disk algebras 
\cite{DavKatsdilation} and tensor algebras of C*-correspondences \cite{KK}.
To our knowledge, all of these results used the strong commutant lifting property
(SCLT), which implies uniqueness of fully extremal extensions. Indeed, the
theorems relate to algebras with a row contractive condition, the most general of
which are tensor algebras of C*=correspondences. Our new result requires
only commutant lifting, and applies much more widely.

With a stronger commutant lifting property and the Fuglede property, we can do the
same for endomorphisms which lift to the C*-envelope.  
This applies, in particular, for the disk algebra (which has
all of the good properties studied here).  This recovers our results \cite{DKdisk} 
for the semicrossed product of $\AD$ by an endomorphism of the form 
$\alpha(f) = f\circ b$, 
in the case where $b$ is a non-constant finite Blaschke product.
These general results that imbed a semicrossed product into a C*-algebra
crossed product are actually dilation theorems.
Typically one proves a dilation theorem first, and then deduces the
structure of the C*-envelope.
However the papers \cite{KK, DKdisk} actually compute the C*-envelope first
and deduce the dilation theorem afterwards.
One of the original motivations for this paper was an attempt to
identify the C*-envelope of a semicrossed product using general dilation
properties such as commutant lifting.
Three such theorems are obtained in section~\ref{S:covariance}.

\section{A review of dilations} \label{S:diln}

In this paper, an operator algebra will be a unital abstract operator algebra $\A$
in the sense of Blecher, Ruan and Sinclair \cite{BRS}.  
A \textit{representation} of $\A$ will mean a unital completely contractive 
representation $\rho$ on some Hilbert space $\H$.
An \textit{extension} of $\rho$ is a representation $\sigma$ on a Hilbert space
$\K = \H^\perp \oplus \H$ which leaves $\H$ invariant, and thus has the form
\[
 \sigma(a) = \begin{bmatrix} \sigma_{11}(a) & 0 \\ \sigma_{12}(a) & \rho(a) \end{bmatrix} .
\]
Dually, a \textit{coextension} of $\rho$ is a representation $\sigma$ on a Hilbert space
$\K = \H \oplus \H^\perp$ which leaves $\H^\perp$ invariant, and thus has the form
\[
 \sigma(a) = \begin{bmatrix} \rho(a) & 0 \\ \sigma_{12}(a) & \sigma_{22}(a) \end{bmatrix} .
\]
A \textit{dilation} of $\rho$ is a representation $\sigma$ on a Hilbert space $\K$ containing $\H$
so that $\rho(a) = P_\H \sigma(a)|_\H$.  A familiar result of Sarason \cite{Sar} shows 
that $K$ decomposes as $\K = \H_- \oplus \H \oplus \H_+$ so that
\[
 \sigma(a) =
 \begin{bmatrix} 
 \sigma_{11}(a) & 0 & 0 \\ 
 \sigma_{21}(a) & \rho(a) & 0 \\ 
 \sigma_{31}(a) & \sigma_{32}(a) &\sigma_{33}(a)
 \end{bmatrix} .
\]

A representation $\rho$ is an \textit{extremal coextension} if whenever $\sigma$ is
a coextension of $\rho$, it necessarily has the form $\sigma = \rho \oplus \sigma'$.
That is, if $\H$ is a subspace of $\K$ and $\sigma$ is a representation of $\A$
on $\K$ which leaves $\H^\perp$ invariant and $P_\H \sigma(a)|_\H = \rho(a)$ for
$a \in \A$, then $\H$ reduces $\sigma$.
Similarly, a representation $\rho$ is an \textit{extremal extension} if whenever $\sigma$ is
an extension of $\rho$, it necessarily has the form $\sigma = \rho \oplus \sigma'$.
That is, if $\H$ is a subspace of $\K$ and $\sigma$ is a representation of $\A$
on $\K$ which leaves $\H$ invariant and $P_\H \sigma(a)|_\H = \rho(a)$ for
$a \in \A$, then $\H$ reduces $\sigma$.
Finally, a representation $\rho$ is an \textit{extremal representation} or a 
\textit{maximal representation} if whenever $\sigma$ is a dilation of $\rho$, 
it necessarily has the form $\sigma = \rho \oplus \sigma'$.
That is, if $\H$ is a subspace of $\K$ and $\sigma$ is a representation of $\A$
on $\K$ so that $P_\H \sigma(a)|_\H = \rho(a)$ for $a \in \A$, then $\H$ reduces $\sigma$.
A dilation $\sigma$ of $\rho$ is an \textit{extremal dilation} or a 
\textit{maximal dilation} of $\rho$ if it is a maximal representation.

\subsection*{Hilbert modules}
In the module language espoused by Douglas and Paulsen in \cite{DougPaul}, 
a representation $\rho$ makes the Hilbert space $\H$ into a left $\A$ module $\H_\rho$
by $a\cdot h := \rho(a)h$ for $a \in \A$ and $h \in \H$.  
If $\K = \M \oplus \H$ and $\sigma$ is a representation of $\A$ on $\K$
which leaves $\M$ invariant, so that with respect to the decomposition 
$\K = \H \oplus \M$ of $\sigma$ is
\[
 \sigma(a) = \begin{bmatrix} \sigma_{11}(a) & 0 \\ \sigma_{12}(a) & \sigma_{22}(a) \end{bmatrix} ,
\]
then $\K_\sigma$ is an $\A$-module with $\M_{\sigma_{22}}$ as a submodule
and $\H_{\sigma_{11}}$ as a quotient module, and there is a short exact sequence
\[ 0 \to \M_{\sigma_{22}} \to \K_\sigma \to \H_{\sigma_{11}} \to 0 .\]
Here all module maps are completely contractive.
So an extension $\sigma$ of $\sigma_{22}$ on $\M$ corresponds to larger 
Hilbert module $\K_\sigma$ containing $\M_{\sigma_{22}}$
as a submodule; and a coextension $\sigma$ of $\sigma_{11}$ corresponds to the
Hilbert module $\K_\sigma$ having $\H_{\sigma_{11}}$ as a quotient module.

A module $\P_\rho$ is \textit{orthoprojective} if whenever there is a short exact sequence
\[
 \xymatrix{
 0 \ar[r] & \M_{\sigma_{22}} \ar[r]^\iota & \K_\sigma \ar[r]^q & \P_\rho \ar[r] &  0} ,
\]
where the module maps are completely contractive, there is a completely contractive
module map $\phi : \P \to \K$ so that $\K_\sigma = \M \oplus \phi(\P)$
(as an $\A$-module).
It is not difficult to see that this is equivalent to saying that $\rho$ is an
extremal coextension.  The term orthoprojective was coined by
Muhly and Solel \cite{MS_hilbert}, and we think that it is superior to the
Douglas-Paulsen terminology of hypo-projective because of its more positive aspect.
Similarly, one can define \textit{orthoinjective} modules, and observe that
they are equivalent to extremal extensions.
A maximal dilation corresponds to a module which is both orthoprojective
and orthoinjective. 

\subsection*{The C*-envelope}
Every unital operator algebra $\A$ has a completely isometric representation 
$\iota$ on a Hilbert space $\H$ so that the C*-algebra $\ca(\iota(\A)) =: \cenv(\A)$
is minimal in the sense that if $\sigma$ is any other completely isometric 
representation on a Hilbert space $\H'$, then there is a unique $*$-homomorphism
$\pi$ of $\ca(\sigma(\A))$ onto $\cenv(\A)$ so that the following diagram commutes:
\[
 \xymatrix{
 \A \ar[r]^(.45)\iota \ar[d]_\sigma & \cenv(\A)\\ \ca(\sigma(\A)) \ar[ur]^(.45)\pi
 }
\]
The C*-envelope was described by Arveson \cite{Arv1,Arv2} in his seminal work on 
non-commutative dilation theory.  Its existence was established by Hamana \cite{Ham}.

Muhly and Solel \cite{MS_bound} show that maximal dilations exist 
by invoking Hamana's theorem.  They accomplish this by showing:

\begin{thm}[Muhly-Solel] \label{T:MS}
A representation is maximal if and only if it is both orthoprojective and orthoinjective.
Equivalently, a representation is maximal if and only if it is 
both an extremal coextension and an extremal extension.
\end{thm}

Dritschel and McCullough \cite{DritsMcC} establish the existence of maximal 
dilations directly based on ideas of Agler \cite{Ag}.
In this way, they provide a new and more revealing proof of the existence
of the C*-envelope.
In fact, they show that every representation has an extension which is extremal;
and dually also has a coextension which is extremal.
In particular, the maximal representations of $\A$ are precisely those representations
which extend to $*$-representations of $\cenv(\A)$.

Arveson \cite{Arv_choq} provides a refinement of this result in the separable case
by showing that there are sufficiently many irreducible maximal representations,
which are the boundary representations that Arveson introduced in \cite{Arv1}
as an analogue of the Choquet boundary of a function algebra.
We will not require this strengthened version.

\subsection*{Extremal versus Shilov}
Douglas-Paulsen \cite{DougPaul} and Muhly-Solel \cite{MS_hilbert} focus on
\textit{Shilov modules}.  
One starts with a $*$-representation $\pi$ of $\cenv(\A)$ on a Hilbert space $\K$.
Consider $\K_\pi$ as an $\A$-module.  
A submodule $\H$ of $\K_\pi$ is a Shilov module. 
It is easy to deduce from the discussion above that every 
orthoprojective module is Shilov.
Unfortunately, the converse is false.  We provide an example below.
In the language of representations, a Shilov module corresponds to a 
representation which has an extension to a maximal representation.
However it may still have proper coextensions.

Shilov modules are useful because every completely contractive $\A$-module $\M$
has a finite resolution of the form 
\[
  0  \to \S_1 \to \S_2 \to \M \to 0 ,
\]
where $\S_1$ and $\S_2$ are Shilov.
Using orthoprojective modules, one can obtain
\[ \S_2 \to \M \to 0 \]
with $\S_2$ orthoprojective.  But since submodules do not inherit
this extremal property, one does not obtain a short exact sequence.
Indeed, while this procedure can be iterated, there need be no
finite resolution.  This occurs, for example, in the theory of 
commuting row contractions due to Arveson \cite[\secsymb 9]{Arv_curv}.
However Arveson also argues that, in his context, these are the
natural resolutions to seek.

Our view is that it is the extremal coextensions rather than Shilov coextensions 
which play the role in dilation theory that best models the classical
example of the unilateral shift as an isometric model of the disc algebra.

\begin{eg} \label{Eg:ncdisk}
Consider the non-commutative disk algebra $\fA_n$.
It is the unital subalgebra of the Cuntz algebra $\O_n$ generated as a unital nonself-adjoint
subalgebra by the canonical isometric generators $\fs_1,\dots,\fs_n$ of $\O_n$.
A representation $\rho$ of $\fA_n$ is determined by $A_i = \rho(\fs_i)$, and it is
completely contractive if and only if 
\[ A = \big[ A_1\ \dots \ A_n \big] \] 
is a contraction as an operator from $\H^{(n)}$ to $\H$ \cite{Pop2}.
The Frazho-Bunce-Popescu  dilation theorem \cite{Fra,Bun,Pop1} states that
$A$ has a coextension to a row isometry.
Conversely, it is clear that any coextension of a row isometry must be obtained
as a direct sum.  Thus these row isometric representations are precisely the
extremal coextensions and correspond to orthoprojective modules.
The Wold decomposition \cite{DP1} shows that this row isometry decomposes as 
a direct sum of a Cuntz row unitary and a multiple of the left regular representation of
the free semigroup $\bF_n^+$ on Fock space.  This representation generates
the Cuntz-Toeplitz C*-algebra, and thus is not a maximal representation.
It can be extended to a maximal dilation in many explicit ways \cite{DP1}.
It is clear in this case that every $*$-representation of $\O_n$ sends
\[ \fs = \big[ \fs_1 \ \dots \ \fs_n \big] \]
to a row unitary, and the restriction to any invariant subspace is a row isometry. 
Thus every Shilov module is orthoprojective.
\end{eg}

\begin{eg}  \label{Eg:DA}
Let $\A_n$ be the universal algebra of a row contraction with commuting coefficients.
This algebra was studied extensively by Arveson beginning in \cite{Arv3}.
The basic von Neumann inequality was proven much earlier by Drury \cite{Drury},
but the full version of the dilation theorem was due to M\"uller and Vascilescu \cite{MV}
and later, Arveson.
Arveson further showed that the multipliers $S_1,\dots,S_n$ on symmetric Fock space
$H^2_n$ in $n$ variables form a canonical model for $\A_n$.  
Also $H^2_n$ is a reproducing kernel Hilbert space, and $\A_n$ is the
algebra of continuous multipliers. 
The C*-algebra generated by these multipliers is the C*-envelope of $\A_n$ \cite{Arv3}. 

The dilation theorem shows that every commuting row contraction has a 
coextension to a direct sum $S_i^{(\alpha)} \oplus U_i$ where $\alpha$ is some cardinal
and $U_i$ are commuting normal operators satisfying 
\[ \sum_{i=1}^n U_iU_i^* = I .\]
These are precisely the extremal coextensions and determine the orthoprojective 
modules.
Surprisingly they are also the maximal representations.
So while one can dilate in both directions to obtain a maximal dilation of
a representation $\rho$, only coextensions are required.

However, no submodule of the symmetric Fock space is orthoprojective.
They are all Shilov, but none model the algebra in a useful way.
Davidson and Le \cite[Example 4.1]{DavLe} provide an explicit example of this 
phenomenon in their paper on the commutant lifting theorem for $\A_n$.
\end{eg}

\section{Fully Extremal Coextensions} \label{S:fully extremal}

There is a natural partial order $\prec$ on dilations: say that $\rho \prec \sigma$
if $\sigma$ acts on a Hilbert space $\K$ containing a subspace $\H$ so that
$P_\H \sigma|_\H$ is unitarily equivalent to $\rho$.
There is also a partial order on extensions $\prec_e$: say that $\rho \prec_e \sigma$
if $\sigma$ acts on a Hilbert space $\K$ containing an invariant subspace $\H$ so that
$\sigma|_\H$ is unitarily equivalent to $\rho$.
Similarly, for coextensions, say that $\rho \prec_c \sigma$
if $\sigma$ acts on a Hilbert space $\K$ containing a co-invariant subspace $\H$ 
so that $P_\H \sigma|_\H$ is unitarily equivalent to $\rho$.

Dritschel and McCullough \cite{DritsMcC} establish the existence of extremals
dominating $\rho$ in each of these classes.
We want something a little bit stronger.
It is possible for an extremal coextension $\sigma$ of $\rho$ to have a proper 
extension which is also a coextension of $\rho$, so that $\sigma$ is not extremal
in the partial order $\prec$.  We provide an example shortly.
We will require knowing that $\rho$ has a coextension which is extremal with
respect to $\prec$.

\begin{defn} \label{D:fully extremal}
If $\rho$ is a representation of $\A$, say that a coextension $\sigma$ of $\rho$
is \textit{fully extremal with respect to $\rho$} if whenever $\sigma \prec \tau$
and $\rho \prec_c \tau$, then $\tau = \sigma \oplus \tau'$.
Similarly we define an extension $\sigma$ of $\rho$ to be 
\textit{fully extremal with respect to $\rho$} if whenever $\sigma \prec \tau$
and $\rho \prec_e \tau$, then $\tau = \sigma \oplus \tau'$.
\end{defn}

\begin{eg}  \label{Eg:zigzag}
Fix an orthonormal basis $e_1,\dots,e_n$ for $\bC^n$,
and let $E_{ij}$ be the canonical matrix units.
Consider the subalgebra $\A$ of $\fM_n$ spanned by the diagonal algebra
\begin{gather*}
 \fD_n = \spn\{E_{ii} : 1 \le i \le n\}
 \shortintertext{and} 
 \spn\{E_{ij} : |i-j|=1,\ j \text{ odd }\} .
\end{gather*}
This is a reflexive operator algebra with invariant subspaces 
\begin{gather*}
 \bC e_{2i} \qfor 1 < 2i \le n
 \shortintertext{and} 
 L_{2i+1} = \spn\{ e_{2i},e_{2i+1},e_{2i+2}\} \qfor  1 \le 2i+1 \le n , 
\end{gather*}
where we ignore $e_0$ and $e_{n+1}$ if they occur.
The elements of $\A$ have the form
\[ 
 A = \begin{bmatrix} a_{11} & 0 & 0 & 0 & 0 & \dots\\
 a_{21} & a_{22} & a_{32} & 0 & 0 & \dots\\
 0 & 0 & a_{33} & 0 &0 &  \dots\\
 0 & 0  & a_{43} & a_{44} & a_{45} & \dots\\
 0 & 0 & 0 & 0 & a_{55} & \ddots\\
 \vdots &  \vdots &  \vdots &  \vdots & \ddots & \ddots
 \end{bmatrix}
\]

Consider the representation $\rho(A) = a_{11}$, the 1,1 matrix entry of $A$.
Since $\bC e_1$ is coinvariant, this is a representation.
The compression $\sigma_2$ of $A$ to $M_2 = \spn\{e_1,e_2\}$ is 
a coextension of $\rho$ given by
\[
 \sigma_2(A) = P_{M_2}A|_{M_2} = \L_1  \begin{bmatrix} a_{11} & 0 \\ a_{21} & a_{22} \end{bmatrix}
\]
This is readily seen to be an extremal coextension of $\rho$.
It is minimal in the sense we use: it contains no proper reducing subspace
containing $\H_\rho = \bC e_1$, and is also minimal in the 
sense that $\H_\sigma = \sigma(\A) \H_\rho$.

However $\sigma_2$ is not fully extremal.
Let 
\[ M_k = \spn\{e_i : 1 \le i \le k \}\]
and set 
\[\sigma_k(A) = P_{M_k} A|_{M_k} .\]
Then $\sigma_{2i+1}$ is an extension of $\sigma_{2i}$ and
$\sigma_{2i+2}$ is a coextension of $\sigma_{2i+1}$.
All are coextensions of $\rho$.  Each $\sigma_{2i}$ is an extremal coextension of $\rho$,
as is $\sigma_n = \id$ even if $n$ is even.
Moreover all are minimal in that they have no proper reducing subspace containing $\bC e_1$.
Only $\sigma_n$ is fully extremal.
Note that to get from $\rho$ to $\sigma_n$, one must alternately coextend and extend
$n-1$ times if at each stage, you take a classical minimal extension or coextension.

One can also define an infinite version of this algebra where it takes a countable
number of steps to attain the fully extremal coextension.
\end{eg}

\begin{eg}  \label{Eg:disk}
Let $\AD$ be the disk algebra.  
A representation of $\AD$ is determined by $T = \rho(z)$, and it is completely
contractive if and only if $\|T\|\le1$.
Every contraction coextends to a unique minimal isometry.
So the extremal coextensions must be isometries.
But conversely, it is easy to see that any contractive coextension of an isometry
is obtained by adding a direct sum.  So when $T$ is an isometry, $\rho$ is 
an extremal coextension.
The minimal isometric dilation $V$ of $T$ yields a fully extremal coextension because
the range of $V$ together with $\H_\rho$ spans the whole space.
Any (contractive) dilation of $V$ must map the new subspace orthogonal to the range of $V$.
So if it is not a summand, the range will not be orthogonal to $\H_\rho$, so it won't
be a coextension of $\rho$.

The extremal coextensions of $\rho$ correspond to all isometric coextensions of $T$,
namely $V \oplus W$ where $V$ is the minimal isometric dilation and $W$ is any isometry.
But if $W$ isn't unitary, it can be extended to a unitary.  This extension is still a
coextension of $T$.  So the fully extremal coextensions correspond to $V \oplus U$
where $U$ is unitary.

Similarly, the maximal dilations of $\rho$ correspond to unitary dilations of $A$.
The restriction of a unitary to an invariant subspace is an isometry.
So a Shilov representations are extremal coextensions.
In particular, a minimal Shilov dilation of $\rho$ is a fully extremal coextension.
\end{eg}

\begin{eg}  \label{Eg:bidisk}
Let $\A = \ADD$ with generators $z_1$ and $z_2$.  
Then a completely contractive representation is determined by a pair of
commuting contractions $A_i = \rho(z_i)$.  
By Ando's Theorem \cite{Ando}, 
every commuting pair coextends to a pair of commuting isometries.
It is clear that any coextension of isometries to a larger space can only 
be obtained by adding a direct summand.  
So the extremal coextensions are the commuting isometries.
It is also clear that any restriction to an invariant subspace is still isometric.

Moreover every pair of commuting isometries extends to a pair of
commuting unitaries.  These are the maximal dilations, and determine
a $*$-representation of 
\[ \cenv(\A) = \rC(\bT^2) .\]
The restriction of a unitary to an invariant subspace is an isometry.
So every Shilov module is orthoprojective (an extremal coextension).

What we wish to point out is that extremal coextensions of a representation
$\rho$  of $\ADD$ need not be fully extremal.
Let 
\[ A_1=A_2=0 \qquad\text{acting on } \H=\bC .\]
Identify $\H$ with $\bC e_{0,0}$ in 
\[ \K = \spn \{e_{m,n} : m,n \ge 0 \} ,\]
where $\{e_{m,n}: m,n \ge 0 \}$ is an orthonormal basis.
Then it is clear that there is a coextension of $A_i$ to the commuting isometries 
\[ S_1 = S \otimes I \qand  S_2 = I \otimes S , \]
where $S$ is the unilateral shift.
Let $\sigma$ be the corresponding coextension of $\rho$.
This is an extremal coextension because the $S_i$ are isometries.

Enlarge this orthonormal basis further to obtain a space
\[ \L = \spn\{e_{k,l} : \max\{k,l\} \ge 1 \OR k=l=0 \} \]
containing $\K$.
Let $T_i$ be the commuting isometries given by
\[ T_1e_{m,n}=e_{m+1,n} \qand  T_2e_{m,n} = e_{m,n+1} . \]
Let $\tau$ be the induced representation of $\ADD$.
It is clear by inspection that $\H = \bC e_{00}$ is coinvariant, and
hence $\rho \prec_c \tau$.
Moreover, $\tau$ is extremal because $T_i$ are isometries.
The subspace $\K$ is invariant for $T_1$ and $T_2$, and $T_i|_\K = S_i$.
Therefore $\sigma \prec_e \tau$.
So $\sigma$ is not fully extremal with respect to $\rho$.

We claim that $\tau$ is fully extremal with respect to $\rho$.
Since it is extremal, it can only fail to be fully extremal if there
is a larger space $\M \supset \L$ and commuting isometries $V_i$ on $\M$
extending $T_i$ so that $\L$ is not coinvariant, but $\bC e_{00}$ is.
Hence one of the isometries, say $V_1$, has $P_{\L^\perp}V_1P_\L \ne 0$.
Let 
\[ \N = (\ran T_1 \vee \bC e_{00})^\perp = \spn\{ e_{1,l} : l < 0 \} .\]
There must be a vector $x\in\N$ so that $V_1^* x \ne 0$.
Equivalently, there are vectors $y,z\in\L^\perp$ so that $V_1 y = z + x$.
Write $x = \sum_{l<0} a_l e_{1,l}$, and let $l_0$ be the least integer
so that $a_{-l_0} \ne0$.
Let 
\[ x'=T_2^{l_0-1}x = \sum_{l<0} a_{l+1-l_0} e_{1,l} =: \sum_{l<0} a'_l e_{1,l}\,; \]
so that $a'_{-1} \ne 0$.  Also set 
\[ y'=V_2^{l_0-1}y \qand  z'=V_2^{l_0-1}z .\]
Then 
\begin{align*}
V_1y' &= V_1V_2^{l_0-1}y = V_2^{l_0-1}V_1y \\
&= V_2^{l_0-1} (z+x) =z' + x' .
\end{align*}
Moreover, $z'$ is orthogonal to the range of $T_2^{l_0-1}$, which contains $\N$.
Hence 
\begin{align*}
 \ip{V_2V_1y',e_{1,0}} &= \ip{ V_2(z'+x'),e_{1,0}}  \\
 &= \ip{z'+x',e_{1.-1}} = a'_{-1} \ne 0 .
\end{align*}
Therefore
\[ 0 \ne \ip{V_1V_2y',e_{1,0}} = \ip{V_2y',e_{0,0}} .\]
This contradicts the fact that $\tau$ is a coextension of $\rho$.
Thus $\tau$ must be fully extremal relative to $\rho$.
\end{eg}

Now we turn to the issue of establishing that fully extremal coextensions 
(and extensions) always exist.

\begin{thm} \label{T:fully extremal}
Let $\A$ be a unital operator algebra, and let $\rho$ be a representation of $\A$ on $\H$.
Then  $\rho$ has a fully extremal  coextension $\sigma$.

If $\A$ and $\H$ are separable, then one can take $\rho$ acting 
on a separable Hilbert space.
\end{thm}

\begin{proof}
Our argument is based on Arveson's proof \cite[Theorem~2.5]{Arv_choq} 
that maximal dilations exist.  
He works with the operator system generated by $\A$, which is self-adjoint.
As we will work directly with $\A$, we need to consider adjoints as well.
The goal is to construct a coextension $\sigma$ of $\rho$ on a Hilbert space $\K$
so that for every $a\in\A$ and $k\in\K$,
\begin{align} \label{E1}
 \| \sigma(a) k \| = \sup \{ \| \tau(a) k\| : \tau \succ  \sigma, \tau \succ_c \rho \} 
\end{align}
and
\begin{align} \label{E2}
 \| \sigma(a)^* k \| = \sup \{ \| \tau(a)^* k\| : \tau \succ  \sigma, \tau \succ_c \rho \} . 
\end{align}
Once this is accomplished, it is evident that any dilation $\tau$ of $\sigma$ which is a 
coextension of $\rho$ must have $\K$ as a reducing subspace, as claimed.

To this end, choose a dense subset of $\A \times \H$, and enumerate it as
\[ \{(a_\alpha,h_\alpha) : \alpha \in \Lambda \} \]
where $\Lambda$ is an ordinal.
Suppose that we have found coextensions $\sigma_\alpha$ of $\rho$ 
for all $\alpha < \alpha_0 < \Lambda$ acting on $\K_\alpha$, 
where $\K_\beta \subset \K_\alpha$ when $\beta < \alpha$, so that 
\begin{align} \label{E1'} \tag{\ref{E1}$'$}
 \| \sigma_\alpha(a_\beta) h_\beta \| = 
 \sup \{ \| \tau(a_\beta) h_\beta\| : \tau \succ \sigma_\alpha, \tau \succ_c \rho \} 
\end{align}
and
\begin{align} \label{E2'} \tag{\ref{E2}$'$}
 \| \sigma_\alpha(a_\beta)^* h_\beta \| = 
 \sup \{ \| \tau(a)^* h_\beta \| : \tau \succ \sigma_\alpha, \tau \succ_c \rho \} . 
\end{align}
for all $\beta < \alpha$.
This latter condition is automatic because each $\tau(a)^*$ leaves $\H$ invariant,
and agrees with $\rho(a)^*$ there.  But we carry this for future use.

If $\alpha_0$ is a limit ordinal, we just form the natural direct limit of the $\sigma_\alpha$
for $\alpha < \alpha_0$, and call it $\sigma_{\alpha_0}$.  Note that it will now satisfy 
(\ref{E1'}) and (\ref{E2'}) for $\beta < \alpha_0$.

Otherwise $\alpha_0 = \beta_0+1$ is a successor ordinal.
Choose a dilation $\tau_1 \succ \sigma_{\beta_0}$ on $\M_1 \supset \K_{\beta_0}$
such that $\tau_1 \succ_c \rho$ and satisfies
\begin{align*} 
 \| \tau_1(a_{\beta_0}) h_{\beta_0} \| \ge 
 \sup \{ \| \tau(a_{\beta_0}) h_{\beta_0}\| : 
 \tau \succ \sigma_\alpha, \tau \succ_c \rho \} 
 - 2^{-1}
\end{align*}
and
\begin{align*} 
 \| \tau_1(a_{\beta_0})^* h_{\beta_0} \| \ge  
 \sup \{ \| \tau(a)^* h_{\beta_0} \| : 
 \tau \succ \sigma_\alpha, \tau \succ_c \rho \} 
 - 2^{-1}. 
\end{align*}
Then choose recusively dilations $\tau_{n+1}$ of $\tau_n$ 
on $\M_{n+1} \supset \M_n$ which are all coextensions of $\rho$ so that
\begin{align*}
 \| \tau_{n+1}(a_{\beta_0}) h_{\beta_0} \| \ge 
 \sup \{ \| \tau(a_{\beta_0}) h_{\beta_0}\| : 
 \tau \succ  \sigma_\alpha, \tau \succ_c \rho \} 
 - 2^{-n-1}
\end{align*}
and
\begin{align*}
 \| \tau_{n+1}(a_{\beta_0})^* h_{\beta_0} \| \ge  
 \sup \{ \| \tau(a)^* h_{\beta_0} \| : 
 \tau \succ  \sigma_\alpha, \tau \succ_c \rho \} 
 - 2^{-n-1}. 
\end{align*}
The inductive limit is a representation $\sigma_{\alpha_0}$ with the desired properties.

Once we reach $\Lambda$, we have constructed a representation $\tilde\sigma_1$
on $\tilde\K_1$ coextending $\rho$ and satisfying (\ref{E1}) and (\ref{E2}) for
vectors $h \in \H$.
Now repeat this starting with $\tilde\sigma_1$ and a dense subset of 
$\A \times \tilde\K_1$, but still considering dilations which are coextensions of $\rho$.
This time, the equations involving the adjoint are important.
The result is a representation $\tilde\sigma_2$ on $\tilde\K_2$ 
dilating $\tilde\sigma_1$ and coextending $\rho$ 
satisfying (\ref{E1}) and (\ref{E2}) for all vectors $k \in \tilde\K_1$.
Repeat recursively for all $n\ge1$ and in the end, we obtain the desired coextension.

If $\A$ and $\K$ are separable, a countable sequence of points suffices,
and at each stage of this countable process, one obtains separable spaces.
So the result is a separable representation.
\end{proof}

\begin{rem} \label{R:fully extremal}
It easily follows from the proof of existence of fully extremal coextensions 
that if $\sigma$ is a coextension of $\rho$, then there is a dilation $\tau$
of $\sigma$ which is a fully extremal coextension of $\rho$.
\end{rem}

\begin{rem} \label{R:exist extremal}
A proof of existence of extremal coextensions can be made along the same lines.
It is only necessary to achieve $\sigma \succ_c \rho$ on $\K$ such that:
\[ \| \sigma(a) k \| = \sup \{ \| \tau(a) k \| : \tau \succ_c \sigma \} .\]
One can always achieve this by repeated coextension, and in this way
one obtains an extremal coextension $\sigma$ of $\rho$ with the additional property
that $\H$ is cyclic, i.e.\ $\K = \ol{\sigma(\A) \H}$.
This is evidently not the case in general for extremal coextensions, and
in particular, for fully extremal coextensions.  See the preceding examples and
Remark~\ref{R:uniqueness}.
\end{rem}

The same result for extensions follows by duality.

\begin{cor}\label{C:fully extremal extensions}
Let $\A$ be a unital operator algebra, and let $\rho$ be a representation of $\A$ on $\H$.
Then  $\rho$ has a fully extremal  extension $\sigma$.
\end{cor}

\begin{cor}\label{C:extremal sums}
If $\rho_1$ and $\rho_2$ are representations of $\A$, then $\sigma_1$ and $\sigma_2$
are fully extremal coextensions of $\rho_1$ and $\rho_2$, respectively, if and only if
$\sigma_1 \oplus \sigma_2$ is a fully extremal coextension of $\rho_1 \oplus \rho_2$.

In particular, $\sigma_1$ and $\sigma_2$ are extremal coextensions of $\A$ 
if and only if $\sigma_1 \oplus \sigma_2$ is an extremal coextension.
\end{cor}

\begin{proof}
First suppose that $\sigma_1$ and $\sigma_2$, acting on $\K_1$ and $\K_2$,
are fully extremal coextensions of $\rho_1$ and $\rho_2$, respectively.
Suppose that $\tau$ is a representation on 
\[ \P = \K_1 \oplus \K_2 \oplus \P' \]
such that
\[ \tau \succ \sigma_1 \oplus \sigma_2 \qand  \tau \succ_c \rho_1 \oplus \rho_2 .\]
Then as $\tau \succ \sigma_i$ and $\tau \succ_c \rho_i$, we deduce that $\tau$
reduces $\K_i$ and hence reduces $\K_1 \oplus \K_2$.
So $\sigma_1 \oplus \sigma_2$ is a fully extremal coextension of $\rho_1 \oplus \rho_2$.

Conversely, if $\sigma_1 \oplus \sigma_2$ is a fully extremal coextension 
of $\rho_1 \oplus \rho_2$, suppose that $\tau$ is a representation on $\P=\K_1 \oplus \P'$
satisfies $\tau \succ \sigma_1$ and $\tau \succ_c \rho_1$.
Then 
\[
 \tau \oplus \sigma_2 \succ \sigma_1  \oplus \sigma_2 \qand
 \tau \oplus \sigma_2 \succ_c \rho_1 \oplus \rho_2 .  
\]
It follows that $\tau \oplus \sigma_2$ reduces $\K_1 \oplus \K_2$.
So $\tau$ reduces $\K_1$.  Whence $\sigma_1$ is fully extremal.

Applying this to $\rho_i=\sigma_i$ yields the last statement.
\end{proof}

If one starts with a representation $\rho$ and alternately forms 
extremal extensions and coextensions, it may require a countable 
sequence of alternating extensions and coextensions in
order to obtain a maximal dilation as in Example~\ref{Eg:zigzag}.
One advantage of fully extremal extensions and coextensions is that only one 
is required to obtain a maximal dilation.  

\begin{prop} \label{P:coextend and fully extend}
Let $\rho$ be a representation of $\A$.
If $\sigma$ is an extremal coextension of $\rho$,
and $\tau$ is a fully extremal extension of $\sigma$,
then $\tau$ is a maximal dilation.
\end{prop}

\begin{proof}
Since $\tau$ is an extremal extension, it suffices to show that it
is also an extremal coextension.
For then the Muhly-Solel result, Theorem~\ref{T:MS}, will show that $\tau$ is a maximal dilation.

Say that $\rho$, $\sigma$ and $\tau$ act on $\H$, $\K$ and $\L$ respectively.
Suppose that $\pi$ is a coextension of $\tau$ acting on $\P$.
Decompose 
\[ \P = (\L\ominus\K) \oplus \H \oplus (\K \ominus \H) \oplus (\P \ominus \L) . \]
Then we have
\[
 \pi = \begin{bmatrix}
         \tau_{11} & 0 & 0 & 0 \\
         \tau_{21} & \rho & 0 & 0 \\
         \tau_{31} & \sigma_{32} & \sigma_{33} & 0 \\
         \pi_{41} & \pi_{42} & \pi_{43} & \pi_{44}
         \end{bmatrix} ,
\]
where $\sigma$ is represented by the middle $2 \times 2$ square,
and $\tau$ is represented by the upper left $3 \times 3$ corner.
The lower right  $3 \times 3$ corner is a coextension of $\sigma$.
Since $\sigma$ is an extremal coextension, we obtain 
\[ \pi_{42} = 0 = \pi_{43} .\]
Thus we can rearrange the decomposition moving $\P \ominus \L$ to the first
coordinate to obtain
\[
 \pi \simeq \begin{bmatrix}
                 \pi_{44} & \pi_{41} & 0 & 0 \\
                 0 & \tau_{11} & 0 & 0 \\
                 0 & \tau_{21} &  \rho & 0 \\
                 0 & \tau_{31} & \sigma_{32} & \sigma_{33} 
         \end{bmatrix} .
\]
This is a coextension of $\tau$ which is an extension of $\sigma$.
By the fact that $\tau$ is a fully extremal extension of $\sigma$,
we deduce that $\pi_{41}=0$ and so
\[ \pi \simeq \pi_{44}\oplus \tau .\]
Therefore $\tau$ is also an extremal coextension.
\end{proof}

The dual result is obtained the same way.

\begin{cor} \label{C:extend and fully coextend}
Let $\rho$ be a representation of $\A$.
If $\sigma$ is an extremal extension of $\rho$,
and $\tau$ is a fully extremal coextension of $\sigma$,
then $\tau$ is a maximal dilation.
\end{cor}

The classical notion of minimal coextension is that the space is cyclic for $\A$.
However, it seems more natural that the original space merely generate the whole
space as a reducing subspace.  This is because fully extremal coextensions
do not generally live on the cyclic subspace generated by the original space. 

\begin{defn}
An extremal coextension $\sigma$ on $\K$ of a representation $\rho$  of $\A$ on $\H$ 
is \textit{minimal} if the only reducing subspace of $\K$ containing $\H$ is $\K$ itself.
Likewise we define minimality for fully extremal coextensions, extremal extensions
and fully extremal extensions.
This minimal (fully) extremal (co)extension is \textit{unique} if any two of these objects
are unitarily equivalent via a unitary which is the identity on $\H$.

Say that a coextension $\sigma$ on $\K$ of a representation $\rho$ of $\A$ on $\H$ 
is \textit{cyclic} if $\K = \ol{\sigma(\A)\H}$.
\end{defn}

\begin{rem} \label{R:uniqueness}
These notions of minimality are subtle.  
Look at Example~\ref{Eg:zigzag}.
In general, to generate the space on which a coextension acts, one must
alternately take the cyclic subspace generated by $\sigma(\A)$ and $\sigma(\A)^*$,
perhaps infinitely often, in order to obtain the reducing subspace generated by $\H$.

In Example~\ref{Eg:bidisk}, the 
1-dimensional zero representation $\rho$ has an extremal coextension $\sigma$.  
It is minimal because 
\[ \K = \ol{\sigma(\ADD) \H} , \]
i.e.\ $\H$ is cyclic for $\sigma(\ADD)$.  
However the extension $\tau$ of $\sigma$ is also an extremal coextension of $\rho$.  
While it is no longer true that $\ol{\tau(\ADD) \H}$ is the whole space, it is nevertheless
the smallest reducing subspace containing $\H$, and so it is also minimal.
Thus it is a minimal fully extremal coextension.
Moreover, $\sigma = \sigma_0 \oplus \sigma_1$ where $\sigma_1$ is a maximal dilation.
Conversely, every representation of this form is fully extremal.
\end{rem}

It is important to note that there are minimal fully extremal coextensions
obtained in the natural way.

\begin{prop} \label{P:minimal fully extremal}
Let $\rho$ be a representation of $\A$ on $\H$.
Let $\sigma$ be a fully extremal $($co$)$extension of $\rho$ on $\K$.
Let $\sigma_0$ be the restriction of $\sigma$ to the smallest
reducing subspace $\K_0$ for $\sigma(\A)$ containing $\H$.
Then $\sigma_0$ is fully extremal. 
Moreover, $\sigma = \sigma_0 \oplus \sigma_1$ where $\sigma_1$
is a maximal representation. Conversely, every $($co$)$extension of this
form is fully extremal.
\end{prop}

\begin{proof}
The proof is straightforward.
Since $\K_0$ reduces $\sigma$, we can write 
\[ \sigma = \sigma_0 \oplus \sigma_1 \]
acting on $\K = \K_0 \oplus \K_0^\perp$.
Suppose that $\tau$ is a dilation of $\sigma_0$ which is a coextension of $\rho$.
Then $\tau \oplus \sigma_1$ is a dilation of $\sigma$ which is a coextension of $\rho$.
Since $\sigma$ is fully extremal, we have a splitting
\[
 \tau  \oplus \sigma_1 \simeq \sigma \oplus \tau_1 
 =  \sigma_0 \oplus \sigma_1 \oplus \tau_1.
\]
Hence 
\[ \tau  = \sigma_0 \oplus \tau_1 .\]
It follows that $\sigma_0$ is a fully extremal coextension of $\rho$.

Any dilation of $\sigma_1$ yields a dilation of $\sigma$ which is a coextension of $\rho$.
As $\sigma$ is fully extremal, this must be by the addition of a direct sum.
Hence $\sigma_1$ is a maximal representation.
Conversely, if $\sigma_0$ is a (minimal) fully extremal coextension of $\rho$
and $\sigma_1$ is a maximal representation, then $\sigma = \sigma_0 \oplus \sigma_1$
is a fully extremal coextension because any dilation of $\sigma$ 
is a dilation of $\sigma_0$ direct summed with $\sigma_1$.

The same argument works for extensions.
\end{proof}

We refine Proposition~\ref{P:coextend and fully extend}.  In light of
Remark~\ref{R:exist extremal}, we know that the coextensions asked
for in the following proposition always exist.

\begin{prop} \label{P:minimal maximal dilation}
Let $\rho$ be a representation of $\A$ on $\H$.
Let $\sigma$ be a cyclic extremal coextension of $\rho$ on $\K$.
Let $\pi$ be a minimal fully extremal extension of $\sigma$.
Then $\pi$ is a minimal maximal dilation of $\rho$.
\end{prop}

\begin{proof}
It suffices to show that the whole space, $\L$, is the smallest reducing subspace 
for $\pi(\A)$ containing $\H$.
In particular, it contains 
\[ \ol{\pi(\A) \H} = \ol{ \sigma(\A) \H} = \K .\]
But the minimality of $\pi$ as a fully extremal extension of $\sigma$
ensures that there is no proper reducing subspace containing $\K$.
So $\pi$ is minimal as a maximal dilation.
\end{proof}

We require a result which is more subtle than Proposition~\ref{P:minimal maximal dilation}
but is valid for fully extremal coextensions.

\begin{thm} \label{T:fullextoffullcoext}
Let $\rho$ be a representation of $\A$ on $\H$.
Let $\sigma$ be a minimal fully extremal coextension of $\rho$ on $\K$.
Let $\pi$ be a minimal fully extremal extension of $\sigma$.
Then $\pi$ is a minimal maximal dilation of $\rho$.

Moreover, the representation $\pi$ determines $\sigma$, and thus
two inequivalent minimal fully extremal coextensions of $\rho$ yield
inequivalent minimal maximal dilations of $\rho$.
\end{thm}

\begin{proof}
Let $\pi$ act on the Hilbert space $\L$.
Note that $\pi$ is a maximal dilation of $\rho$ by 
Proposition~\ref{P:coextend and fully extend}.
Let 
\[ \M = \ol{\pi(\A)^*\H} \ominus \H .\]
Then $\M^\perp$ is the largest invariant subspace for $\pi(\A)$ in which $\H$ is coinvariant.
Let $\tau$ denote the restriction of $\pi$ to $\M^\perp$.
Since $\M^\perp$ contains $\K$, we have 
\[ \tau \succ \sigma \qand  \tau \succ_c \rho .\]
Hence by the fully extremal property of $\sigma$, we deduce that
\[
 \tau = \sigma \oplus \tau' 
 \quad\text{on}\quad
 \M^\perp = \K \oplus (\M + \K)^\perp .
\]

Now the smallest reducing subspace for $\pi(\A)$ containing $\H$ clearly contains $\M$.
Thus it contains the smallest $\tau(\A)$ reducing subspace of $\M^\perp$ contaning $\H$.
But since $\tau = \sigma \oplus \tau'$ and $\sigma$ is minimal as a fully extremal coextension,
the smallest $\tau(\A)$ reducing subspace containing $\H$ is $\K$.
Then since $\pi$ is a minimal fully extremal extension of $\sigma$,
we see that $\L$ is the smallest reducing subspace containing $\K$.
So $\pi$ is minimal.

{}From the arguments above, we see that $\sigma$ is recovered from $\pi$
by forming 
\[ \M = \ol{\pi(\A)^*\H} \ominus \H .\]
restricting $\pi$ to $\M^\perp$ to get $\tau$,
and taking the smallest $\tau$ reducing subspace of $\M^\perp$ containing $\H$.
The restriction to this subspace is $\sigma$.
Hence $\pi$ determines $\sigma$.
Consequently, two inequivalent fully extremal coextensions of $\rho$ yield
inequivalent minimal maximal dilations of $\rho$.
\end{proof}

The following is immediate by duality.

\begin{cor} \label{C:fullcoextoffullext}
Let $\rho$ be a representation of $\A$ on $\H$.
Let $\sigma$ be a minimal fully extremal extension of $\rho$ on $\K$.
Let $\pi$ be a minimal fully extremal coextension of $\sigma$.
Then $\pi$ is a minimal maximal dilation of $\rho$.

Moreover, the representation $\pi$ determines $\sigma$, and thus
two inequivalent minimal fully extremal extensions of $\rho$ yield
inequivalent minimal maximal dilations of $\rho$.
\end{cor}

\section{Semi-Dirichlet algebras} \label{S:sD}

In this section, we consider a class of algebras where the theory is more like the classical one.
The semi-Dirichlet property is a powerful property that occurs often in practice.
{}From the point of view of dilation theory, these algebras are very nice.

\begin{defn}
Say that an operator algebra $\A$ is \textit{semi-Dirichlet} if 
\[ \A^*\A \subset \ol{\A+\A^*}\]
when $\A$ is considered as a subspace of its C*-envelope.

A unital operator algebra (not necessarily commutative) is called 
\textit{Dirichlet} if $\A+\A^*$ is norm dense in $\cenv(\A)$.
\end{defn}

Notice that since $\A$ is unital, we always have 
\[ \A + \A^* \subset \spn(\A^*\A) ,\]
so semi-Dirichlet means that 
\[ \ol{\spn}(\A^*\A) = \ol{\A+\A^*} .\]

The interested reader can note that in the case of w*-closed algebras, the proofs below
can be modified to handle the natural w*-closed condition which we call 
\textit{semi-$\sigma$-Dirichlet} if 
\[ \A^*\A \subset \ol{\A+\A^*}^{w*} .\]
Free semigroup algebras and free semigroupoid algebras of graphs and nest algebras
all are semi-$\sigma$-Dirichlet.

The following simple proposition establishes a few elementary observations.

\begin{prop} \label{P:sDprops} \strut\\[-3ex]
\begin{enumerate}
\item $\A$ is Dirichlet if and only if  $\A$ and $\A^*$ are semi-Dirichlet.

\item If $\sigma$ is a completely isometric representation of $\A$ on $\H$,
and $\sigma(\A)^*\sigma(\A) \subset \ol{\sigma(\A) + \sigma(\A)^*}$, then
$\A$ is semi-Dirichlet.

\item If $\sigma$ is a Shilov representation of a semi-Dir\-ich\-let algebra $\A$, 
then $\sigma(\A)$ is semi-Dir\-ich\-let.
\end{enumerate}
\end{prop}

\begin{proof}
(i) It is obvious that if $\A$ is Dirichlet, then both $\A$ and $\A^*$ are semi-Dirichlet.
For the converse, notice that if $\A$ is semi-Dirichlet, then an easy
calculation shows that $\ol{\spn}(\A\A^*)$ is a C*-algebra \cite{Arv3}.
Since $\A$ generates $\cenv(\A)$, this is the C*-algebra $\ol{\spn}(\A\A^*)$.
Thus the semi-Dirichlet property for $\A^*$ now shows that 
$\A+\A^*$ is norm dense in $\cenv(\A)$.

(ii) If $\sigma$ is completely isometric, then $\fA=\ca(\sigma(\A))$ is a C*-cover of $\A$.
By the minimal property of the C*-envelope, there is a quotient map $q:\fA \to\cenv(\A)$
so that $q\sigma|_\A$ is the identity map.
If $\sigma(\A)^*\sigma(\A)$ is contained in $ \ol{\sigma(\A) + \sigma(\A)^*}$, then passing
to the quotient yields the semi-Dirichlet property.

(iii) If $\sigma$ is Shilov, then there is a $*$-representation $\pi$ of $\cenv(\A)$ on $\K$
and an invariant subspace $\H$ so that $\sigma(a) = \pi(a)|_\H$.  
The map 
\[ \tilde\sigma(x) = P_\H \pi(x)|_\H \qfor  x \in \cenv(\A) \]
is a completely positive map extending $\sigma$.  
In particular, 
\[ \tilde\sigma(a^*) = \sigma(a)^* \qfor a \in \A .\]
For $a,b\in\A$, we calculate
\begin{align*}
 \pi(a^*b) &= \begin{bmatrix} * & * \\ * & \tilde\sigma(a^*b) \end{bmatrix} 
 = \pi(a)^* \pi(b) \\ &= 
 \begin{bmatrix} * & * \\ 0 & \sigma(a)^*\end{bmatrix}
 \begin{bmatrix} * & 0 \\ * & \sigma(b)\end{bmatrix}  
 = \begin{bmatrix} * & * \\ * & \sigma(a)^*\sigma(b) \end{bmatrix}
\end{align*}
Hence 
\[ \tilde\sigma(a^*b) = \sigma(a)^*\sigma(b) \qforal  a,b\in\A .\]

Since $\A$ is semi-Dirichlet, we can write 
\[ a^*b = \lim_n c_n^* + d_n \]
where $c_n,d_n\in\A$.  Thus,
\[
 \sigma(a)^*\sigma(b) = \tilde\sigma(a^*b) = \lim \sigma(c_n)^* + \sigma(d_n).
\]
That is, 
\[ \sigma(\A)^*\sigma(\A) \subset \ol{\sigma(\A) + \sigma(\A)^*} .\]
It now follows from (ii) that $\sigma(\A)$ is semi-Dirichlet.
\end{proof}

\begin{eg} \label{E:dirichlet}
Observe that if $\A$ is a function algebra with Shilov boundary $X = \partial \A$,
then $\spn(\A^*\A)$ is a norm closed self-adjoint algebra which separates points.
So by the Stone-Weierstrass Theorem, it is all of $\rC(X)$.
So the semi-Dirichlet property is just the Dirichlet property for function algebras.
\end{eg}

\begin{eg} \label{E:tensor_sD}
The non-commutative disk algebras $\fA_n$ are semi-Dir\-ich\-let.  
This is immediate from the relations $\fs_j^* \fs_i = \delta_{ij}I$.

Indeed, it is easy to see that all tensor algebras of directed graphs
and tensor algebras of C*-correspondences are semi-Dirichlet.
For those familiar with the terminology for the tensor algebra of a 
C*-correspondence $E$ over a C*-algebra $\fA$,
the algebra $\T^+(E)$ is generated by 
\[ \sigma(\fA) \qand   \{T(\xi) : \xi \in E\} ,\]
where $\sigma$ and $T$ are the canonical representations 
of $\fA$ and $E$, respectively, on the Fock space of $E$.
The relation 
\[ T(\xi)^*T(\eta) = \sigma(\ip{\xi,\eta}) \]
yields the same  kind of cancellation as for the non-commutative 
disk algebra to show that 
\[ \T^+(E)^*\T^+(E) \subset \ol{\T^+(E) + \T^+(E)^*} .\]
\end{eg}

\begin{eg}
There is no converse to Proposition~\ref{P:sDprops}(ii).
Consider the disk algebra $\AD$.
The Toeplitz representation on $H^2$ given by $\sigma(f) = T_f$,
the Toeplitz operator with symbol $f$, is completely isometric.
This is Shilov, and so has the semi-Dirichlet property.
This is also readily seen from the identity 
\[ T_f^*T_g = T_{\bar f g} \qforal  f, g \in \AD .\]

However the representation 
\[ \rho(f) = T_{f(\bar z)} \]
generated by $\rho(z)=T_z^*$ is also completely isometric.  However 
\[ \rho(z)^* \rho(z) = T_z T_z^* = I - e_0e_0^* .\]
This is not a Toeplitz operator, and so is a positive distance from
\begin{align*}
 \ol{\rho(\AD) + \rho(\AD)^*} &= \ol{ \{T_{\bar f + g} : f,g \in \AD \}} \\
 &= \{ T_f : f \in \CT \} .
\end{align*}
\end{eg}

We will establish the following theorem.

\begin{thm} \label{T:sD}
Suppose that $\A$ is a semi-Dirichlet unital operator algebra.
Let $\rho$ be a representation of $\A$.
Then $\rho$ has a unique minimal extremal coextension $\sigma$, 
it is fully extremal and cyclic (i.e.\ $\K = \ol{\sigma(\A)\H}$).
Moreover, every Shilov representation is an extremal coextension.
\end{thm}

We begin with a couple of lemmas.

\begin{lem} \label{L1:sD}
Suppose that $\A$ is a  semi-Dirichlet unital operator algebra.
Let $\rho$ be a representation of $\A$, and let $\sigma$ be a cyclic
extremal coextension of $\rho$ on $\K$.
Then $\sigma$ is fully extremal.
\end{lem}

\begin{proof}
Suppose that that $\tau$ is an extremal coextension of $\rho$ which is a dilation of $\sigma$.
Say $\tau$ acts on $\L \supset \K$.
Let $\pi$ be a fully extremal extension of $\tau$.
Then $\pi$ is a maximal dilation of $\rho$ by Proposition~\ref{P:coextend and fully extend}.
Moreover $\L$ is invariant for $\pi(\A)$, as is $\K$; so that 
\[ \tau(a) = \pi(a)|_\L \qand \sigma(a) = \pi(a)|_\K= \tau(a)|_\K .\] 
Also $\H$ is semi-invariant for $\pi(\A)$ and coinvariant for $\tau(\A)$ in $\L$.

If $\sigma$ is not a direct summand of $\tau$, then $\K$ is not coinvariant for $\tau(\A)$.
Thus there is a vector $x \in \L \ominus \K$ and $a \in \A$ so that
\[ P_\K \tau(a) x \ne 0 .\]
This vector in $\K$ can be approximated by a vector 
$\sigma(b) h$ for some $b\in\A$ and $h \in \H$ sufficiently well so that
\[ \ip{\tau(a)x, \sigma(b) h} \ne 0 .\]
Now $a^*b \in \A^*\A$ can be written as 
\[ a^*b = \lim_n c_n + d_n^* \quad\text{where } c_n,d_n \in \A .\]
Therefore
\begin{align*}
 0 &\ne \ip{\tau(a)x, \sigma(b) h}  = \ip{\pi(a)x, \pi(b) h} \\
 &= \ip{x, \pi(a^*b) h} 
  = \lim_{n\to\infty} \ip{x, \pi(c_n)h + \pi(d_n)^* h} \\
 &= \lim_{n\to\infty} \ip{x, \pi(d_n)^* h} 
  = \lim_{n\to\infty} \ip{x, \tau(d_n)^* h}
\end{align*}
Here we used the fact that 
\[ \pi(c_n)h = \sigma(c_n)h \in \K , \]
which is orthogonal to $x$, and then the fact that the compression of 
$\pi(d_n)^*$ to $\L$ is $\tau(d_n)^*$.
This calculation shows that $\H$ is not coinvariant for $\tau$, contrary to our hypothesis.
This means that $\tau$ does indeed have $\sigma$ as a direct summand.
So $\sigma$ is fully extremal. 
\end{proof}

\begin{lem} \label{L2:sD}
Suppose that $\A$ is a  semi-Dirichlet unital operator algebra.
Let $\rho$ be a representation of $\A$.
Then any two cyclic Shilov coextensions $\sigma_i$ of $\rho$, $i=1,2$, on $\K_i$ 
are equivalent. Hence a cyclic Shilov extension of $\rho$ is fully extremal.
\end{lem}

\begin{proof}
Let $\sigma_i$, $i=1,2$, be two minimal cyclic Shilov coextensions of $\rho$
on $\K_i$; so that $\K_i = \ol{\sigma_i(\A)\H}$.
Let $\pi_i$ be the maximal dilations of $\rho$ on $\L_i\supset \K_i$
such that $\K_i$ is invariant and $\pi_i(a)|_{\K_i} = \sigma_i(a)$ for $a \in \A$.
The idea is to follow the standard proof by showing that there is a map
$U\in\B(\K_1,\K_2)$ given by 
\[ U\sigma_1(a)h = \sigma_2(a)h \]
which is  a well defined isometry of $\K_1$ onto $\K_2$.
To this end, it suffices to verify that
\[ \ip{\sigma_1(a_1)h_1, \sigma_1(a_2)h_2} = \ip{\sigma_2(a_1)h_1, \sigma_2(a_2)h_2} \]
for all $a_1,a_2 \in \A$ and $h_1,h_2 \in \H$.

By hypothesis, we can find $b_n,c_n\in\A$ so that
\[ a_2^*a_1 = \lim_n b_n + c_n^* .\]  
We calculate
\begin{align*}
 \ip{\sigma_i(a_1)h_1, \sigma_i(a_2)h_2} &=
 \ip{\pi_i(a_1)h_1, \pi_i(a_2)h_2} \\&=
 \ip{\pi_i(a_2^*a_1)h_1, h_2} \\&=
 \lim_{n\to\infty} \ip{\big( \pi_i(b_n) + \pi_i(c_n)^* \big)h_1, h_2} \\&=
 \lim_{n\to\infty} \ip{\big( \rho(b_n) + \rho(c_n)^* \big)h_1, h_2} .
\end{align*}
This quantity is independent of the dilation, and thus $U$ is a well-defined isometry.

Since $\K_i = \ol{\sigma_i(\A)\H}$, it follows that $U$ is unitary.
It is also evident that $U|_\H$ is the identity map.
So $U$ is the desired unitary equivalence of $\sigma_1$ and $\sigma_2$.

Since $\A$ always has a cyclic extremal coextension $\sigma$, 
it must be the unique cyclic Shilov extension.  
By Lemma~\ref{L1:sD}, $\sigma$ is fully extremal.
\end{proof}

\begin{proof}[\textbf{\em Proof of Theorem~\ref{T:sD}.}]
Let $\tau$ be any minimal extremal coextension of $\rho$ on $\L \supset \H$.
Set 
\[ \K = \ol{\tau(\A)\H} \qand  \sigma = \tau|_\K .\]
Also let $\pi$ be a fully extremal extension of $\tau$.
By Proposition~\ref{P:coextend and fully extend}, $\pi$ is a maximal dilation of $\rho$.
Since $\L$ is invariant for $\pi(\A)$ and $\K$ is invariant for $\tau(\A)$, 
it follows that $\K$ is invariant for $\pi(\A)$.
Hence $\sigma$ is Shilov.
By Lemma~\ref{L2:sD}, $\sigma$ is fully extremal.
It follows that $\tau = \sigma \oplus \tau'$.
However $\tau$ is minimal.  So 
\[ \tau = \sigma \qand  \L = \K = \ol{\tau(\A)\H} .  \]
Hence $\tau$ is cyclic. By Lemma~\ref{L2:sD}, $\tau$ is unique.

Now let $\sigma$ be a Shilov representation of $\A$.
Let $\tau$ be a cyclic extremal coextension of $\sigma$.
By Lemma~\ref{L2:sD}, $\sigma$ and $\tau$ are equivalent coextensions of $\sigma$.
Therefore $\tau = \sigma$. Thus $\sigma$ is extremal.
\end{proof}

The consequences for Dirichlet algebras are apparent.

\begin{cor} \label{C:Dirichlet}
If $\A$ is a Dirichlet operator algebra, then every Shilov extension 
and every Shilov coextension is fully extremal; and the minimal extremal 
$($co-$)$extension of a representation is unique.
Moreover the minimal maximal dilation of a representation is unique.
\end{cor}

\begin{proof}
The first statement is immediate from Theorem~\ref{T:sD} and its dual.
Let $\pi$ be a minimal maximal dilation of a representation $\rho$ on a Hilbert space $\L$.
Let $\K = \pi(\A)\H$.  This is the minimal Shilov subspace containing $\H$.
Thus by Theorem~\ref{T:sD}, it is the unique minimal fully extremal coextension of $\rho$.
Let $\L_0 = \pi(\A)^*\K$.  This is the minimal Shilov extension of $\sigma$.
Hence by the dual of Theorem~\ref{T:sD}, this coincides with the 
unique minimal fully extremal extension of $\sigma$.  
By Corollary~\ref{P:coextend and fully extend}, the restriction of $\pi$ to $\L_0$
is a maximal dilation.  Since $\pi$ is minimal, $\L_0=\L$.
So $\pi$ is obtained by taking the unique minimal extremal coextension of $\rho$
to get $\sigma$, followed by the unique minimal extremal extension of $\sigma$.
So $\pi$ is uniquely determined.
\end{proof}

While semi-Dirichlet algebras behave exceptionally well for coextensions,
they are not nearly so well behaved for extensions.

\begin{eg}
We consider extensions for representations of the non-commutative disk algebra
$\fA_n$. Denote the generators by $\fs_1,\dots,\fs_n$, 
and write $\fs = \big[ \fs_1\ \dots \ \fs_n \big]$.
A representation $\rho$ on $\H$ is determined by a row contraction
\[
 A = \rho( \fs ) = \big[ \rho(\fs_1)\ \dots \ \rho(\fs_n) \big] =: \big[ A_1\ \dots \ A_n \big] ,
\]
where $\|A\| = \big\| \sum_i A_iA_i^* \big\|^{1/2} \le 1$.
We have seen that $A$ has a unique minimal coextension to a row isometry,
and this is the unique minimal fully extremal coextension.

Now consider an extension $\sigma$ of $\rho$ acting on $\K$.
This correspond to simultaneous extensions of $A_i$ to 
\[
 B_i = \sigma(\fs_i) =  \begin{bmatrix} A_i & B_{i,12}\\0 & B_{i,22} \end{bmatrix}
\]
such that $B = \big[ B_1\ \dots \ B_n \big]$ is a row contraction.
It is straightforward to verify that it is extremal if and only if
$B$ is a coisometry.
We claim that: \textit{an extension $\sigma$ of $\rho$ is fully extremal if and only if 
$B$ is a coisometry such that} 
\[ \ran B^* \vee \H^{(n)} = \K^{(n)} . \] 

Indeed, if this condition holds, then there is no proper extension of $B$;
so consider any row contractive coextension $C$ of $B$ which is an
extension of $A_i$.
Then $C_i^*$ are extensions of $B_i^*$ which are coextensions of $A_i^*$.
So 
\[
 C^* = \begin{bmatrix} B^* & X\\0&Y \end{bmatrix} 
 = \begin{bmatrix} A^* & 0 & 0 \\B_{12}^* & B_{22}^* & X_2\\ 0 & 0 & Y \end{bmatrix} .
\]
Since $B^*$ is an isometry, we require that $\ran X$ be orthogonal to
$\ran B^*$. And since $C$ is an extension of $A$, we have $\ran X$ is 
orthogonal to $\H^{(n)}$. Therefore by hypothesis, $\ran X$ is orthogonal to $\K^{(n)}$,
and thus $X=0$. Therefore $C = B \oplus Y$ is a direct sum.

Conversely, suppose that there is a unit vector $x = (x_1,\dots,x_n)^t$ in $\K^{(n)}$
which is orthogonal to $\ran B^* \vee \H^{(n)}$.
Define an extension of $B_i^*$ to $\K \oplus \bC$ by
\[ C^*_i = \begin{bmatrix} B_i^* & x_i\\0&0 \end{bmatrix} .\]
Since $x_i \in \H^\perp$, this is a coextension of $A_i^*$.
So $C$ determines an extension of $A$ which is a coextension of $B$.
Clearly it does not split as a direct sum.
Finally, $C$ is a coisometry because 
$C^* = \begin{sbmatrix} B^* & x\\0&0 \end{sbmatrix}$ is an isometry.
In particular, $C$ is a row contraction.
\smallbreak 

Next we observe that the minimal fully extremal extensions 
are far from unique in general by showing how to construct
a fully extremal coextension.  

Start with $A$ which is not coisometric.
Then 
\[ D = \big( I - \sum A_iA_i^* \big)^{1/2} \ne 0 .\]
Consider a fully extremal extension $B$ as above.
Then $B$ is a coisometry on $\K = \H \oplus \K_0$;
whence $\big[ A\  B_{12} \big]$ is a coisometry in $\B(\K,\H)$.
Therefore 
\[ I_\H = \big[ A\  B_{12} \big] \ \big[ A\  B_{12} \big]^* = AA^* + B_{12}B_{12}^* .\]
Hence 
\[ B_{12}B_{12}^* = D^2 , \]
and therefore $B_{12} = DX$
where $X =  \big[ X_1\ \dots \ X_n \big]$ is a coisometry in $\B(\K_0^{(n)},\H)$.
Let $\R = \ran X^*$.
Then to be fully extremal, we  have that $B_{22}^*$ is an isometry
from $\K_0$ onto $\R^\perp \subset \K_0^{(n)}$.
Now let $V$ be any isometry in $\B(\K_0^{(n)})$ with $\ran V = \R^\perp$.
Then $V^*B_{22}^*$ is a unitary in $\B(\K_0, \K_0^{(n)})$.
Decompose the unitary 
\[ S := B_{22}V = \big[ S_1\ \dots \ S_n \big] \] 
where $S_i \in \B(\K_0)$.
Observe that $S_i$ are isometries such that 
\[ \sum_i S_iS_i^* = I ;\]
in other words they are Cuntz isometries.
Since 
\[ B_{22} =  B_{22}VV^* = SV^* \]
in $\B(\K_0^{(n)}, \K_0)$, we decompose 
this as 
\[ B_{22} = SV^* =  \big[ T_1\ \dots \ T_n \big] .\]
We obtain 
\[ 
 B_i =  \begin{bmatrix} A_i & B_{i,12}\\0 & B_{i,22} \end{bmatrix}
 =  \begin{bmatrix} A_i & DX_i\\0 & T_i \end{bmatrix} .
\]

Conversely, if we choose any coisometry $X$ in $\B(\K_0^{(n)},\H)$,
we may define $\R = \ran X^*$, choose an isometry 
$V$ in $\B(\K_0^{(n)})$ with $\ran V = \R^\perp$,
and a set of Cuntz isometries $S_i$ in $\B(\K_0)$, then the formulae
above yield a fully extremal extension.
This may not be minimal in general, but the restriction to the smallest
reducing subspace containing $\H$ is a minimal fully extremal extension.
This restriction will not change $X$.
So if two minimal fully extremal extensions are equivalent, then at
the very least, there is a unitary $U \in \B(\K_0)$ so that $XU = X'$.
It is easy to see that there are many inequivalent choices for $X$
even if $D$ is rank one.
\end{eg}

\section{Commutant Lifting} \label{S:CLT}

Many variants of the commutant lifting theorem have been established
for a wide range of operator algebras.  They differ somewhat in the
precise assumptions and conclusions.  The general formulation in
Douglas-Paulsen \cite{DougPaul} and Muhly-Solel \cite{MS_hilbert}
uses Shilov modules.  But we will formulate it using only 
fully extremal coextensions.  The second definition is motivated by the
lifting results of Paulsen-Power \cite{PP_tensor}.

\begin{defn}
An operator algebra $\A$ has the \textit{strong commutant lifting property} (SCLT)
if whenever $\rho$ is a completely contractive representation of $\A$  on $\H$ with a
fully extremal coextension $\sigma$ on $\K \supset \H$, 
and $X$ commutes with $\rho(\A)$, 
then $X$ has a coextension $Y$ in $\B(\K)$ with $\|Y\|=\|X\|$ 
which commutes with $\sigma(\A)$.

An operator algebra $\A$ has the \textit{commutant lifting property} (CLT)
if whenever $\rho$ is a completely contractive representation of $\A$  on $\H$
and $X$ commutes with $\rho(\A)$, 
then $\rho$ has a fully extremal coextension $\sigma$ on $\K \supset \H$
and $X$ has a coextension $Y$ in $\B(\K)$ with $\|Y\|=\|X\|$ 
which commutes with $\sigma(\A)$.

An operator algebra $\A$ has the \textit{weak commutant lifting property} (WCLT)
if whenever $\rho$ is a completely contractive representation of $\A$  on $\H$
and $X$ commutes with $\rho(\A)$, 
then $\rho$ has an extremal coextension $\sigma$ on $\K \supset \H$
and $X$ has a coextension $Y$ in $\B(\K)$ with $\|Y\|=\|X\|$ 
which commutes with $\sigma(\A)$.
\end{defn}

The important distinction is that in SCLT, the coextension is prescribed first,
while in CLT, it may depend on $X$.

It is clear that the more that we restrict the class of coextensions for which 
we have strong commutant lifting, the weaker the property.  Thus SCLT using only
fully extremal coextensions is asking for less than using all extremal coextensions,
which in turn is weaker than using all Shilov extensions.  
As we will want a strong commutant lifting theorem, 
it behooves us to limit the class of extensions.
On the other hand, as we limit the class of coextensions, 
the property CLT becomes stronger.

Observe that for the SCLT and CLT, it suffices to consider minimal
fully extremal extensions.  This is because any fully extremal extension
decomposes as $\sigma = \sigma_0 \oplus \tau$ where $\sigma_0$ is minimal.
Any operator $Y$ commuting with $\sigma(\A)$ will have a 1,1 entry
commuting with $\sigma_0(\A)$.

\begin{eg}
The disk algebra has SCLT by the Sz.Nagy--Foia\c{s} 
Commutant Lifting Theorem \cite{SF_CLT}.
In fact, as noted above, any isometric dilation is an extremal coextension,
but not all are fully extremal.  So $A(\bD)$ has the SCLT with respect to the
larger class of all extremal coextensions, and these are all of the Shilov extensions.
The reason it works is that every isometric coextension splits as 
$\sigma_0 \oplus \tau$ where $\sigma_0$ is the unique minimal isometric coextension.
\end{eg}

\begin{eg}\label{E:SCLTtensor}
The non-commutative disk algebra $\fA_n$ also has SCLT by Popescu's
Commutant Lifting Theorem \cite{Pop1}.  As noted in Example~\ref{Eg:ncdisk}, the
extremal coextensions are the row isometric ones, and these are Shilov.
As in the case of the disk algebra, it is only fully extremal if it is the
direct sum of the minimal isometric coextension with a row unitary.

More generally, the tensor algebra of any C*-correspondence has SCLT by
the Muhly-Solel Commutant Lifting Theorem \cite{MScenv}.
\end{eg}

\begin{eg}
The bidisk algebra $\ADD$ does not have WCLT,
since commutant lifting implies the simultaneous unitary dilation 
of three commuting contractions \cite{Var,Par}.
\end{eg}

\begin{eg}
The algebra $\A_n$ of continuous multipliers on symmetric Fock space
(see Example~\ref{Eg:DA}) has SCLT \cite{DavLe}.  
The extremal extensions are in fact maximal dilations, and so
in particular are fully extremal.
\end{eg}

The relationship between SCLT and CLT is tied to uniqueness of 
minimal coextensions.  We start with an easy lemma.

\begin{lem} \label{L:CLT}
Let $\A$ be a unital operator algebra.  
Suppose that $\sigma$ is a minimal dilation on $\K$ of a representation
$\rho$ on $\H$, in the sense that $\K$ is the smallest reducing subpace
for $\sigma(\A)$ containing $\H$.
If $X$ is a contraction commuting with $\sigma(\A)$ such that $P_\H X|_\H=I$,
then $X=I$.
\end{lem}

\begin{proof}
Since $\|X\|=1$ and $P_\H X|_\H=I$, $X$ reduces $\H$.
Note that for all $h \in \H$ and $a \in \A$,
\begin{align*}
 X \sigma(a) h &= \sigma(a) Xh = \sigma(a) h 
 \shortintertext{and}
 X^* \sigma(a)^* h &= \sigma(a)^* X^*h = \sigma(a)^* h .
\end{align*}
So the restriction of $X$ to $\sigma(\A)\H$ is the identity.
As $X$ is a contraction, it reduces this space.
Similarly, the restriction of $X^*$ to $\sigma(\A)^*\H$ is
the identity; and $X$ reduces this space as well.
Recursively we may deduce that $X$ is the identity
on the smallest reducing subspace containing $\H$, which is $\K$.
\end{proof}

\begin{thm} \label{T:SCLT}
Let $\A$ be a unital operator algebra.  Then $\A$ has SCLT if and only if it has
CLT and unique minimal fully extremal coextensions.
\end{thm}

\begin{proof}
Assume first that $\A$ has CLT and unique minimal fully extremal coextensions.
Let $\rho$ be a representation of $\A$  on $\H$ with a fully extremal coextension 
$\sigma$ on $\K \supset \H$, and suppose that $X$ commutes with $\rho(\A)$.
By CLT, there is a fully extremal coextension $\tau$ on $\L \supset \H$
and $X$ has a coextension $Z$ in $\B(\L)$ with $\|Z\|=\|X\|$ 
which commutes with $\tau(\A)$.
By uniqueness of minimal fully extremal coextensions, 
there is a fully extremal coextension $\mu$ on $\K_0$ so that 
\[ \sigma \simeq \mu \oplus \sigma' \qand \tau \simeq \mu \oplus \tau' .\]
With respect to the latter decomposition, $Z$ can be written as a $2\times2$ matrix
commuting with $\mu(a) \oplus \tau'(a)$ for all $a \in \A$.
Moreover, the corner entry $Z_{11}$ is a coextension of $X$.
A simple calculation of the commutator shows that $\mu(a)$ commutes with $Z_{11}$.
Thus $Y \simeq Z_{11} \oplus 0$ is the desired 
coextension of $X$ commuting with $\sigma(\A)$.

Conversely, suppose that $\A$ has SCLT.  A fortiori, it will have CLT.
Suppose that a representation $\rho$ on $\H$ has two 
minimal fully extremal coextensions $\sigma_1$ and $\sigma_2$ 
on $\K_1 = \H \oplus \K'_1$ and $\K_2 = \H \oplus \K'_2$, respectively.  
Then $\rho\oplus\rho$ has $\sigma_1\oplus\sigma_2$ as a fully extremal coextension.  
This can be seen, for example, because of the identities (\ref{E1}) and (\ref{E2})
in the proof of Theorem~\ref{T:fully extremal}.
The operator $X = \begin{sbmatrix}0&I\\I&0\end{sbmatrix}$ commutes with
$(\rho\oplus\rho)(\A)$.
So by SCLT, $X$ has a coextension $Y$ on $\K_1\oplus\K_2$ of norm $1$
which commutes with $(\sigma_1\oplus\sigma_2)(\A)$.
Since $P_{\H\oplus\H}Y|_{\H\oplus\H} = X$ is unitary, $Y$ reduces $\H\oplus\H$.

Now $Y^2$ commutes with $\sigma_1\oplus\sigma_2(\A)$ and its restriction to 
$\H\oplus\H$ is $X^2 = I$.
Thus by Lemma~\ref{L:CLT}, $Y^2=I$.
In particular, $Y$ is unitary.
Let 
\[ Y_{12} = P_{\K_1} Y|_{\K_2} \qand Y_{21} = P_{\K_2} Y|_{\K_1} .\]
Observe that for $a \in \A$, 
\[
 Y_{21} \sigma_1(a) = \sigma_2(a) Y_{21} 
 \qand
 Y_{12} \sigma_2(a) = \sigma_1(a) Y_{12}.
\]
Moreover the restriction of $Y_{21}$ to $\H$ is $X$ restricted to $\H \oplus \{0\}$,
which is the identity map if we identify $\H \oplus \{0\}$ and $\{0\} \oplus \H$ with $\H$.
We deduce that $Y_{12}Y_{21}$ commutes with $\sigma_1(\A)$ and
coincides with $I$ on $\H$.
So by Lemma~\ref{L:CLT}, 
\[ Y_{12}Y_{21}=I .\]
Similarly, 
\[Y_{21}Y_{12} = I .\]
Since they are contractions, $Y_{12}$ is unitary and $Y_{21}=Y_{12}^*$.
The identities above now show that $Y_{21}$ implements a unitary equivalence
between $\sigma_1$ and $\sigma_2$ fixing $\H$.
Hence the minimal fully extremal coextension of $\rho$ is unique.
\end{proof}

We can weaken CLT to WCLT if we strengthen the uniqueness hypothesis to minimal 
extremal coextensions.  This seems a fair bit stronger in comparison however.

\begin{cor}
If $\A$ has WCLT and unique minimal extremal coextensions, then $\A$ has SCLT.
\end{cor}

\begin{proof}
If $\rho$ is a representation of $\A$, then the unique minimal extremal coextension
$\sigma$ of $\rho$ must be fully extremal, since by Theorem~\ref{T:fully extremal},
there are fully extremal coextensions and hence there are minimal ones.
These are also minimal as extremal coextensions.  Thus $\A$ has 
unique minimal fully extremal coextensions.  Moreover, as in the proof above,
if a contraction $X$ commutes with $\rho(\A)$, then WCLT provides a coextension to a 
contraction $Y$ commuting with an extremal coextension $\tau$.  But $\tau = \sigma \oplus \tau'$.
So arguing as before, the compression $Z$ of $Y$ to $\K_\sigma$ commutes with $\sigma(\A)$
and is a coextension of $X$.  Now if $\phi$ is an arbitrary extremal coextension of $\rho$,
again split $\phi = \sigma \oplus \phi'$.  One extends $Z$ to $Z \oplus 0$ to commute with $\phi(\A)$.
\end{proof}

It is common to look for a version of commutant lifting for intertwining maps 
between two representations. 
In the case of WCLT and SCLT, this is straightforward.
Such a version for CLT is valid here too, but some
care must be taken.

\begin{prop} \label{P:SCLT intertwine}
Suppose that $\A$ has SCLT.
Let $\rho_i$ be representations of $\A$ on $\H_i$ for $i=1,2$
with fully extremal coextensions $\sigma_i$ on $\K_i$.
Suppose that $X$ is a contraction in $\B(\H_2,\H_1)$ such that
\[ \rho_1(a) X = X \rho_2(a) \qforal a \in \A .\]
Then there is a contraction $Y$ in $\B(\K_2,\K_1)$ so that 
\[ P_{\H_1} Y = X P_{\H_2} \]
and
\[ \sigma_1(a) Y = Y \sigma_2(a) \qforal a \in \A .\]
\end{prop}

\begin{proof}
Let $\rho = \rho_1 \oplus \rho_2$.
By Corollary~\ref{C:extremal sums}, $\sigma = \sigma_1 \oplus \sigma_2$
is a fully extremal coextension of $\rho$.
Observe that 
\[ \tilde X = \begin{bmatrix} 0 & X \\ 0 & 0 \end{bmatrix} \]
commutes with $\rho(\A)$.
Hence by SCLT, there is a coextension $\tilde Y$ of $\tilde X$
which commutes with $\sigma(\A)$.

Write $Y$ as a matrix with respect to
\[
 \K = \K_1 \oplus \K_2 =
 \H_1 \oplus \H_2 \oplus (\K_1 \ominus \H_1) \oplus (\K_2 \ominus \H_2) 
\]
and rearrange this to the decomposition
\[ \K = \H_1 \oplus (\K_1 \ominus \H_1) \oplus \H_2 \oplus (\K_2 \ominus \H_2) .\]
We obtain the unitary equivalence
\[
 \tilde Y = \begin{bmatrix} \tilde X & 0 \\ * & * \end{bmatrix} = 
       \left[ \begin{array}{cc|cc} 0 & X & 0 & 0 \\ 0 & 0 & 0 & 0 \\ \hline
       * & * & * & * \\ * & * & * & * \end{array} \right] \simeq
       \left[ \begin{array}{cc|cc} 0 & 0 & X & 0 \\ * & * & * & * \\ \hline
       0 & 0 & 0 & 0 \\ * & * & * & * \end{array} \right]
\]
Restricting to the upper right $2\times2$ corner, we obtain 
\[ Y := P_{\K_1} \tilde Y|_{\K_2} = \begin{bmatrix} X & 0 \\ * & * \end{bmatrix} .\]
Then $Y$ is a contraction, and as an operator in $\B(\K_2,\H_1)$ we have
\[ P_{\H_1} Y = \begin{bmatrix} X & 0 \end{bmatrix} = X P_{\H_2} .\]
Finally the commutation relations show that
\[ \sigma_1(a) Y = Y \sigma_2(a) \qforal a \in \A . \qedhere\]
\end{proof}

A similar argument shows the following:

\begin{prop} \label{P:WCLT intertwine}
Suppose that $\A$ has WCLT. 
Let $\rho_i$ be representations of $\A$ on $\H_i$ for $i=1,2$.
Suppose that $X$ is a contraction in $\B(\H_2,\H_1)$ such that
\[ \rho_1(a) X = X \rho_2(a) \qforal a \in \A .\]
Then there are extremal coextensions $\sigma_i$ of $\rho_i$
acting on $\K_i \supset \H_i$ for $i=1,2$ and a contraction $Y$ 
in $\B(\K_2,\K_1)$ so that 
\[ P_{\H_1} Y = X P_{\H_2} \]
and
\[ \sigma_1(a) Y = Y \sigma_2(a) \qforal a \in \A .\]
\end{prop}

\begin{proof}
Again form $\rho = \rho_1 \oplus \rho_2$ and $\tilde X$ as above.
Use WCLT to coextend $\rho$ to an extremal coextension $\sigma$ 
and $\tilde X$ to a contraction $\tilde Y$ commuting with $\sigma(\A)$
on $\K = \H_1 \oplus \H_2 \oplus \K'$.
Now notice that $\sigma$ is an extremal coextension of both $\rho_i$.
Considering $\tilde Y$ as a map from 
\[ \H_1 \oplus (\H_2 \oplus \K') \quad\text{to}\quad \H_2 \oplus (\H_1 \oplus \K') ,\]
one finds that it has a matrix form
\[
 \tilde Y =  \left[ \begin{array}{c|cc} X & 0 & 0 \\  \hline
       0 & 0 & 0 \\ * & * & *  \end{array} \right] .
\] 
This has the desired form.
\end{proof}

\begin{rem}
The issue with CLT and fully extremal coextensions is that a fully extremal
coextension of $\rho = \rho_1 \oplus \rho_2$ need not even contain a fully
extremal coextension of $\rho_i$ as a summand. 
Consider the subalgebra $\A \subset \fM_5$ given by
\[ \A = \spn\{ E_{21}, E_{32}, E_{31}, E_{34},E_{45}, E_{35}, E_{ii} : 1 \le i \le 5 \} .\]
Let 
\[ \rho_1(a) = E_{22}a|_{\bC e_2} \qand \rho_2(a) = E_{44}a|_{\bC e_4} .\]
The minimal fully extremal coextensions of $\rho_i$ are
\[
 \sigma_1(a) = E_{11}^\perp a|_{(\bC e_1)^\perp} \qand 
 \sigma_2(a) = E_{55}^\perp a|_{(\bC e_5)^\perp} .
\]
However the minimal fully extremal coextension of $\rho_1\oplus\rho_2$ is
\[ \sigma(a) = (E_{22}+E_{33}+E_{44}) a|_{\spn\{e_2,e_3,e_4\}} .\]
\end{rem}

Thus a proof of the following result must follow different lines.
This proof has its roots in the work of Sz.Nagy and Foia\c{s}.
Notice that it allows a specification of one of the coextensions.
Normally we will use a fully extremal coextension $\sigma_2$ of $\rho_2$.

\begin{thm} \label{T:CLT intertwine}
Suppose that $\A$ has CLT. 
Let $\rho_i$ be representations of $\A$ on $\H_i$ for $i=1,2$.
Suppose that $X$ is a contraction in $\B(\H_2,\H_1)$ such that
\[ \rho_1(a) X = X \rho_2(a) \qforal a \in \A .\]
Let $\sigma_2$ be an extremal coextension of $\rho_2$ on $\K_2$.
Then there is a fully extremal coextension $\sigma_1$ of $\rho_1$
acting on $\K_1$ and a contraction $Y$ in $\B(\K_2,\K_1)$ so that 
\[ P_{\H_1} Y = X P_{\H_2} \]
and
\[ \sigma_1(a) Y = Y \sigma_2(a) \qforal a \in \A .\]
\end{thm}

\begin{proof}
Let $\K_2 = \H_2 \oplus \K'_2$, and decompose 
\[
 \sigma_2(a) = 
 \begin{bmatrix} \rho_2(a) & 0 \\ \sigma_{21}(a) & \sigma_{22}(a) \end{bmatrix} 
\]
Observe that $\big[ X \ 0 \big] \in \B(\K_2,\H_1)$ satisfies
\[
 \rho_1(a) \big[ X \ 0 \big] = \big[ \rho_1(a) X \ 0 \big] = 
 \big[ X \ 0 \big]  
 \begin{bmatrix} \rho_2(a) & 0 \\ \sigma_{21}(a) & \sigma_{22}(a) \end{bmatrix} .
\]
Therefore 
\[
 \tilde X =  \left[ \begin{array}{c|cc} 0 & X & 0 \\  \hline
       0 & 0 & 0 \\ 0 & 0 & 0  \end{array} \right]
\]  
commutes with the range of $\rho = \rho_1 \oplus \sigma_2$.   

Now we apply the CLT property to $\rho$ and $\tilde X$ to obtain a fully extremal
coextension $\tau$ of $\rho$ and contraction $\tilde Y$ coextending $\tilde X$
on $\L = \H_1 \oplus \K_2 \oplus \L'$ which commute.  
We may write
\[
 \tau(a) = \begin{bmatrix} \rho_1(a) & 0 & 0 \\ 0 & \sigma_2(a) & 0 \\
              \tau_{31}(a) & \tau_{32}(a) & \tau_{33}(a) \end{bmatrix} .
\]
Observe that the lower right $2 \times 2$ corner is a coextension of $\sigma_2$.
Since $\sigma_2$ is extremal, we see that $\tau_{32} = 0$.
Define
\[
 \sigma_1(a) = 
 \begin{bmatrix} \rho_1(a) & 0 \\ \tau_{31}(a) & \tau_{33}(a) \end{bmatrix} .
\]

To complete the proof, we need to establish that $\sigma_1$ is a 
fully extremal coextension of $\rho_1$.
Suppose that $\gamma$ is a representation of $\A$ which dilates $\sigma_1$
and coextends $\rho_1$.  
Then $\gamma \oplus \sigma_2$ dilates $\sigma_1 \oplus \sigma_2 = \tau$
and coextends $\rho_1 \oplus \sigma_2$.
Since $\tau$ is fully extremal, $\gamma = \sigma_1 \oplus \gamma'$ as desired.
\end{proof}

We make a few more definitions. (Apologies for all the acronyms.)

\begin{defn}
If $\A^*$ has SCLT, CLT or WCLT, we say that $\A$ has SCLT*, CLT* or WCLT*.

Say that $\A$ has \textit{maximal commutant lifting} (MCLT) if 
for every representation $\rho$ on $\H$ 
and contraction $X$ commuting with $\rho(\A)$, there is a maximal dilation
$\pi$ of $\rho$ on a Hilbert space $\K \supset\H$ and a contraction $Y$
commuting with $\pi(\A)$ such that 
\[ P_\H \pi(a) Y^n |_\H = \rho(a) X^n \qforal a \in \A \AND n \ge 0 . \]
If the maximal dilation $\pi$ can be specified a priori, then say that
$\A$ has \textit{strong maximal commutant lifting} (SMCLT).
\end{defn}

It is clear that the commutant lifting properties for $\A^*$ can be interpreted as
lifting commutants to (fully) extremal extensions instead of coextensions.
On rare occasions, one gets both.  For example, the disk algebra $\AD$ is
completely isometrically isomorphic to its adjoint algebra.  Hence it has
both SCLT and SCLT*. 

The property MCLT for $\A$ is equivalent to MCLT for $\A^*$.
So we will not have a property MCLT*.

The definition of MCLT contains the information that the compression
of the algebra generated by $\pi(\A)$ and $Y$  to $\H$ is an algebra
homomorphisms which sends $\pi$ to $\rho$ and $Y$ to $X$.
It follows that $\H$ is semi-invariant for this algebra; i.e. $\H$ is the difference
of two subspaces which are invariant for both $\pi(\A)$ and $Y$.

\begin{thm}\label{T:MCLT}
If $\A$ has WCLT and WCLT*, then $\A$ has MCLT.
\end{thm}

\begin{proof}
One uses WCLT to coextend $\rho$ on $\H$ to an extremal 
coextension $\sigma_1$ on $\K_1$ and coextend $X$ to a
 contraction $Y_1 \in \B(\K_1)$ commuting with $\sigma_1(\A)$.
Then use WCLT* extend $\sigma_1$ to an extremal extension $\tau_1$ on $\L_1$,
and lift $Y_1$ to a contraction $Z_1 \in \B(\L_1)$ commuting with $\tau_1(\A)$.
Alternate these procedures, obtaining an extremal coextension 
$\sigma_{n+1}$ of $\tau_n$ on $\K_{n+1}$ and a contractive coextension
$Y_{n+1} \in \B(\K_{n+1})$ in the commutant of $\sigma_{n+1}(\A)$; and
then extending $\sigma_{n+1}$ to an extremal $\tau_{n+1}$ on $\L_{n+1}$ 
and extending $Y_{n+1}$ to a contraction $Z_{n+1} \in \B(\L_{n+1})$ 
in the commutant of $\tau_{n+1}(\A)$.
It is easy to see that at every stage, the original space $\H$ is semi-invariant for
both the representation and the contraction---so that these are always simultaneous
dilations of the representation and the contraction.

Moreover, we can write $\sigma_{n+1}$ as a dilation of $\sigma_n$ in the matrix form 
relative to 
\[ \K_{n+1} = (\L_n \ominus \K_n) \oplus \K_n \oplus (\K_{n+1} \ominus \L_n) \]
as
\[
 \sigma_{n+1} = 
 \begin{bmatrix} * & 0 & 0 \\ * & \sigma_n & 0 \\ * & * & * \end{bmatrix}
\]
where the upper left $2 \times 2$ corner represents $\tau_n$.
The lower right $2 \times 2$ corner is a coextension of $\sigma_n$.
Since $\sigma_n$ is an extremal coextension, the $3,2$ entry is $0$.
Rearranging this as 
\[ \K_{n+1} = \K_n \oplus (\L_n \ominus \K_n) \oplus (\K_{n+1} \ominus \L_n) ,\]
we have
\[
 \sigma_{n+1} = 
 \begin{bmatrix} \sigma_n & * & 0 \\ 0 & * & 0 \\ 0 & * & * \end{bmatrix} .
\]
A similar analysis holds for the $\tau_n$. 
Therefore, with respect to 
\[
 \K_1 \oplus (\L_1 \ominus \K_1) \oplus 
 (\K_2 \ominus \L_1) \oplus (\L_2 \ominus \K_2) \oplus \dots ,
\]
these representations have a tridiagonal form
\[
  \begin{bmatrix} \sigma_1 & * & 0 & 0 & 0 & 0 & 0 & \dots\\
   0 & * & 0 & 0 & 0 & 0 & 0 & \dots\\
   0 & * & * & * & 0 & 0 & 0 & \dots \\ 
   0 & 0 & 0 & * & 0 & 0 & 0 & \dots \\
   0 & 0 & 0 & * & * & * & 0 & \dots \\
   0 & 0 & 0 & 0 & 0 & * & 0 & \dots \\
   \vdots & \vdots & \vdots & \vdots & \vdots & \ddots & \ddots & \ddots 
   \end{bmatrix} .
\]
It follows that the direct limit $\pi$ exists as a \sot-$*$ limit.
Moreover, as $\pi$ is a limit of extremal coextensions, it is an extremal coextension;
and similarly it is an extremal extension.  Thus it is a maximal dilation.

The operators $Y_n$ and $Z_n$ each leave the subspaces $\H$ semi-invariant,
and the restriction of $Z_n$ to $\K_n$ is $Y_n$, and the compression (actually
a co-restriction) of $Y_{n+1}$ to $\L_n$ is $Z_n$.
Therefore the direct limit $Y$ exists as a \wot\ limit.
It follows that $Y$ is a contraction that commutes with $\pi(\A)$.
To see this, let $x_n \in \K_n$ and $y_n \in \L_n$.
Then for $a\in\A$, 
\[ \pi(a) x_n = \sigma_k(a) x_n \qforal k \ge n .\]
Similarly, 
\[ \pi(a)^* y_n = \sigma_k(a)^* y_n \qforal k \ge n+1 . \]
So we can compute:
\begin{align*}
 \bip{ (\pi(a)Y - Y \pi(a)) x_n, y_n} &= 
 \bip{x_n, Y^* \pi(a)^* y_n} - \bip{ Y \pi(a) x_n, y_n} \\ &=
 \lim_{k\to\infty} \bip{x_n, Y_k^* \pi(a)^* y_n} - \bip{ Y_k \pi(a) x_n, y_n} \\ &=
 \lim_{k\to\infty} \bip{x_n, Y_k^* \sigma_k(a)^* y_n} - \bip{ Y_k \sigma_k(a) x_n, y_n} \\ &=
 \lim_{k\to\infty}  \bip{ (\sigma_k(a)Y_k- Y_k \sigma_k(a))  x_n, y_n} \\ &= 0.
\end{align*}
So we have obtained the desired commutant lifting.
\end{proof}

We can modify the proof of Theorem~\ref{T:SCLT} characterizing 
SCLT to characterize SMCLT.

\begin{thm} \label{T:SMCLT}
The following are equivalent for $\A$:
\begin{enumerate}
\item $\A$ has SMCLT.
\item $\A$ has MCLT and unique minimal maximal dilations.
\end{enumerate}
\end{thm}

\begin{proof}
Suppose that $\A$ has SMCLT.  Then evidently it has MCLT.
Moreover, suppose that a representation $\rho$ has two minimal 
maximal dilations $\pi_1$ and $\pi_2$.  
Then $\pi_1 \oplus \pi_2$ is a maximal dilation of $\rho \oplus \rho$.
Now $(\rho \oplus \rho)(\A)$ commutes with $X = \begin{sbmatrix}0&I\\I&0\end{sbmatrix}$.
By SMCLT, this dilates to a contraction $Y$ which commutes with 
$(\pi_1 \oplus \pi_2)(\A)$ and $\H\oplus\H$ is jointly semi-invariant for
$Y$ and $(\pi_1 \oplus \pi_2)(\A)$.
Arguing exactly as in the proof of Theorem~\ref{T:SCLT}, we deduce that
$\pi_1$ and $\pi_2$ are unitarily equivalent via a unitary which fixes $\H$.
So $\A$ has unique maximal dilations.

Conversely, it is routine to see that unique maximal dilations and MCLT 
yields SMCLT. So (i) and (ii) are equivalent.
\end{proof}

\begin{rem} One might suspect, as we did, that SMCLT is also equivalent to the following:
\begin{enumerate}
\addtocounter{enumi}{2}
\item $\A$ has SCLT and SCLT*.
\item $\A$ has MCLT and unique minimal fully extremal extensions and coextensions.
\end{enumerate}
But this is not the case.

If (ii) holds, then by Theorem~\ref{T:fullextoffullcoext} there is uniqueness of 
minimal fully extremal coextensions; and Corollary~\ref{C:fullcoextoffullext}
yields uniqueness of minimal fully extremal extensions.  So (iv) holds.

Also if (iii) holds, then there are unique minimal fully extremal extensions 
and coextensions by Theorem~\ref{T:SCLT} and its dual result for SCLT*.
Also MCLT holds by Theorem~\ref{T:MCLT}.  So (iv) holds.

However, it is possible that SCLT and SCLT* hold, yet $\A$ does not
have unique minimal maximal dilations.  See the example of $2\times2$ matrices
developed in the next section.  Thus SMCLT fails to hold.

We do not know if SMCLT implies (iii).
\end{rem}

In the case of semi-Dirichlet algebras, we have something extra.
We do not know if all semi-Dirichlet algebras have SCLT.
However, Muhly and Solel \cite{MScenv} show that the tensor algebra
over any C*-correspondence has SCLT.

\begin{prop}\label{P:sD-MCLT}
Suppose that $\A$ is semi-Dirichlet and has MCLT.
Then $\A$ has SCLT.
\end{prop}

\begin{proof}
Let $\rho$ be a representation of $\A$ on $\H$ commuting with a contraction $X$.
Use MCLT to obtain a simultaneous dilation of $\rho$ to a maximal dilation $\pi$
and $X$ to a contraction $Y$ commuting with $\pi(\A)$.
Let $\K$ be the common invariant subspace for $\pi(\A)$ and $Y$ containing $\H$.
Since $\H$ is semi-invariant, the restriction of $\pi$ to $\K$ is a coextension
$\sigma$ of $\rho$.  The compression $Z$ of $Y$ to $\K$ is a contraction commuting
with $\sigma(\A)$.

By Theorem~\ref{T:sD}, there is a unique minimal extremal coextension $\sigma_0$
of $\rho$, and it must coincide with $\sigma|_{\K_0}$ where $\K_0 = \sigma(\A)\H$.
Thus $\sigma = \sigma_0 \oplus \sigma'$.
It follows that the compression $Z_0$ of $Z$ to $\K_0$ commutes with $\sigma_0(\A)$.
Moreover, since $\H$ is coinvariant for $Z$, it is also coinvariant for $Z_0$.
By Theorem~\ref{T:SCLT}, $\A$ has SCLT.
\end{proof}

\begin{cor} \label{C:D-MCLT}
If $\A$ is a Dirichlet algebra, the following are equivalent:
\begin{enumerate}
\item $\A$ has MCLT
\item $\A$ has SMCLT
\item $\A$ has SCLT and SCLT*
\item $\A$ has WCLT and WCLT*.
\end{enumerate}
\end{cor}

Finally we point out that there is also an intertwining version for MCLT.

\begin{prop} \label{P:MCLT intertwine}
Suppose that $\A$ has MCLT. 
Let $\rho_i$ be representations of $\A$ on $\H_i$ for $i=1,2$.
Suppose that $X$ is a contraction in $\B(\H_2,\H_1)$ such that
\[ \rho_1(a) X = X \rho_2(a) \qforal a \in \A .\]
Then there are maximal dilations $\pi_i$ of $\rho_i$
acting on $\K_i \supset \H_i$ for $i=1,2$ and a contraction $Y$ 
simultaneously dilating $X$ in $\B(\K_2,\K_1)$ so that 
\[ \pi_1(a) Y = Y \pi_2(a) \qforal a \in \A .\]
\end{prop}

\begin{proof}
Again form $\rho = \rho_1 \oplus \rho_2$ and $\tilde X$ as before.
Use MCLT to dilate $\rho$ to a maximal dilation $\pi$ 
and $\tilde X$ to a contraction $\tilde Y$ commuting with $\pi(\A)$
on 
\[ \K = \K_- \oplus \H_1 \oplus \H_2 \oplus \K_+ .\]
Write
\[
 \tilde Y =  \begin{bmatrix} * & 0 & 0 & 0 \\  * & 0 & X & 0 \\
      * & 0 & 0 & 0 \\ * & * & * & * \end{bmatrix} .
\]
Now notice that $\pi$ is an maximal dilation of both $\rho_i$.
Considering $\tilde Y$ as a map from 
\[ 
 \K_- \oplus \H_1 \oplus (\H_2 \oplus \K_+) \quad\text{to}\quad 
 (\K_- \oplus \H_1) \oplus \H_2 \oplus \K_+
\]
one finds that it has a matrix form
\[
 \tilde Y =  
  \left[ \begin{array}{cc|c|c} * & 0 & 0 & 0 \\ \hline * & 0 & X & 0 \\ \hline
      * & 0 & 0 & 0 \\ * & * & * & * \end{array} \right] .
\] 
This is a dilation of $X$ which commutes with $\pi(\A)$, where $\pi$
is considered as a dilation of both $\rho_1$ and $\rho_2$ using the
two decompositions of $\K$.
\end{proof}

\section{A $\boldsymbol{2 \times 2}$ matrix example.} \label{S:2x2}

Consider the algebra $\A = \spn\{I_2, n\} \subset \fM_2$ where
\[ n = \begin{bmatrix}0&0\\1&0\end{bmatrix} .\]
Observe that $\A$ is unitarily equivalent to $\A^*$.
It is not semi-Dirichlet.
We will show that it has unique minimal extremal (co)extensions, which are
always maximal dilations.  But it does not have unique maximal dilations.
It also has SCLT, SCLT* and MCLT, but does not have SMCLT.

Observe that a representation $\rho$ of $\A$ is determined by $N:= \rho(n)$,
which satisfies $N^2=0$ and $\|N\|\le1$.  Conversely, any such $N$ yields 
a completely contractive representation.
It is easy to check that $N$ is unitarily equivalent to an operator of the form
\[ \begin{bmatrix}0&0\\B&0\end{bmatrix} \]
where $B$ has dense range
by setting $\H_2 = \ol{\ran N}$ and decomposing $\H = \H_1 \oplus \H_2$.
This can be refined by using the polar decomposition of $B$ to the form
\[
 \begin{bmatrix}0&0\\A&0\end{bmatrix} \oplus 0 \quad\text{on }
 \H =\ran N^* \oplus \ran N \oplus (\ker N \cap \ker N^*) ,
\]
where $A$ is a positive injective operator.

Clearly $\cenv(\A) = \fM_2$.
Thus a maximal representation $\pi$ extends to a $*$-representation of $\fM_2$.
Hence $\pi(n) = N$ is a partial isometry such that 
\[ NN^* + N^*N = I \qand  N^2=0 ;\]
or equivalently, 
\[ (N+N^*)^2=I \qand  N^2=0 .\]
In other words, there is a unitary $U$ so that
\[
 N \simeq \begin{bmatrix}0&0\\U&0\end{bmatrix} 
  \simeq \begin{bmatrix}0&0\\I&0\end{bmatrix} .
\]
Geometrically, this says that $N$ is a partial isometry such that 
\[ \ran N + \ran N^* = \H .\]

\begin{prop} \label{T:2x2}
The algebra $\A$ of $2\times2$ matrices of the form
\[ \begin{bmatrix}a&0\\b&a\end{bmatrix} \]
has unique minimal extremal coextensions $($extensions$)$.
Moreover they are fully extremal coextensions $($extensions$)$,
and in fact, are maximal dilations; 
and the original space is cyclic $($cocyclic$)$.
\end{prop}

\begin{proof}
We first show that a representation $\rho$ which is not maximal has a proper coextension.
Use the matrix form 
\[ \rho(n) = N=\begin{bmatrix}0&0\\B&0\end{bmatrix} \] 
on $\H=\H_1\oplus\H_2$ where $B$ has dense range.
Suppose that $B$ is not an isometry, 
and let 
\[ D_B = (I-B^*B)^{1/2} \qand  \D_B = \ol{\ran D_B} .\] 
Consider $P_{\D_B}D_B$ as an operator from $\H_1$ to $\D_B$.
Then $B$ may be coextended to an isometry
\[ V = \begin{bmatrix}B\\P_{\D_B}D_B\end{bmatrix} \]
mapping $\H_1$ to $\K_2 = \H_2 \oplus \D_B$.
So $\rho$ coextends to $\sigma$ where 
\[
 \sigma(n) = S = \begin{bmatrix}0&0\\V&0\end{bmatrix}
 = \begin{bmatrix}0&0&0\\B&0&0\\P_{\D_B}D_B&0&0\end{bmatrix}
\]
on 
\[ \K = \H_1 \oplus \K_2 = \H_1 \oplus \H_2 \oplus \D_B .\]
Hence when $\rho$ is extremal, $B$ is an isometry with dense range,
so it is unitary.  Thus $\rho$ is a maximal representation.

The coextension constructed above is generally not extremal because
$V$ is not unitary.  So one can coextend $S$ again using the same procedure.
$\K_2$ splits as $\R_V \oplus \D$, where 
\[ \R_V = \ran V \qand  \D = \ran(I_{\K_2}-VV^*) .\]
We now decompose 
\[ \K = \H_1 \oplus \D \oplus \R_V , \]
and write
\[ 
  S = \left[\begin{array}{cc|c} 
  0&0&0\\0&0&0\\ 
  \hline P_{\R_V}V&0&0
  \end{array}\right] 
\]
Using the same dilation as the previous paragraph, but noting that
\[ D_{[P_{\R_V}V\ 0]} = (I_{\K_2}-VV^*) ,\]
we obtain the coextension $\tau$ of $\sigma$ on
a Hilbert space 
\[ \L = \H_1 \oplus \D \oplus \R_V \oplus \D \]
by
\[
 \tau(n) = W = 
 \left[\begin{array}{cc|cc} 0&0&0&0\\0&0&0&0\\ \hline
 P_{\R_V}V&0&0&0\\ 0&P_\D(I_{\K_2}-VV^*)&0&0
 \end{array}\right] 
\]
The bottom left $2\times2$ corner is now a surjective isometry.
So this is a maximal dilation.

To see that $\tau$ is minimal as a coextension, we need to verify that
\[ \L = \A \H = \H \vee W\H .\]
To see this, we rewrite $W$ with respect to the decomposition
\[ \L = \H_1 \oplus \H_2 \oplus \D_B \oplus \D\]
to get
\[
 W = 
 \left[\begin{array}{cc|cc} 0&0\ \ \ \ &0&0\\B&0\ \ \ \ &0&0\\ \hline
 P_{\D_B}D_B&0\ \ \ \ &0&0\\ 0&\multicolumn{2}{c}{\!\!\!\!\!\!P_\D(I_{\K_2}-VV^*)}&0
 \end{array}\right] 
\]
Since the range of $D_B$ is dense in $\D_B$, this is contained in $\H \vee W\H$.
For the space $\D$, we expand the expression for $I_{\K_2}-VV^*$ on 
$\K_2 = \H_2 \oplus \D_B$:
\begin{align*}
 I_{\K_2}-VV^* &= 
 \begin{bmatrix}I&0\\0&I\end{bmatrix} - 
 \begin{bmatrix}BB^*&BD_B\\D_BB^*&I-B^*B\end{bmatrix} \\&=
 \begin{bmatrix}I-BB^*&-BD_B\\-D_BB^*&B^*B\end{bmatrix} \\&=
 \begin{bmatrix}D_{B^*}^2&-BD_B\\-BD_{B^*}&B^*B\end{bmatrix} 
\end{align*} 
Thus one sees that 
\[
 (I_{\K_2}-VV^*)P_{\H_2} = 
 \begin{bmatrix} \ D_{B^*}\\-B\end{bmatrix}D_{B^*}  .
\]
Observe that $D_{B^*}$ maps $\H_2$ onto a dense subspace of $\ker(I-BB^*)^\perp$,
and 
\[ \ker(I-BB^*) = \ran V \cap \H_2 \subset \ker (I-VV^*)\cap \H_2 .\]
The range of 
\[ \begin{bmatrix} \ D_{B^*}\\-B\end{bmatrix} \]
is easily seen to  be the orthogonal complement of $\ran V$, so this is $\D$.  
Restricting this map to the range of $D_{B^*}$ does not affect the closed range,
since we only miss some of the kernel.
Thus $W\H_2$ is dense in $\D$.
Therefore this is a minimal extremal coextension.

Now we consider uniqueness.
To this end, suppose that $\tau'$ is a minimal extremal coextension of $\rho$
on 
\[ \L' = \H_1 \oplus \H_2 \oplus \L_3 .\]
Then we can write
\[
 \tau'(n) = W' = \begin{bmatrix}0&0&0\\B&0&0\\X&Y&Z\end{bmatrix} .
\]
This is a partial isometry satisfying 
\[ W^{\prime 2}=0 \qand W'W^{\prime*} + W^{\prime*}W' = I . \]
In particular, $\H_1$ is orthogonal to the range of $W'$, so the 
restriction of $W'$ to $\H_1$ is an isometry.
Therefore 
\[ X = UP_{\D_B}D_B \]
where $U$ is an isometry of $\D_B$ into $\L_3$.
Split 
\[ \L_3 = U\D_B \oplus \L_4 \simeq \D_B \oplus \L_4 .\]
By identifying the range of $U$ with $\D_B$, we have the refined form
\[
 W' = 
 \begin{bmatrix}0&0&0&0\\B&0&0&0\\
 P_{\D_B}D_B&X_1&Y_1&Z_1\\ 0 & X_2&Y_2&Z_2
 \end{bmatrix} .
\]
The range of $W' \H_1^\perp$ is orthogonal to 
\[ W'\H_1 \vee \H_2 = \H_2 \oplus \D_B .\]
So 
\[ X_1=Y_1=Z_1=0 .\]

Next note that minimality ensures that $X_2$ has dense range in $\L_4$.
So $\L_4$ is in 
\[ \ran W' = \ker W^{\prime*} .  \]
Hence $Z_2=0$.
Observe that 
\[ \ran W' = \ran V \oplus \L_4 , \]
and hence 
\[ \ran W^{\prime*} = (\ran W')^\perp = \H_1 \oplus \D .\]
Moreover, the operator 
\[ \big[ X_2 \ Y_2 \big] \]
is an isometry of $\D$ onto $\L_4$.
Hence we may identify $\L_4$ with $\D$ in such a way that we obtain
\[ \big[ X_2 \ Y_2 \big] \simeq P_\D(I_{\K_2}-VV^*) .\]
This shows that $W'$ is unitarily equivalent to $W$ via a unitary which fixes $\H$.
Therefore the minimal extremal coextension is unique.

The proof for extensions follows immediately since 
$\A^*$ is unitarily equivalent to $\A$.
\end{proof}

\begin{cor}
Every representation of the algebra $\A$ is Shilov.
\end{cor}

\begin{proof}
By the previous theorem, one can extend $\rho$ to a maximal dilation $\pi$.
Thus $\rho$ is obtained as the restriction of a maximal representation to
an invariant subspace; i.e. it is Shilov.
\end{proof}

\begin{eg}
Take $\rho$ to be the character representation on $\H = \bC$ given by
\[ \rho(aI_2 + bn) = a .\]
This coextends to a maximal representation on $\K = \H \oplus \bC$ as
\[ \sigma(aI_2 + bn) = \begin{bmatrix}a&0\\b&a\end{bmatrix} .\]
It also extends to a maximal representation $\tau$ on $\K = \bC \oplus \H$
where 
\[ \tau(aI_2 + bn) = \begin{sbmatrix}a&0\\b&a\end{sbmatrix} .\]
Note that $\sigma$ and $\tau$ are not unitarily equivalent by a \textit{unitary
which fixes $\H$}.  So these are inequivalent minimal maximal dilations.
\end{eg}

\begin{cor}
The minimal maximal dilation of a representation $\rho$ of $\A$ is not unique
except when $\rho$ is already maximal.
\end{cor}

\begin{proof}
Let $\tau$ be the minimal extremal coextension of $\rho$, 
and let $\sigma$ be the minimal extremal extension of $\rho$.
Then in the first case, $\H$ is identified with a coinvariant subspace
and in the latter with an invariant subspace.
If these two dilations are unitarily equivalent via a unitary which
fixes $\H$, then $\H$ is reducing, and therefore by minimality, 
$\rho=\tau$ is maximal.
\end{proof}

Next we show that $\A$ has commutant lifting.

\begin{thm}
The algebra $\A$ of $2\times2$ matrices of the form
\[ \begin{bmatrix}a&0\\b&a\end{bmatrix} \]
has SCLT, SCLT* and MCLT, but not SMCLT.
\end{thm}

\begin{proof}
It is enough to verify CLT.
Since the minimal fully extremal coextensions are unique, 
it then has SCLT by Theorem~\ref{T:SCLT}.  
Since $\A^* \simeq\A$, it has SCLT* as well.
Thus by Theorem~\ref{T:MCLT}, it has MCLT.  
But by Theorem~\ref{T:SMCLT}, it does not have SMCLT.  

We make use of the construction of the minimal extremal extension in the proof of
Theorem~\ref{T:2x2}.
Write 
\[ \rho(n) = N =  \begin{bmatrix}0&0\\B&0\end{bmatrix} \]
as before on
\[ \H = \H_1 \oplus \H_2 \quad\text{where } \H_2 = \ol{\ran N} .\]
Suppose that it commutes with a contraction $T$.
Then it is routine to check that 
\[ T = \begin{bmatrix}X&0\\Y&Z\end{bmatrix} \]
such that 
\[ BX = ZB .\]
Coextend $\rho$ to the coextension $\sigma$ on 
\[ \K = \H_1 \oplus \H_2 \oplus \D_B \]
where
\[ \sigma(n) = S = \begin{bmatrix}0&0&0\\B&0&0\\D_B&0&0\end{bmatrix} .\]
We first find a coextension of $T$ to $\tilde T$ which 
commutes with $S$ and has norm one.

Consider the isometric dilation of $N$.
Observe that 
\[ D_N = (I-N^*N)^{1/2} = \begin{bmatrix} D_B & 0 \\ 0& I \end{bmatrix} .\]
So the minimal isometric dilation acts on 
\[ (\H_1\oplus\H_2) \oplus (\D_B \oplus \H_2)^{(\infty)}  ,\]
and has the form
\[
 U = \left[\begin{array}{cc|c|cccccc}
 0&0&\,0\,&0&0&0&0&0&\dots\\
 B&0&0&0&0&0&0&0&\dots\\ \hline
 D_B\!&0&0&0&0&0&0&0&\dots\\ \hline
 0&I&0&0&0&0&0&0&\dots\\
 0&0&I&0&0&0&0&0&\dots\\
 0&0&0&I&0&0&0&0&\dots\\
 \vdots&\vdots&\vdots&\ddots\!\!&\ddots\!\!&\ddots\!\!&\ddots\!\!&\ddots\!\!&\ddots
 \end{array}\right]
\]
Notice that $S$ is the upper left $3\times3$ corner.
By the Sz.Nagy-Foia\c{s} Commutant Lifting Theorem, we can coextend
$T$ to a contraction $R$ commuting with $U$.  It has the form
\[
 R = \begin{bmatrix} X&0&0&\dots\\Y&Z&0&\dots\\
 C_1&C_2&C_3&\dots\\ \vdots&\vdots&\vdots&\ddots
\end{bmatrix}
\]
It is routine to verify that 
\[
 \tilde T =  \begin{bmatrix}X&0&0\\Y&Z&0\\C_1&C_2&C_3\end{bmatrix}
\]
commutes with $S$.

Now we repeat the argument with $S$ and $\tilde T$.
The same procedure was shown in Theorem~\ref{T:2x2} to yield the
minimal (fully) extremal coextension $\tau$ of $\rho$.
The operator $\tilde T$ is coextended once again to obtain a contraction
commuting with $\tau(n)=W$.  This establishes SCLT.
\end{proof}

\begin{rem}
We will show that in the commutant lifting theorem for $\A$, it is not possible
to coextend so that $\tilde T$ is an isometry. In the language of the next section,
this will show that $\A$ does not have ICLT (isometric commutant lifting) nor
the Ando property.

To see this, consider the identity representation 
\[ \id(n) = N = \begin{bmatrix}0&0\\1&0\end{bmatrix} \]
on $\H = \bC^2$.
Then $\id(\A)$ commutes with $T=N$.
Suppose that there were a coextension of $\id$ and $T$ to $\sigma$ 
and an isometry $V$ on $\K$ so that $\sigma(\A)$ commutes with $V$.
Since $\id$ is maximal, $\sigma = \id \oplus \tau$.
So 
\[ \sigma(n) = M = N \oplus M_0 \] 
where $M_0^2=0$.
Let the canonical basis for $\H$ be $e_1,e_2$.
Since $Te_2=0$, we have $Ve_2 = v$ is a unit vector in $\H^\perp$;
while 
\[ Ve_1 = Te_1 = e_2 .\]
Therefore
\[
 v = Ve_2 = VMe_1 = MVe_1 = Me_2 = 0 .
\]
This contradiction shows that no such coextension is possible.
\end{rem}

\section{Relative Commutant Lifting and Ando's Theorem} \label{S:ando}

Paulsen and Power \cite{PP_tensor} formulate commutant lifting
and Ando's theorem in terms of tensor products.  In doing so, they are 
also able to discuss lifting commuting relations between two arbitrary
operator algebras.  They are interested in dilations which extend to the 
enveloping C*-algebra, which are the maximal dilations when this
C*-algebra is the C*-envelope.  
The Paulsen-Power version of Ando's theorem involves maximal 
dilations and a commuting unitary.
The classical Ando Theorem, from our viewpoint, states that two commuting contractions
have coextensions to commuting isometries.
We will give a similar definition using extremal co-extensions instead which is
actually stronger than the Paulsen-Power version.

\begin{defn}
Let $\A$ and $\B$ be unital operator algebras.
Say that $\A$ has $\B$-CLT (or \textit{commutant lifting with respect to $\B$})
if whenever $\alpha$ and $\beta$ are (completely contractive) representations
of $\A$ and $\B$ on a common Hilbert space $\H$ which commute:
\[ \alpha(a) \beta(b) = \beta(b) \alpha(a) \qforal a \in \A \AND  b \in \B ,\]
then there exists an 
extremal coextension $\sigma$ of $\alpha$ on a 
Hilbert space $\K$ and a coextension $\tau$ of $\beta$ on $\K$ which commute.

If $\AD$ has $\A$-CLT, then we say that $\A$ has \textit{isometric commutant lifting}
(ICLT).
\end{defn}

Because there is no uniqueness for the classical Ando's theorem, we do not
seek a strong form.  So we do not use the adjective weak either.
Note that ICLT is not stronger than WCLT because the extremal condition
is on the isometry, not on the coextension of $\A$.  Nevertheless it does imply
a much stronger conclusion, as we show in Theorem~\ref{T:ICLT} below.

It is often observed that Ando's Theorem is equivalent to commutant lifting.  
However neither direction is completely trivial.  From Ando's theorem, one easily
gets WCLT.  So the uniqueness of the minimal isometric coextension,
and the fact that this is fully extremal, then yields SCLT.
Conversely, in deducing Ando's theorem from WCLT, one is really using
WCLT for \textit{both} contractions.  One iteratively dilates one contraction
to an isometry and lifts the other to commute.  The inductive limit is a pair
of commuting isometries.  We will see this more clearly for operator algebras
other than the disk algebra.

The next result shows that the Paulsen-Power version of Ando's theorem is
equivalent to ICLT. It also shows why we consider $\A$-CLT for $\AD$
as a strong property for $\A$, and makes it worthy of the term ICLT.

\begin{thm} \label{T:ICLT}
For a unital operator algebra $\A$, the following are equivalent:
\begin{enumerate}
\item $\A$ had ICLT; i.e. if $\rho$ is a representation of $\A$ on $\H$ 
commuting with a contraction $X$, then there is a coextension $\sigma$ of $\rho$ and
an isometric coextension $V$ of $X$ on a common Hilbert space $\K$
which commute.

\item If $\rho$ is a representation of $\A$ on $\H$ commuting with a contraction $X$,
then there is a Shilov coextension $\sigma$ of $\rho$ and
an isometric coextension $V$ of $X$  on a common Hilbert space $\K$
which commute.

\item If $\rho$ is a representation of $\A$ on $\H$ commuting with a contraction $X$,
then there is a simultaneous dilation of $\rho$ to a maximal dilation  $\pi$ on $\K$ 
and of $X$ to a unitary $U$ commuting with $\pi(\A)$; i.e.\ there is a Hilbert space
$\K \supset \H$, a $*$-representation $\pi$ of $\cenv(\A)$ on $\K$ and a unitary operator
$U$ on $\K$ commuting with $\pi(\cenv(\A))$ so that 
\[ \strut\qquad P_\H \pi(a) U^n |_\H = \rho(a) X^n\]
for all $a \in \A \AND n\ge0$. 
\end{enumerate}
\end{thm}

\begin{proof}
It is evident that (iii) implies (ii) by restriction to the smallest invariant subspace
containing $\H$.  And (ii) clearly implies (i).  So assume that (i) holds.
We will establish (iii).

First we dilate $\sigma$ and $V$  to $\tau$ and $W$ so that $W$ is unitary
and commutes with $\tau(\A)$.
To accomplish this, consider the system with $\K_n =\K$ and $V$
considered as a map from $\K_n$ into $\K_{n+1}$:
\[
\xymatrix{
\K_1 \ar[r]^{V}  \ar[d]_V &\K_2 \ar[r]^{V}  \ar[d]_V & \K_3 \ar[r]^{V}  \ar[d]_V 
&\dots \ar[r]&\P \ar[d]^W\\
\K_1 \ar[r]^{V}   &\K_2 \ar[r]^{V}   & \K_3 \ar[r]^{V}  &\dots \ar[r]&\P
}
\]
Then $\P$ is the Hilbert space direct limit of copies of $\K$ under $V$.
Let $J_n$ denote the canonical injection of $\K_n$ into $\P$. 
Thus
\[ J_n = J_{n+1}V \qfor n\ge1 . \]
The map $V$ also determines isometries acting on each $\K_n$,
which we also denote by $V$. The direct limit of this system of maps is a 
unitary $W$ on $\P$ such that its restriction to each $\K_n$ coincides with $V$.
Hence
\[ J_n V = W J_n \qfor n\ge1 . \]
In particular, $W$ is an extension of $V$ acting on $\K_1$, which we identify with $\K$.

We define a representation $\tau$ of $\A$ on $\P$ by
\[ \tau(a) J_n k = J_n \sigma(a) k \qfor a\in\A,\ k\in\K_n \AND n \ge 1 .\]
Clearly each subspace $J_n \K_n$ is invariant for $\tau$ and the restriction 
of $\tau$ to $\K_n$ is equivalent to $\sigma$.
In particular,  $\tau$ is an extension of $\sigma$, where we identify $\K$ with $\K_1$.
Additionally, since $\tau$ is completely contractive when restricted to
each $J_n\K_n$, we see that $\tau$ is completely contractive.
Finally, for $a \in \A$ and $k\in\K_n$ for $n\ge2$,
\begin{align*}
 \tau(a) W J_n(a) k &= \tau(a) J_n V k = \tau(a) J_{n-1} k \\
 &= J_{n-1} \sigma(a) k = J_n V \sigma(a) k\\
 &= W J_n \sigma(a) k = W \tau(a) J_n k .
\end{align*}
Therefore, $W$ commutes with $\tau(\A)$.

The completely contractive map $\tau$ extends to a unique
completely positive unital map on the operator system
\[ \ol{\A + \A^*} \subset \cenv(\A) .  \]
By Fuglede's Theorem, $W$ commutes with $\ol{\tau(\A) + \tau(\A)^*}$.   
The commutant $\fN$ of $W$ is a type I von Neumann
algebra, and therefore it is injective.  
Therefore by Arveson's Extension Theorem,  there is a completely positive 
extension  of $\tau$ to $\cenv(\A)$ with range in $\fN$.  
By Stinespring's Dilation Theorem, there is a minimal dilation to a $*$-representation 
$\pi$ of $\cenv(\A)$ on a larger Hilbert space.
Now  a commutant lifting result of Arveson \cite[Theorem 1.3.1]{Arv1} 
shows that there is a unique extension of $W$
to an operator $U$ commuting with  $\pi(\cenv(\A))$.  
This extension map is a $*$-homomorphism, so $U$ is unitary.  
Moreover the fact that the restriction of $\pi|_\A$ to the space $\P$  is a
homomorphism means that $\pi$ is a maximal dilation of $\tau$, and hence of $\rho$.
\end{proof}

The following corollary is a consequence of (i) implies (iii) above.

\begin{cor} \label{C:ICLTtoMSLT}
Property ICLT implies MCLT for $\A$.
\end{cor}

Another easy corollary is a consequence of the fact that (iii) is symmetric.

\begin{cor} \label{C:ICLTisICLT*}
Property ICLT is equivalent to ICLT*.
So if $\A$ has ICLT, so does $\A^*$.
\end{cor}

\begin{eg} 
Finite dimensional nest algebras have ICLT by Paulsen and Power 
\cite{PP_nest,PP_tensor}.
They actually prove variant (iii).
They also claim that the minimal $*$-dilation is unique.
This follows because finite dimensional nest algebras are Dirichlet.
So by Theorem~\ref{T:SMCLT}, they have SMCLT.

Dirichlet implies semi-Dirichlet for the algebra and its adjoint.
So there are unique minimal fully extremal (co)extensions.
The proof of ICLT in fact first coextends to an isometry in the
commutant.  Hence finite dimensional nest algebras have SCLT and SCLT*.

Nest algebras have the SCLT, SCLT* and MCLT for weak-$*$ 
continuous completely contractive representations.
\end{eg}

The first part of the following proposition uses exactly the same proof as 
Proposition~\ref{P:MCLT intertwine} with the exception that the dilation 
$\tilde Y$ obtained can be taken to be a unitary operator when ICLT
is invoked.

\begin{prop} \label{P:ICLT intertwine}
Suppose that $\A$ has ICLT. 
Let $\rho_i$ be representations of $\A$ on $\H_i$ for $i=1,2$.
Suppose that $X$ is a contraction in $\B(\H_2,\H_1)$ such that
\[ \rho_1(a) X = X \rho_2(a) \qforal a \in \A .\]
Then there are maximal dilations $\pi_i$ of $\rho_i$
acting on $\K_i \supset \H_i$ for $i=1,2$ and a unitary operator $U$ 
simultaneously dilating $X$ in $\B(\K_2,\K_1)$ so that 
\[ \pi_1(a) U = U \pi_2(a) \qforal a \in \A .\]

Consequently, there exist Shilov coextensions $\sigma_i$ of $\rho_i$ on $\L_i$
and a  coextension of $X$ to an isometry $V \in \B(\L_2,\L_1)$ so that
\[\sigma_1(a) V = V \sigma_2(a) \qforal a \in \A .\]
\end{prop}

\begin{proof}
We only discuss the second statement. 
Let $\pi_i$ act on $\K_i = \K_i^- \oplus \H_i \oplus \K_i^+$,
where $\L_i = \H_i \oplus \K_i^+$ and $\K_i^+$ are invariant
subspaces for $\pi_i(\A)$.  With respect to these decompositions,
we have the matrix forms
\[
 \pi_i(a) = 
 \begin{bmatrix} * & 0 & 0 \\ * & \rho_i(a) & 0 \\ * & \pi_{32}(a) & \pi_{33}(a) \end{bmatrix}
 \qand
 U = \begin{bmatrix} * & 0 & 0 \\ *  & X & 0 \\ * & U_{32} & U_{33} \end{bmatrix} .
\]
Set 
\[
 \sigma_i(a) = P_{\L_i} \pi_i(a)|_{\L_i}  = 
 \begin{bmatrix} \rho_i(a) & 0 \\  \pi_{32}(a) & \pi_{33}(a) \end{bmatrix}
 \quad\FOR i = 1,2.
\]
and 
\[ V = P_{\L_1} U|_{\L_2}  = \begin{bmatrix} X & 0 \\ U_{32} & U_{33} \end{bmatrix} \]
Then $\sigma_i$ are Shilov coextensions of $\rho_i$, $V$ is an isometry, and
\begin{align*}
 \sigma_1(a) V &= P_{\L_1} \pi_1(a)|_{\L_1} P_{\L_1} U|_{\L_2} \\
 &= P_{\L_1} \pi_1(a) UP_{\L_2} = P_{\L_1} U \pi_2(a) P_{\L_2} \\
 &= P_{\L_1} U P_{\L_2} \pi_2(a) P_{\L_2} = V P_{\L_2} \pi_2(a)|_{\L_2} \\
 &= V \sigma_2(a) . \qedhere
\end{align*}
\end{proof}

We ask a bit more for what we will call the Ando property.
This is stronger than the classical Ando Theorem for $\AD$.
The weak version for $\AD$ is just Ando's Theorem.

\begin{defn}
A unital operator $\A$ satisfies the \textit{Ando property} if whenever
$\rho$ is a representation of $\A$ on $\H$ and $X\in\B(\H)$ is a 
contraction commuting with $\rho(\A)$, there is a fully extremal
coextension $\sigma$ of $\rho$ on a Hilbert space $\K$ and a
coextension of $X$ to an isometry on $\K$ which commute.

Likewise say that $\A$ satisfies the \textit{weak Ando property} if whenever
$\rho$ is a representation of $\A$ on $\H$ and $X\in\B(\H)$ is a 
contraction commuting with $\rho(\A)$, there is an extremal
coextension $\sigma$ of $\rho$ on a Hilbert space $\K$ and a
coextension of $X$ to an isometry on $\K$ which commute.

If $\A^*$ has the (weak) Ando property, say that $\A$ has the 
\textit{(weak) Ando* property}.
\end{defn}

It is apparent that the Ando property implies CLT and ICLT for $\A$;
and the weak Ando property implies WCLT and ICLT.
The converse of the latter fact follows the same lines as the 
classical deduction of Ando's theorem from CLT.
But the converse for the full Ando property is more difficult.
The difference is that an extremal coextension of a coextension is extremal,
but a fully extremal coextension of a coextension is generally not fully extremal.
So more care has to be taken, and a Schaeffer type construction makes it work.

\begin{prop} \label{P:wAndoWCLT}
A unital operator algebra $\A$ has the weak Ando property 
if and only if $\A$ has WCLT and ICLT.
\end{prop}

\begin{proof}
Let $\rho$ be a representation of $\A$ which commutes with a contraction $X$.
Assume WCLT and ICLT.
Coextend $\rho$ and $X$ to an extremal coextension 
$\sigma_1$ and a commuting contraction $Y_1$ using WCLT.
Then coextend $\sigma_1$ to $\rho_1$ and $Y_1$ to dilate to a commuting isometry
by ICLT.  Iterate these procedures recursively.  
The inductive limit has the desired properties. 
\end{proof}

\begin{thm} \label{T:AndoCLT}
A unital operator algebra $\A$ has the Ando property if and only if
$\A$ has CLT and ICLT. 
\end{thm}

\begin{proof}
Let $\rho$ be a representation of $\A$ which commutes with a contraction $X$.
Assume CLT and ICLT.
Use CLT to coextend $\rho$ and $X$ to a fully extremal coextension $\sigma$ and
a commuting contraction $Y$ on $\K \supset \H$.
Then use ICLT to coextend this to a Shilov dilation $\tau$ 
and commuting isometry $V$ on $\L = \K \oplus \L'$.
Since $\sigma$ is extremal, $\tau = \sigma \oplus \tau'$.
Write 
\[ V = \begin{bmatrix}Y&0\\Z&V'\end{bmatrix} \]
with respect to this decomposition of $\L$.
Then 
\[ \tau'(a) Z = Z \sigma(a) \qfor  a \in \A .\]
Since $\tau$ is Shilov, there is a maximal dilation $\pi$ which is an extension of $\tau$
on a Hilbert space $\M$ containing $\L$ as an invariant subspace. 
So $\L'$ is also invariant, as it reduces $\pi(\A)|_\L = \tau(\A)$.  
Let $P$ be the projection of $\M$ onto $\L'$.
Then 
\[ \pi(a)P = \tau'(a) \]
and thus
\[ \pi(a) (PZ) = P\tau'(a)Z  = (PZ) \sigma(a) \qforal a \in \A .\]

The representation $\sigma \oplus \pi^{(\infty)}$ is a fully extremal 
coextension of $\rho$. Moreover, $X$ coextends to $Y$ which coextends to
\[
 W = 
 \begin{bmatrix} 
 Y&0&0&0&\dots\\
 PZ&0&0&0&\dots\\
 0&I&0&0&\dots\\
 0&0&I&0&\dots\\
 0&0&0&\ddots&\ddots\\
 \vdots& \vdots& \vdots& \ddots&\ddots
 \end{bmatrix} .
\]                    
It is easy to see that $W$ is an isometry. 
The relations established in the previous paragraph ensure that 
$W$ commutes with $(\sigma \oplus \pi^{(\infty)})(\A)$.                                                                                                                                                                                                                                                                                                                                                                                                                                                                                                                                                           
Thus we have verified that $\A$ has the Ando property.
\end{proof}

This yields a strengthening of the classical Ando Theorem.
The usual Ando Theorem verifies the weak Ando property,
and hence ICLT.
But the disk algebra has SCLT.
So by Theorem~\ref{T:AndoCLT}, it has the Ando property.
We provide a direct proof that is of independent interest.

\begin{cor}\label{C:ando}
The disk algebra has the Ando property; i.e.\ 
if $A_1$ and $A_2$ are commuting contractions on $\H$, then they have
commuting isometric coextensions $V_i$ on a common Hilbert space $\K$. 
Moreover, we can arrange for $V_2$ to be a fully extremal coextension
$($i.e. $V_{A_2} \oplus U$, where $V_{A_2}$ is the minimal isometric coextension
and $U$ is unitary$)$.
\end{cor}

\begin{proof}
If $A$ is a contraction, let $V_A$ denote the unique minimal isometric coextension of $A$.
Let $V_i$ be isometric dilations of $A_i$, $i=1,2$, acting
on a common Hilbert space $\K = \H \oplus \K'$.
(The minimal dilations may not have additional subspaces of 
the same dimension, for example if one is already an isometry.  
In this case, we just add a unitary summand to one of them.)
Let $V=V_{A_1A_2}$.

Note that both $V_1V_2$ and $V_2V_1$ are isometries of the form
\[ \begin{bmatrix}A_1A_2 & 0 \\ * & * \end{bmatrix} .\]
So by the uniqueness of the minimal dilation, we can
write 
\[ V_1V_2 \simeq V \oplus W_1 \qand  V_2V_1 \simeq V \oplus W_2 ,\]
where $W_i$ is an isometry acting on a Hilbert space $\K_i$ (possibly of dimension 0).

Now we dilate $A_1$ to 
\[ S_1 := V_1 \oplus W_1^{(\infty)} \oplus W_2^{(\infty)} \]
and dilate $A_2$ to 
\[ S_2 := V_2 \oplus I_{\K_1}^{(\infty)} \oplus I_{\K_2}^{(\infty)} \]
on
\[ \K \oplus \K_1^{(\infty)} \oplus \K_2^{(\infty)} .\]
So 
\begin{align*}
 S_1S_2 \simeq V \oplus W_1 \oplus W_1^{(\infty)} \oplus W_2^{(\infty)}
 \simeq V \oplus W_1^{(\infty)} \oplus W_2^{(\infty)}
 \shortintertext{and} 
 S_2S_1 \simeq V \oplus W_2 \oplus W_1^{(\infty)} \oplus W_2^{(\infty)}
 \simeq V \oplus W_1^{(\infty)} \oplus W_2^{(\infty)} .
\end{align*}
These unitary equivalences both fix $\H$.
Therefore there is a unitary operator $U$ that fixes $\H$ so that
\[ S_2S_1 = U^* S_1S_2U .\]
Now define isometric dilations $U^*S_1$ of $A_1$ and $S_2U$ of $A_2$.
Then
\[ (U^*S_1)(S_2U) = U^* S_1S_2U = S_2S_1 = (S_2U)(U^*S_1) . \]
This yields a commuting isometric dilation.

Moreover, if 
\[ V_{A_2}=  \begin{bmatrix} A_2 & 0 \\ D_2 & J \end{bmatrix} , \]
then  $S_2U$ has the form
\[
 S_2U \simeq 
 \begin{bmatrix} A_2 & 0 & 0 \\ D_2 & J & 0 \\ 0 & 0 & I \end{bmatrix}
 \begin{bmatrix} I & 0 & 0 \\ 0 & U_{22} & U_{23} \\ 0 & U_{32} & U_{33} \end{bmatrix} 
 \simeq
 \begin{bmatrix} A_2 & 0 & 0 \\ D_2 & J & 0 \\ 0 & 0 & U' \end{bmatrix}
\]
The basic observation is that $J \oplus I$ is an isometry on $\H^\perp$ with range 
equal to the orthocomplement of $\ran D_2$.  The same is therefore true for the
lower $2\times2$ corner of $S_2U$.  By the uniqueness of the minimal isometric
dilation, this corner splits as a direct sum $J \oplus U'$ where $U'$must map onto
the complement of the range of $D$ and $J$.  So $U'$ is unitary as claimed.
\end{proof}

\begin{eg}
One might ask to dilate both $A_1$ and $A_2$ to commuting isometries
of the form $V_{A_i}\oplus U_i$ with $U_i$ unitary.  This is not possible,
as the following example due to Orr Shalit shows.
Let 
\[ A_1 = 0 \qand  A_2 = S , \]
where $S$ is the unilateral shift on $\H = \ltwo$.
Then 
\[ V_{A_2}=A_2=S\]
and 
\[ V_{A_1} = I \otimes S \quad\text{acting on } \H \otimes \ltwo .\]
Suppose that $U \in \B(\H_0)$ is unitary, and that 
\[ W_1 = U \oplus V_{A_1} \]
commutes with 
\[ W_2 = A_2 \oplus X = S \oplus X \quad\text{(with $\H$ appropriately identified).} \]
Then they can be written as
\[
 W_1 = \begin{bmatrix} 
 U & 0 & 0 & 0 & \dots\\ 0 & 0 & 0 & 0 & \dots\\ 0 & I & 0 & 0 & \dots\\ 0 & 0 & I & 0 & \dots \\
 \vdots & \vdots & \vdots & \ddots & \ddots
 \end{bmatrix}
 \AND W_2 = \begin{bmatrix} 
 X_{00} & 0 & X_{02} & X_{03} & \dots\\ 0 & S & 0 & 0 & \dots\\ 
 X_{20} & 0 & X_{22} & X_{23} & \dots\\ X_{30} & 0 & X_{32} & X_{33} & \dots \\
 \vdots & \vdots & \vdots & \ddots & \ddots
 \end{bmatrix} .
\]
Computing the two products yields
\[
\begin{bmatrix} 
 UX_{00} & 0 & UX_{02} & UX_{03} & \dots\\ 0 & 0 & 0 & 0 & \dots\\ 0& S & 0 & 0 & \dots\\
 X_{20} & 0 & X_{22} & X_{23} & \dots\\ X_{30} & 0 & X_{32} & X_{33} & \dots \\
\vdots & \vdots & \vdots & \ddots & \ddots
 \end{bmatrix}
 = \begin{bmatrix} 
 X_{00}U & X_{02} & X_{03} & X_{04} & \dots\\ 0 & 0 & 0 & 0 & \dots\\ 0 & S & 0 & 0 & \dots\\ 
 X_{20}U  & X_{22} & X_{23} & X_{24} & \dots\\ X_{30}U & X_{32} & X_{33} & X_{34} &\dots \\
 \vdots & \vdots & \vdots & \ddots & \ddots
 \end{bmatrix} .
\]
Equating terms shows that 
\[
 X_{00}U = UX_{00} , \quad  X_{jj}=S \FOR j \ge 2 
 \qand X_{ij} =0 \text{  otherwise}.
\]
Thus 
\[ W_2 \simeq X_{00} \oplus (S \otimes I) , \]
which is not $V_{A_2}$ direct sum a unitary.
\end{eg}

Katsoulis and Kakariadis \cite[Theorem~3.5]{KK} 
(in the special case of the identity automorphism) show that 
every tensor algebra of a C*-correspondence has the weak Ando property.
As noted in Example~\ref{E:SCLTtensor},
the Muhly-Solel Commutant Lifting Theorem \cite{MScenv}
shows that the tensor algebra $\T^+(E)$ of a 
C*-cor\-resp\-on\-dence $E$ satisfies SCLT.
Thus tensor algebras have the Ando property by Theorem~\ref{T:AndoCLT}.
This forms a large class of algebras with this property.
This includes all tensor algebras of graphs.
The case of the non-commutative disk algebra follows from \cite{DavKatsdilation}.
We would like to know if all semi-Dirichlet algebras have this property.

\begin{thm}[Katsoulis-Kakariadis] \label{T:AndoTensor}
The tensor algebra of a C*-correspondence has the Ando property.
\end{thm}

\begin{cor} \label{C:AndoGraph}
The tensor algebra of a directed graph has the Ando property.
In particular, the non-commutative disk algebras have this property.
\end{cor}

The proof in \cite{KK} proves more, providing a lifting for relations that
intertwine an automorphism. More will be said about this later when we 
discuss semi-crossed products.

Another proof can be based on an Ando type theorem of 
Solel \cite[Theorem~4.4]{SolCP}.
He shows that any representation of a product system over $\bN^2$ 
coextends to an isometric representation.
One can think of a product system over $\bN^2$ as a pair of
C*-cor\-resp\-on\-dences over $\bN$ together with commutation relations.
Here we only need a special case, where there is one C*-correspondence
$E$ over $\bN$, and take the second correspondence to be $F=\bC$
with the relations that $F$ commutes with $E$.
Then a representation $\rho$ of $\T^+(E)$ which commutes with a contraction
$X$ determines a representation of the product system.
Applying Solel's Theorem yields the desired isometric lifting.
This verifies ICLT and, in fact, the weak Ando property.
Now Theorem~\ref{T:AndoCLT} shows that $\T^+(E)$ has the Ando property.

We finish this section by giving the intertwining version of the Ando property.

\begin{thm} \label{T:Ando intertwine}
Suppose that $\A$ has the Ando property.
Suppose also that  $\rho_1$ and $\rho_2$ are representations of $\A$ 
on $\H_1$ and $\H_2$
and $X$ is a contraction in $\B(\H_2,\H_1)$ satisfying 
\[ \rho_1(a) X = X \rho_2(a)  \qforal a\in\A .\]
Given a fully extremal coextension $\sigma_2$ of $\rho_2$,
there exist a fully extremal coextension $\tilde\sigma_1$ of $\rho_1$ on $\K_1$,
a maximal representation $\pi$ determining a fully extremal coextension 
$\tilde\sigma_2 = \sigma_2 \oplus \pi$ of $\rho_2$ on $\K_2$, and 
an isometry $W \in \B(\K_2,\K_1)$ so that 
\[ P_{\H_1}W = XP_{\H_2} , \]
and 
\[ \tilde\sigma_1(a) W = W \tilde\sigma_2(a) \qforal a\in\A .\]
\end{thm}

\begin{proof}
The proof parallels the proof of Theorem~\ref{T:AndoCLT} using the
intertwining versions of CLT and ICLT. So we just sketch the plan.

One starts with the fully extremal coextension $\sigma_2$ of $\rho_2$ on $\K_2$.
By Theorem~\ref{T:CLT intertwine}, coextend $X$ to $Y$ and $\rho_1$
to a fully extremal $\sigma_1$ on $\K_1$ so that
\[ \sigma_1(a) Y = Y \sigma_2 (a) \qforal a \in \A .\]
Then use Proposition~\ref{P:ICLT intertwine} to coextend $Y$ to an isometry $V$ 
and $\sigma_i$ to a Shilov representation $\tau_i$ on $\L_i$ so that
\[ \tau_1(a) V = V \tau_2(a) \qforal a \in \A .\]
Since $\sigma_i$ are fully extremal, we have $\tau_i = \sigma_i \oplus \tau'_i$.
If $Z = P_{\L_1 \ominus \K_1} V |_{\K_2}$ is the $2,1$ entry of $V$ 
with respect to the decompositions $\L_i = \K_i \oplus (\L_i \ominus \K_i)$, then
\[ \tau'_1(a) Z = Z \sigma_2(a)  \qforal a \in \A .\]
Since $\tau_i$ are Shilov, so are $\tau'_i$.
So choose maximal representations $\pi_i$ on $\P_i$ which have invariant 
subspaces $\M_i$ so that $\pi_i|_{\M_i} \simeq \tau'_i$.
Set
\[
 \tilde\sigma_1 = \sigma_1 \oplus (\pi_1 \oplus \pi_2)^{(\infty)}
\qand
\tilde\sigma_2 = \sigma_2 \oplus (\pi_2 \oplus \pi_1)^{(\infty)} .
\]
Then the isometry 
\[ W \in \B(\K_2 \oplus (\P_2 \oplus \P_1)^{(\infty)}, \K_1 \oplus (\P_1 \oplus \P_2)^{(\infty)} \]
described in the proof of Theorem~\ref{T:AndoCLT} is the desired intertwiner.
The details are left to the reader.
\end{proof}

\section{Incidence Algebras} \label{S:incidence}

An \textit{incidence algebra} is a subalgebra $\A$ of the $n\times n$ matrices, $\fM_n$,
containing the diagonal algebra $\fD_n$ with respect to a fixed orthonormal basis.
Clearly, $\A$ is spanned by the matrix units $E_{ij}$ that it contains.
One can define a partial order $R$ on $\{1,2,\dots,n\}$ by 
\[ i \prec j  \text{  (or }  (i,j) \in R ) \IF E_{ij} \in \A .\]
This can be identified with a directed graph, but note that generally the algebra 
is not the same as the tensor algebra of the graph.
There is a reduced partial order obtained by identifying equivalent indices
\[ i \equiv j \IF i \prec j \AND j \prec i .\]
The representation theory of the algebra of
a partial order and its reduced partial order are related simply by multiplicity.

The algebra $\A \cap \A^*$ is a C*-algebra containing the diagonal $\fD_n$,
and is spanned by $\{E_{ij} : i \equiv j \}$.
A representation $\rho$ restricts to a completely contractive representation 
of $\A \cap \A^*$, and thus is a $*$-representation.  
So each diagonal matrix unit $E_{ii}$ is sent
to an orthogonal projection $P_i = \rho(E_i)$ onto a subspace $\H_i$.
Since $\rho$ is unital, 
\[ \H = \sum_{1\le i \le n} \!\!\oplus\, \H_i .  \]
In general, there are contractions $T_{ij} \in \B(\H_j,\H_i)$ so that 
\[ \rho(E_{ij}) = P_i T_{ij} P_j . \]
When $i \equiv j$, $T_{ij}$ is a unitary and $T_{ji} = T_{ij}^*$.
The homomorphism property shows that 
\[ T_{ik} = T_{ij}T_{jk} \quad \text{when  }i \prec j \prec k.\]
These relations are sufficient to determine an algebraic homomorphism.

Not all choices of contractions $\{T_{ij} : (i,j) \in R \}$ 
yield a completely contractive representation in general.  
However this does hold in some situations.
Paulsen and Power \cite{PP_nest} establish this for nest algebras.
Davidson, Paulsen and Power \cite{DPP_tree} establish this for
bilateral tree algebras.  These are the algebras where the reduced relation
is generated as a transitive relation by a directed bilateral tree 
(a directed graph with no loops).
Finally the class of all such algebras was determined by Davidson 
in \cite{D_local} as the interpolating graphs.  

Muhly and Solel consider unilateral tree algebras, which are the incidence algebras
$\A$ which are semi-Dirichlet.  They show that graphically means that the relation
is generated by a directed unilateral tree (each vertex is the range of at most one
edge and there are no loops).  
These incidence algebras actually coincide with the tensor algebra of the unilateral
tree because there is always at most one edge into each vertex.
For example, the algebra 
\[ \A = \spn\{E_{11}, E_{22}, E_{33}, E_{13}, E_{23}\} , \]
which is determined by the graph formed by edges from 
$v_3$ to each of $v_1$ and $v_2$, is a unilateral tree algebra.
However $\A^*$, which is determined by edges from $v_1$ and $v_2$ into $v_3$ 
is not a unilateral tree algebra, but it is a bilateral tree algebra.
See \cite[Chapter 5]{MS_hilbert} for a discussion of ``trees and trees''.

If $\rho$ is a representation of an incidence algebra $\A$ as above,
then a coextension $\sigma$ will act on a Hilbert space 
\[ \K = \sum_{1\le i \le n} \!\!\oplus\, \K_i , \]
where $\K_i \supset \H_i$ is the range of the projection $\sigma(E_{ii})$, 
and is determined by coextensions
$V_{ij}$ of $T_{ij}$ in $\B(\K_j,\K_i)$ of the form 
\[ V_{ij} = \begin{bmatrix}T_{ij} & 0\\D_{ij}&S_{ij}\end{bmatrix} \]
with respect to the decompositions $\K_i = \H_i \oplus \K'_i$.
The homomorphism property requires that 
\[ V_{ij}V_{jk}=V_{ik} \quad\text{whenever  } i\prec j \prec k .\]

In general, these are complicated relations to dilate.
One of the simplest examples where things get complicated is the $2k$-cycle 
graph $\C_{2k}$  for $k\ge2$.  This graph has vertices $\{1,2,\dots,2k\}$ and
edges 
\[ 2i+1 \succ 2i , \quad  2i+1 \succ 2i+2 \qand 1 \succ 2k .\]
The algebra for this graph has representations such that $\rho(e_{ij})=T_{ij}$
are all contractions, but $\|\rho\|_{cb}>1$ \cite[Theorem 2.2]{D_local}.
This is related to the famous example of Parrott \cite{Par} for $\rA(\bD^3)$ 
and a similar example due to Paulsen and Power \cite[Theorem 3.1]{PP_tensor} 
for the incidence algebra $\T_2 \otimes \T_2 \otimes \T_2$, where $\T_2$ is the 
algebra of $2\times2$ upper triangular matrices.  
Also see the exposition in \cite[Example 20.27]{DavNest}.

The case of bilateral tree algebras is more condusive to analysis.

\begin{thm} \label{T:tree_extremal}
Let $\A$ be a bilateral tree incidence algebra.
Then a representation $\rho$ is an extremal coextension if and only 
if each edge $E_{ij}$ is mapped to a partial isometry $V_{ij}$ such that
$V_{ij}^*V_{ij} = \rho(E_{ii})$.

If $\rho$ is a representation, then a coextension
$\sigma$ of $\rho$ is fully extremal if and only if it is extremal
and $\K_i = \H_i \vee V_{ij}\K_j$ for all edges of the tree. 
\end{thm}

\begin{proof}
The key observation from \cite{DPP_tree} is that every matrix unit in $\A$ 
factors uniquely as a product of matrix units in the tree $\T$, corresponding to the
combinatorial fact that every edge in the transitive relation corresponds to the
unique path on the tree from one vertex to another.
Thus if for each matrix unit in the tree, $T_{ij}=\rho(E_{ij})|_{\H_j}$ is coextended
to $V_{ij}$, then we can extend this definition to every matrix unit in a 
unique way; and the homomorphism property guarantees that each is 
a coextension of $T_{ij}$.
It is possible to coextend each $T_{ij}$ in the tree to an isometry from $\K_j$ to $\K_i$.
If more than one edge is entering a single vertex $i$, then one has to ensure
that $\K_i$ is large enough to accomodate all $T_{ij}$.  (Of course, the ranges
can overlap or even coincide.) 
By \cite{DPP_tree}, this representation is still completely contractive.
Thus to be extremal, each $\rho(E_{ij})$ needs to be isometric from $\K_j$ into $\K_i$.

Conversely, if each $T_{ij} = \rho(E_{ij})|_{\H_j}$ is an isometry on $\H_j$, 
consider any coextension $\sigma$ of $\rho$.
Then each $\sigma(E_{ii})$ has range $\K_i \supset \H_i$,
and each isometry $T_{ij}$ coextends to a contraction $S_{ij} = \sigma(E_{ij})|_{\K_j}$.
Therefore $S_{ij}|_{\H_j} = T_{ij}$ and
$S_{ij}|_{\K_j \ominus\H_j}$ is a contraction with range orthogonal to $\H_i$.  
Hence $\K' = \K \ominus \H$ is reducing. Therefore $\sigma$ decomposes 
as a direct sum of $\rho$ and $\sigma' = \sigma|_{\K'}$.

Fully extremal coextensions of $\rho$ are more complicated.
If we start with $\rho$, and coextend to an extremal $\sigma$, 
so that each $T_{ij}$ is coextended to an isometry $V_{ij}$,
it may be possible to dilate $\sigma$ to $\tau \succ_c \rho$
if there is some `room' left.  More precisely, if for some edge
$E_{ij}$ in the tree, we have 
\[ \K_i \ne \H_i \vee V_{ij}\K_j , \]
then one can extend $\sigma$ to $\tau$, a coextension of $\rho$, 
to use this extra space.  Pick a vector $e \in \K'_i = \K_i \ominus \H_i$
which is orthogonal to the range of $V_{ij}$.  
Form 
\[ \tilde{\K}_j = \K_j \oplus \bC f .\]
Extend $V_{ij}$ to 
\[ \tilde V_{ij} := V_{ij} + ef^* .\]
It is apparent that this is indeed a coextension of $\rho$ and an extension
of $\sigma$ which is not obtained as a direct sum.
So $\sigma$ is not fully extremal.

Conversely, if for all edges $i \prec j$ of the tree, 
\[ \K_i = \H_i \vee V_{ij}\H_j \]
then for any extension $\tau$ of $\sigma$, the operators
$\tilde V_{ij} = \tau(E_{ij})$ will have to 
map any new summand of $\tilde\K_j$, namely $\tilde\K_j \ominus \K_j$, 
to a space orthogonal to $\H_i$ because it is a coextension of $\rho$,
and orthogonal to $V_{ij}\K_j$ because $\tilde V_{ij}$ is a contraction.
Hence by hypothesis, it maps into $\tilde\K_i \ominus \K_i$.
This makes it apparent that $\tau$ splits as a direct sum of $\sigma$ and 
another representation.  Therefore $\sigma$ is fully extremal.
\end{proof}

\begin{cor} \label{C:tree_Shilov}
Every Shilov extension of a bilateral tree algebra $\A$ is extremal.
\end{cor}

\begin{proof}
A maximal representation $\pi$ of $\A$ extends to a $*$-representation of $\fM_n$.
In particular, each $\pi(E_{ij})$ is a unitary from $\H_j$ onto $\H_i$.
Thus any restriction $\sigma$ to an invariant subspace sends each vertex to a projection onto a 
subspace $\K_i \subset \H_i$ and each edge $E_{ij}$ to an isometry
of $\K_j$ into $\K_i$.
By Theorem~\ref{T:tree_extremal}, $\sigma$ is extremal.
\end{proof}

\begin{rem}
Since $\rho(E_{ii})$ can be $0$, an isometry can be vacuous.
So for the algebra $\A_n$ of Example~\ref{Eg:zigzag}, the coextensions
$\sigma_{2i}$ are all extremal.  But to be fully extremal, each edge has
to be mapped to an isometry with maximal range, so only $\sigma_n$ is fully extremal.
\end{rem}

\subsection*{Uniqueness}
Let us explain why only the unilateral tree algebras have unique minimal
fully extremal coextensions.
The unilateral tree algebras are semi-Dirichlet.  So this property is a
consequence of Theorem~\ref{T:sD}.

A typical example is 
\[ \A=\spn\{E_{21},E_{31}, \fD_3 \} \subset \fM_3 .\]
Consider a representation $\rho$ on $\H=\H_1\oplus\H_2\oplus\H_3$
where $\rho(E_{i1})=T_i$ for $i=2,3$.
Let 
\[ D_i = P_{\D_i}(I-T_i^*T_i)^{1/2} \in \B(\H_1,\D_i) , \]
where 
\[ \D_i = \ol{\ran(I-T_i^*T_i)} \] 
is the closed range of $(I-T_i^*T_i)^{1/2}$.
Coextend $\rho$ to $\sigma$ on $\K=\K_1 \oplus \K_2\oplus\K_3$ where 
\[ \K_1=\H_1 \qand \K_i = \H_i \oplus\D_i \FOR i=2,3 \]
by setting
\[ V_{i1} = \begin{bmatrix}T_i\\ D_i\end{bmatrix} .\]
This is easily seen to be a fully extremal coextension by the previous proposition.

Any other isometric coextension $\tau$ will act on a Hilbert space $\L$ where
$\L_i = \H_i \oplus \L'_i$, and 
\[ \tau(E_{i1}) = \begin{bmatrix}T_i& 0\\ U_iD_i& S_i\end{bmatrix}\qfor i=2,3 .\]
Here $U_i$ is an isometric imbedding of $\D_i$ into $\L_i'$ 
and $S_i$ is an isometry of $\L'_1$ into $\L'_i$ with range orthogonal to 
the range of 
\[ \begin{bmatrix}T_i\\ U_iD_i\end{bmatrix} .\]
A bit of thought shows that this splits as a direct sum of a representation on 
\[ \H_1 \oplus (\H_2\oplus U_2\D_2) \oplus (\H_3 \oplus U_3\H_3) \]
which is unitarily equivalent to $\sigma$ and another piece.

On the other hand, any graph which is a bilateral tree but not a unilateral tree
will have two edges mapping into a common vertex.
The compression to this three dimensional space yields the algebra $\A^*$. 
We explain why $\A^*$ has non-unique minimal fully extremal coextensions.

Fix $\theta\in (0,\pi/4)$, and consider a representation $\rho$ on 
\[ \H=\H_1\oplus\H_2\oplus\H_3 \]
where $\H_i = \bC^2$ given by 
\[
 \rho(E_{1i}) = T = 
 \begin{bmatrix}\cos\theta&0\\0&\sin\theta\end{bmatrix} \qfor i=2,3.
\]
Let 
\[
 D = (I-T^*T)^{1/2} = 
 \begin{bmatrix}\sin\theta&0\\0&\cos\theta\end{bmatrix} .
\]
Then for any unitaries $U_i$ in the $2\times 2$ unitary group $\fU_2$, 
we can coextend $\rho$ to an isometric
representation $\sigma$ on 
\[ \K = \bC^4 \oplus \bC^2\oplus\bC^2 \]
by setting
\[ V_i = \begin{bmatrix}T_i\\U_iD\end{bmatrix} \qfor i=2,3. \]
These are all fully extremal coextensions of $\rho$ by Theorem~\ref{T:tree_extremal}.
Moreover they are evidently minimal.

Consider when two such representations will be unitarily equivalent via 
a unitary $W$ which is the identity on $\H$, and thus has the form 
\[ (I_2 \oplus V) \oplus I_2 \oplus I_2 .  \]
Conjugation by $W$ carries $\sigma$
to the representation which replaces $U_iD$ by $VU_iD$ for $i=1,2$.
It is clear that one can arrange to match up the $1,2$ entry by
appropriate choice of $V$.  But that leaves no control over the $1,3$ entry.
The possible minimal fully extremal coextensions of $\rho$ 
are parameterized by $\fU_2$.

\subsection*{Commutant lifting}
Davidson, Paulsen and Power \cite{DPP_tree} showed that bilateral
tree algebras have ICLT, and hence MCLT.
They do not generally have SMCLT because of failure of unique dilations.
Among finite dimensional incidence algebras, these are precisely the
algebras with ICLT \cite[Theorem~4.6]{D_local}.

Muhly and Solel show that unilateral tree algebras satisfy SCLT.
This now can be seen from the fact that they have ICLT, whence MCLT.
So by Theorem~\ref{T:sD}, we can obtain unique minimal fully extremal  
coextensions, and that every Shilov extension is fully extremal.
So this implies SCLT.

The connected graphs which are unilateral trees and have adjoints which are unilateral
trees as well are evidently chains.  So the incidence algebras with this property
are just direct sums of finite dimensional nest algebras.
Since these algebras are Dirichlet, they have many good
properties from section~\ref{S:sD}.

\begin{prop} \label{P:treeWCLT}
Bilateral tree algebras have WCLT and WCLT*,
and well as the weak Ando and Ando* properties.
\end{prop}

\begin{proof}
If $\A$ is a bilateral tree algebra, then so is $\A^*$.  
So it suffices to prove WCLT.
Let $\rho$ be a representation of $\A$, 
and let $X$ be a contraction commuting with $\rho(\A)$.
By \cite{DPP_tree}, $\A$ has ICLT.
Hence by Theorem~\ref{T:ICLT} (ii), there is a Shilov coextension $\sigma$
and an isometric coextension $V$ of $X$ on $\K \supset \H$ which commute.
By Corollary~\ref{C:tree_Shilov}, $\sigma$ is extremal.
Thus $\A$ has WCLT and the weak Ando property.
\end{proof}

The goal now is to refine this construction to  obtain fully extremal coextensions
to obtain CLT and hence the Ando property. We begin by establishing the 
Ando property for $\T_2$, the $2\times2$ upper triangular matrices.
Since 
\[ \T_2 = \spn\{E_{11},E_{12},E_{22}\} , \]
a representation $\rho$ is determined by $\rho(E_{ii})=P_i=P_{\H_i}$,
where $\H=\H_1\oplus\H_2$, and a contraction $X \in \B(\H_2,\H_1)$
where $\rho(E_{12}) = P_1XP_2$.
A contraction $A$ commuting with $\rho(\T_2)$ commutes with $P_i$, 
and so has the form $A=A_1\oplus A_2$; plus we have $A_1X=XA_2$.
Thus Ando's Theorem for $\T_2$ can be reformulated as a commutant lifting theorem:

\begin{lem} \label{L:T2Ando}
Suppose that $A_i \in \B(\H_i)$ for $i=1,2$ and $X \in \B(\H_2,\H_1)$ 
are contractions such that 
\[ A_1X = X A_2 . \]
Then there are coextensions of $A_i$, $i=1,2$ and 
$X$ to isometries $\tilde A_i$ in $\B(\K_i)$ and 
$\tilde X$ in $\B(\K_2,\K_1)$ so that 
\[ \tilde A_1 \tilde X = \tilde X \tilde A_2 \qand  \K_1 = \H_1 \vee \tilde X \K_2 .\]
\end{lem}

\begin{proof}
The algebra $\T_2$ is a tree algebra, and so has the weak Ando property
by the previous proposition. 
Hence there are isometric coextensions $B_i$ of $A_i$ and $Y$ of $X$
so that 
\[ B_1Y = YB_2 , \]
acting on spaces $\L_i \supset \H_i$. 
We will modify this to obtain the desired form.

Observe that the commutation relation implies that the 
range  of $Y$ is invariant for $B_1$.
Let 
\[ \L = \H_1 \vee Y\L_2 ; \qand  B'_1 = P_\L B_1|_\L .\]
Let $Y' \in \B(\L_2,\L)$ be $Y$ considered an an operator into $\L$.
Then 
\[ B'_1 Y' = P_\L B_1 P_\L Y = P_\L B_1Y = P_\L YB_2 = Y'B_2 .\]
Also since $Y'$ is an isometry, 
\[ B_2 = Y^{\prime\ast} B'_1 Y' .\]
In particular, the commutation relations hold, and 
\[ \H_1 \vee Y\L_2 = \L  .\]
The contraction $B'_1$ may no longer be an isometry, but it
is a coextension of $A_1$.

Let $\tilde A_1$ be the minimal isometric dilation of $B'_1$ on $\K_1 \supset \L$.
Write $\K'_1 = \K_1 \ominus \L$. 
With respect to $\K_1 = \L \oplus \K'_1$, we have
\[ \tilde A_1 = \begin{bmatrix} B'_1 & 0 \\ C&D \end{bmatrix} .\]
Define 
\[ \K_2 = \L_2 \oplus \K'_1 \qand \tilde X = Y' \oplus I \in \B(\K_2,\K_1) .\]
Set
\begin{align*}
 \tilde A_2 &= \tilde X^* \tilde A_1 \tilde X \\&= 
 \begin{bmatrix} Y^{\prime\ast} & 0 \\ 0 & I \end{bmatrix}
 \begin{bmatrix} B'_1 & 0 \\ C&D \end{bmatrix}
 \begin{bmatrix} Y' & 0 \\ 0&I \end{bmatrix} \\ &=
 \begin{bmatrix} Y^{\prime\ast}B'_1Y' & 0 \\ CY' &D \end{bmatrix} \\&=
 \begin{bmatrix} B_2 & 0 \\ CY' &D \end{bmatrix}
\end{align*}
Thus $A_2$ is a coextension of $B_2$, and hence of $A_2$.
It is easy to verify that 
\[ \tilde A_1 \tilde X = \tilde X \tilde A_2 .\]
Since $\tilde A_1$ and $\tilde X$ are isometries, so is $\tilde A_2$.
Moreover, we now have 
\[ \H_1 \vee \tilde X \K_2 = (\H_1 \vee Y' \L_2) \oplus \K'_1 = \K_1 . \qedhere\]
\end{proof}

\begin{thm} \label{T:treeAndo}
Bilateral tree algebras have the Ando and Ando* properties.
\end{thm}

\begin{proof}
Again it suffices to establish the Ando property.
We first assume that the $\A \cap \A^* = \fD_n$, so that the
relation and reduced relation coincide.

Before proceeding, we make a few observations and set some notation.
Suppose that the tree $\T$ has vertices $v_i$ for $1 \le i \le n$.
Let $\rho$ be a representation of the algebra $\A$ commuting with a 
contraction $A$. Then since 
\[ \rho(v_i)= P_i = P_{\H_i} \]
are pairwise orthogonal
projections summing to the identity, we see that 
\[ A = \sum_{1 \le i \le n}\!\!\oplus\, A_i \]
where 
\[ A_i = A|_{\H_i} \in \B(\H_i) .\]
If $e_{ij}$ is an edge of the tree, let 
\[ T_{ij}=\rho(e_{ij})|_{\H_j} \in \B(\H_j,\H_i) .\]
We have 
\[ A_iT_{ij}=T_{ij}A_j , \]
and conversely any $A$ with these two properties commutes with $\A$.

A finite bilateral tree has an elimination scheme, in the sense that 
every bilateral tree has a vertex $v$ which has at most one edge $e$
such that either $s(e)=v$ or $r(v)=v$.
This allows a proof by induction.
So proceed by induction on the number of vertices.

If the graph has a single vertex, then it has no edges and $\A = \bC$.
It is trivial to verify Ando's property in this case.

Now suppose that the result holds for every bilateral tree on fewer
than $n$ vertices, and let $\T$ be a bilateral tree on $n$ vertices.
Let $\rho$ and $A$ be as above.
We may assume that $\T$ is connected; for otherwise 
we may dilate each component by the induction hypothesis.
So every vertex has an edge.
Let $v_{i_0}$ be a vertex with one edge $e$.
Restrict the representation to $\T\setminus\{v_{i_0},e\}$ acting on $\H_{i_0}^\perp$,
called $\rho'$, commuting with $A' = A|_{\H_{i_0}^\perp}$.
Use the induction hypothesis to coextend $\rho'$ to a fully extremal coextension
$\sigma'$ commuting with an isometric coextension $B'$ of $A'$.
Let 
\[\ \ran \sigma'(E_{ii}) =: \L'_i \qfor i \ne i_0 .\]
Then 
\[ B' =  \sum_{i \ne i_0} \oplus\, B'_i .\]

There are two cases: either $s(e)=v_{i_0}$ and $r(e) = j_0$ or
$r(e)=v_{i_0}$ and $s(e) = j_0$. Assume the former.
Let 
\[ \rho(e) = X \in \B(\H_{i_0},\H_{j_0}) .\]
Then 
\[ A_{j_0} X = X A_{i_0} .\]
Use Lemma~\ref{L:T2Ando} to coextend $A_{i_0}$, $A_{j_0}$ and $X$ to
isometries $\tilde A_{i_0}$, $\tilde A_{j_0}$ and $\tilde X$ so that
\[
 \tilde A_{j_0} \tilde X = \tilde X \tilde A_{i_0} \qand 
 \K_{j_0} = \H_{j_0} \vee \tilde X \K_{i_0} .
\]

We can decompose 
\[ \tilde A_{j_0} = V_{j_0} \oplus W_{j_0} , \]
where  $V_{j_0}$ is the minimal isometric coextension of $A_{j_0}$,
with respect to 
\[ \K_{j_0} = \K^0_{j_0} \oplus \K^1_{j_0} .\]
Similarly, decompose the isometry 
\[ B'_{j_0} = V_{j_0} \oplus W'_{j_0} \]
with respect to 
\[ \L'_{j_0} \simeq \K^0_{j_0} \oplus \L^{\prime 1}_{j_0} .\]
Define the coextension $\sigma$ of $\rho$ and isometric coextension
$B$ of $A$ as follows.
Set 
\[
 \L_i = \L'_i \oplus \K^1_{j_0} \quad\FOR  i \ne i_0 \qand
 \L_{i_0} = \K_{i_0} \oplus \L^{\prime 1}_{j_0} .
\]
Define
\begin{alignat*}{2}
 \sigma(e_{ij}) &= \sigma'(e_{ij}) \oplus (E_{ij}\otimes I_{\K^1_{j_0}})
 \FOR  i \ne i_0 &\qand 
 \sigma(e) &= \tilde X \oplus I_{\L^{\prime 1}_{j_0}} \\
 B_i &= B'_i \oplus W_{i_0} \quad \FOR  i \ne i_0 &\qand\ 
 B_{i_0} &=  \tilde A_{i_0} \oplus W_{j_0} .
\end{alignat*}
Here $E_{ij}\otimes I_{\K^1_{j_0}}$ is interpreted as the unitary that
maps the copy of $\K^1_{j_0}$ contained in $\L_j$ to the corresponding
copy in $\L_i$.
One needs only verify that each $\sigma(e_{ij})$ intertwines $B_j$ with $B_i$
and 
\[ \H_i \vee \sigma(e_{ij}) \L_j = \L_i . \]
Both of these facts are routine.
Thus a fully extremal coextension of $\rho$ is produced that 
commutes with an isometric coextension of $A$.
This verifies Ando's property.

The second case, in which the edge $e$ maps $v_{j_0}$ to $v_{i_0}$ 
is handled similarly by first dilating the graph on $n-1$ vertices, and 
producing a dilation of the one edge $e$ using Lemma~\ref{L:T2Ando}.
Then as above, split the two isometries over the vertex $v_{j_0}$ into
the minimal isometric coextension direct summed with another isometry; 
and then define an explicit coextension with the desired properties. 
\end{proof}

The following consequence is immediate.

\begin{cor} \label{C:treeAndo}
Bilateral tree algebras have CLT and CLT*.
\end{cor}

\section{The Fuglede Property} \label{S:FP}

We introduce another property of an abstract unital operator algebra reminiscent
of the classical Fuglede Theorem that the commutant of a normal operator
is self-adjoint.

\begin{defn}
Let $\A$ be a unital operator algebra with C*-envelope $\cenv(\A)$.
Say that $\A$ has the \textit{Fuglede Property} (FP) if for every
faithful unital $*$-representation $\pi$ of $\cenv(\A)$, one has
\[ \pi(\A)' = \pi(\cenv(\A))'. \]
\end{defn}

It is easy to characterize this property among abelian algebras.

\begin{prop} \label{P:FPabelian}
If $\A$ is an abelian operator algebra, then the following are equivalent:
\begin{enumerate}
\item $\A$ has the Fuglede property.
\item $\A$ is a function algebra.
\item $ \cenv(\A)$ is abelian.
\end{enumerate}
\end{prop}

\begin{proof}
If (1) holds, then for every $a \in \A$, $\pi(a)$ lies in $\pi(\A)'$ and hence in
$\pi(\cenv(\A))'$.  Thus $\pi(\cenv(\A))$ is abelian.  So (3) holds.
If (3) holds, then $\A$ is a function algebra since $\A$ separates points
by the definition of the C*-envelope.
Finally if (2) holds, then $\pi(\A)$ is contained in $\pi(\rC(X))$ which is
an algebra of commuting normal operators.  So the FP property follows
from the usual Fuglede Theorem.
\end{proof}

The following is a useful class of operator algebras which has the FP property.

\begin{prop} \label{P:FP_matrix}
Suppose that there is a family $\big\{ U_k = \big[a_{ij}^{(k)} \big] \big\}$ 
of unitary matrices $U_k \in \fM_{m_k,n_k}(\A)$ such that the set of matrix
coefficients $\{ a_{ij}^{(k)} \}$ generate $\A$ as an operator algebra.  
Then $\A$ has $FP$.
\end{prop}

\begin{proof}
If $B$ commutes with $\pi(A)$, then 
\[ B^{(m)} \pi(U_k) = \pi(U_k) B^{(n)} .\]
By the Fuglede--Putnam Theorem, we obtain 
\[ B^{*(m)} U_k = U_k B^{*(n)} .\]
Therefore $B^*$ commutes with each $a_{ij}^{(k)}$.  As these generate $\A$,
we deduce that $\pi(\A)'$ is self-adjoint.
\end{proof}

\begin{eg} \label{E:FPncdisk}
The non-commutative disk algebras of Popescu, $\fA_n$, are generated by a 
row isometry $S = [S_1\  \dots\ S_n]$.  The C*-envelope is the Cuntz algebra
$\O_n$.  As an element of $\fM_{1,n}(\O_n)$, $S$ is a unitary operator.
Thus $\fA_n$ has the FP property.  This property has been explicitly observed in 
\cite[Proposition~2.10]{DKP}.
\end{eg}

\begin{eg} \label{E:FPfAinfty}
The algebra $\fA_\infty$ generated by a countable family of isometries with
pairwise orthogonal ranges does not have the Fuglede property.  This is
because $*$-representations of the C*-envelope, $\O_\infty$, are determined
by any countably infinite family of isometries with pairwise orthogonal ranges,
and does not force the sum of these ranges to be the whole space.
In the left regular representation, the commutant of $\fA_\infty$ is the 
\wot-closed algebra generated by the right regular representation,
which is not self-adjoint.
\end{eg}

\begin{eg} \label{E:FPDA}
The algebra $\A_d$ of continuous multipliers on the Drury-Arveson
space $H^2_d$ is abelian, but the norm is not the sup norm.  So this is not a
function algebra.  Arveson \cite{Arv3} identifies the C*-envelope, which contains
the compact operators; so it is clearly not abelian.

Since $\A_d$ is a quotient of $\fA_d$, one sees that FP does not pass to quotients.
One can also see this by noting that there are quotients of functions
algebras which are not themselves functions algebras.
\end{eg}

\begin{eg} \label{E:FPgraph}
The tensor algebra of any finite graph has FP.
It does not follow immediately from Proposition~\ref{P:FP_matrix},
but does follow by a simple modification.
A finite graph $G=(V,E,r,s)$ consists of a finite set $V$ of vertices, 
a finite set $E$ of edges, and range and source maps $r,s:E\to V$.
The graph C*-algebra $\ca(G)$ is the universal C*-algebra generated by
pairwise orthogonal projections $p_v$ for $v\in V$ and partial isometries
$u_e$ for $e \in E$ such that 
\[
 \sum_{v\in V} p_v = 1 ,\quad 
 u_e^* u_e = p_{s(e)} \qand 
 \sum_{r(e) = v} u_e u_e^* = p_v 
\]
unless $v$ is a source, meaning that there are no edges with range $v$.
The tensor algebra of the graph $\T^+(G)$ is the non-self-adjoint subalgebra of 
$\ca(G)$ generated by 
\[ \{p_v, u_e : v\in V,\, e\in E \} .\]
Then $\cenv(\T(G)) = \ca(G)$ \cite{FMR,KatKribs2}.

Suppose that $\pi$ is a $*$-representation of $\ca(G)$ and $T\in \pi(\T(G))'$.
Then $T$ commutes with $\pi(p_v)=:P_v$, and thus 
$T = \oplus\sum_{v\in V} T_v$ where $T_v\in \B(P_v\H)$.
If there are edges with $r(e)=v$, say $e_1,\dots,e_k$, then
let $s(e_i) = v_i$ and consider
\[
 U = \big[ \pi(u_{e_1}) \ \dots \ \pi(u_{e_k}) \big] 
 \in \B(P_{v_1}\H \oplus \dots P_{v_k}\H, P_v\H) .
\]
This is a unitary element, and 
\[ T_v U = U (T_{v_1} \oplus \dots \oplus T_{v_k}) .\]
By the Fuglede-Putnam Theorem, we also obtain 
\[ T_v^* U = U (T_{v_1}^* \oplus \dots \oplus T_{v_k}^*) .\]
If there are no edges with range $v$, there is nothing to check.
We deduce as in Proposition~\ref{P:FP_matrix} that $\T^+(G)$ has FP.
\end{eg}

\begin{eg} The algebra of any finite $k$-graph has FP.
This is established as in the case of a 1-graph.
\end{eg}

\begin{eg}  \label{E:FPnest}
Let $\N$ be a nest (a complete chain of closed subspaces of a Hilbert space).
Set 
\[ \A = \T(\N) \cap \fK^+ , \]
where $\T(\N)$ is the nest algebra \cite{DavNest} and
$\fK^+ = \bC I + \fK$ is unitization of the space of compact operators.
By the Erdos Density Theorem, $\T(\N) \cap \fK$ contains a norm 1 approximate
identity; and thus $\A$ is weak-$*$ dense in $\T(\N)$.
Therefore its commutant is trivial, and coincides with the commutant of $\fK^+$, the
enveloping C*-algebra.
Moreover, $\fK^+$ is the C*-envelope because $\fK$ is the only ideal, and
the quotient $q$ of $\fK^+$ onto $\bC$ is clearly not isometric on $\A$.
A $*$-representation of $\fK^+$ has the form
\[ \pi(A) = q(A) I_{\K_0} \oplus A \otimes I_{\K_1} \]
on a Hilbert space $\K = \K_0 \oplus (\H \otimes\K_1)$.
By the earlier remark, the commutant of $\pi(\A)$ is seen to be
\[ \B(\K_0) \oplus (\bC I_\H \otimes \B(\K_1)) , \]
which is the commutant of $\pi(\fK^+)$.
So $\A$ has FP.

In particular, any finite dimensional nest algebra $\T \subset \fM_n$ has FP.
\end{eg}

\begin{eg}  More generally, let $\L$ be a completely distributive 
commutative subspace lattice (see \cite{DavNest}); 
and let $\Alg(\L)$ be the corresponding CSL algebra.
Let $\fM$ be a masa containing (the projections onto) $\L$.
Also let 
\[ \fN = \Alg(\L \cap \L^\perp) .  \]
Observe that $\L\cap\L^\perp$ is a
completely distributive Boolean algebra, and thus is atomic.
Therefore 
\[ \fN = \oplus\sum_i \B(\H_i) \]
is an $\ell^\infty$ direct sum
with respect to the decomposition $\H = \oplus\sum_i \H_i$, 
where $\H_i$ are the ranges of the atoms of $\L\cap\L^\perp$.

Also by complete distributivity,  $\Alg(\L)\cap\fK$ has a 
norm one approximate identity.
So again $\A = \Alg(\L)\cap\fK^+$ is weak-$*$ dense in $\Alg(\L)$.
It is straightforward to see that 
\[
 \cenv(\A) = 
 \{ \lambda I + \oplus\sum_i A_i : A_i \in \fK^+(\H_i) \AND \lim_i A_i= 0 \} . 
\]
The irreducible representations of $\cenv(\A)$ are unitarily equivalent to
compression to some $\H_i$ and the character that evaluates $\lambda$.
So every representation is a direct sum of these irreducible ones with
multiplicity.  
As in the nest case, it is straightforward to show that 
\[ \pi(\A)'=\pi(\cenv(\A))' .\]
So $\A$ has FP.
\end{eg}

We have one minor result relating FP with commutant lifting.

\begin{prop} \label{P:FP+MCLT}
If an operator algebra $\A$ has FP and MCLT, then it has ICLT.
\end{prop}

\begin{proof}
Suppose that $\rho$ is a representation of $\A$ on $\H$ 
and $X\in \B(\H)$ is a contraction  commuting with $\rho(\A)$.
Then by MCLT, there is a maximal dilation $\pi$ of $\rho$ and 
a contractive dilation $Y$ of $X$ which commutes with $\pi(\A)$
and has $\H$ as a common semi-invariant subspace.
Since $\pi$ is maximal, it extends to a $*$-representation of $\cenv(\A)$
which we also denote by $\pi$.
By the Fuglede property, $Y$ commutes with $\pi(\cenv(\A))$.
Hence $\ca(Y)$ is contained in $\pi(\cenv(\A))'$.

The standard Schaeffer dilation of $Y$ to a unitary $U$ on $\K^{(\infty)}$
has coefficients in $\ca(Y)$. So $U$ commutes with $\pi^{(\infty)}$.
This establishes ICLT.
\end{proof}

\begin{cor} \label{C:WCLT+FP}
If an operator algebra $\A$ has properties FP, WCLT and WCLT*, 
then it has the weak Ando property.

If an operator algebra $\A$ has properties FP, CLT and CLT*, 
then it has the Ando property.
\end{cor}

\begin{proof}
Theorem~\ref{T:MCLT} shows that WCLT and WCLT* imply MCLT.
So by the preceding proposition, we obtain ICLT.
By Proposition~\ref{P:wAndoWCLT} and Theorem~\ref{T:AndoCLT},
WCLT and ICLT imply the weak Ando property and
CLT and ICLT imply the Ando property.
\end{proof}

\section{Completely Isometric Endomorphisms} \label{S:c.i.endo}

In the category of operator algebras, the natural morphisms are
completely bounded maps.  Among the endomorphisms, those
that work best for semicrossed products are the completely isometric ones.
These are the analogue of the faithful $*$-endomorphisms of C*-algebras.
In this section, we investigate lifting completely isometric endomorphisms
of operator algebras to $*$-endomorphisms of some C*-cover.

Let $\Aut(\A)$ denote the completely isometric automorphisms of an operator
algebra $\A$.  If $\fA$ is a C*-algebra, then (completely) isometric automorphisms
are automatically $*$-automorphisms.  If $\A \subset \fA$, let
\[ \Aut_\A(\fA) = \{ \alpha\in\Aut(\fA) : \alpha(\A) = \A \} . \]
Similarly, let $\End(\A)$ denote the completely isometric unital endomorphisms of $\A$. 
Again, for a C*-algebra, these are faithful unital $*$-endomorphisms.
When $\A \subset \fA$, we let
\[ \End_\A(\fA) = \{ \alpha\in\End(\fA) : \alpha(\A) \subset \A \} . \]

First we begin with an easy result.

\begin{prop}\label{P:liftauto}
Let $\A$ be a unital operator algebra.  
Then every completely isometric automorphism of $\A$ 
lifts to a $*$-automorphism of $\cenv(\A)$ which fixes $\A$ $($as a set$)$.
Thus 
\[ \Aut_\A(\cenv(\A)) \simeq  \Aut(\A) \]
by restriction to $\A$.
\end{prop}

\begin{proof}
Let $\alpha \in \Aut(\A)$.  Consider $\A$ as a subalgebra of $\cenv(\A)$.
The map $\alpha$ takes $\A$ completely isometrically and isomorphically onto itself,
and the image generates $\cenv(\A)$ as a C*-algebra.
By the basic property of C*-envelopes, $\alpha$ extends to a $*$-homomorphism
$\tilde\alpha$ of $\cenv(\A)$ onto itself.  The kernel of $\tilde\alpha$
is a boundary ideal because the map is completely isometric on $\A$,
and hence is $\{0\}$.
Thus $\tilde\alpha$ is an automorphism which fixes $\A$ as a set.
The converse is evident.

The restriction of $\alpha \in \Aut_\A(\cenv(\A))$ to $\A$ is injective
since $\A$ generates $\cenv(\A)$ as a C*-algebra.
Thus the restriction map is an isomorphism.
\end{proof}

\begin{eg}
It is not true that every $\alpha\in\End(\A)$ lifts to an endomorphism of $\cenv(\A)$.
Take $\A = \AD$ and let $\tau \in \AD$ be the composition of a conformal map of
$\bD$ onto the rectangle 
\[ \{x+iy: -1 < x < 0 \AND |y|\le 10 \} \]
followed by the exponential map.  
Thus $\tau$ maps $\bD$ onto the annulus 
\[ A :=\{z : e^{-1} < |z| < 1 \} .\]
It follows that 
\[ \alpha(f) = f\circ\tau \]
is a completely isometric endomorphism.
However, since $\tau$ maps parts of $\bT$ into the interior of $\bD$, this map does not
extend to an endomorphism of $\CT$.  

Observe that $\alpha$ lifts to an endomorphism of $\rCD$ by 
\[ \tilde\alpha(f) = f \circ \tau .\]
Unfortunately this map is not faithful, as its kernel is 
\[ \ker \tilde\alpha = I(A) = \{f \in\rCD : f|_A = 0 \} .\]

The remedy is a bit subtle.  Let 
\[ X = \bigcap_{n\ge0} \tau^n(\ol{\bD}) .\]
This is a connected compact set with two key properties: 
\[ \tau(X) = X \qand \bT \subset X .\]
The latter holds because $\tau(\bT)$ contains $\bT$.
Now consider $\AD$ as a subalgebra of the C*-algebra $\rC(X)$.
The embedding is isometric because $\bT \subset X$.
This is a C*-cover by the Stone-Weierstrass Theorem.  
Here $\alpha$ extends to 
\[ \bar\alpha(f) = f \circ \tau .\]
This is a faithful $*$-endomorphism because $\tau$ is surjective on $X$.
\end{eg}

\begin{eg}
Here is a different example.
Take
\[ \A = \AD \oplus (c\otimes \T) \]
where $c$ is the space of convergent sequences
and $\T = \ca(T_z)$ is the Toeplitz algebra generated by the shift $T_z$ on $H^2$.
It is easy to see that 
\[ \cenv(\A) = \CT  \oplus (c\otimes \T) .\]

Write an element of $\A$ as $f \oplus (T_n)_{n\ge 1}$, 
where $\lim_{n\to\infty} T_n =: T_0$ exists.
Fix $\lambda \in \ol{\bD}$ and consider the map
\[ \alpha( f \oplus (T_n)_{n\ge1} ) = f(\lambda)I \oplus (T_f,T_{n-1})_{n\ge2} .\]
This is evidently a completely isometric unital endomorphism.
However one can restrict $\alpha$ to a map $\beta$ which
takes $\AD$ to a subalgebra of $\T$ by $\beta(f) = T_f$.
The range of $\beta$ generates $\T$ as a C*-algebra, which is non-abelian.
Therefore there is no extension  of $\beta$ to a homomorphism of
$\CT$ into $\T$.  
Thus $\alpha$ does not lift to a $*$-endomorphism of $\cenv(\A)$.

If $|\lambda|=1$, we can embed $\A$ into 
\[ \fA = \T \oplus (c\otimes \T) \]
in the natural way and extend $\alpha$ to the endomorphism
\[ \tilde\alpha(T\oplus(T_n)_{n\ge1}) = qT(\lambda) I \oplus (T,T_{n-1})_{n\ge2} \]
where $q$ is the quotient map of $\T$ onto $\CT$.

However if $|\lambda|<1$, evaluation at $\lambda$ is not multiplicative on $\T$,
so $\alpha$ does not lift to an endomorphism of $\fA$.
We can instead let 
\[ \fB = \bC \oplus \T \oplus (c\otimes \T)   \]
and imbed $\A$ by
\[ j(f\oplus (T_n)_{n\ge1}) = f(\lambda) \oplus T_f \oplus (T_n)_{n\ge1} .\]
Clearly $j$ is completely isometric.
The C*-algebra generated by $j(\AD)$ is all of $\fB$ because
\[ j(1) - j(z)^*j(z) = (1-|\lambda|^2) \oplus 0 \]
shows that $\bC \oplus 0$ is contained in this algebra.

Observe that evaluation at $\lambda$ is now a character of $\fB$.  
Moreover
\begin{align*}
 j(\alpha(f\oplus (T_n)_{n\ge1})) &= j(f(\lambda)I \oplus (T_f,T_{n-1})_{n\ge2}) \\
 &= f(\lambda) \oplus f(\lambda)I \oplus (T_f,T_{n-1})_{n\ge2} .
\end{align*}
So we may extend $\alpha$ to $\tilde\alpha\in\End(\fB)$ by
\[  \tilde\alpha(w \oplus T\oplus(T_n)_{n\ge1}) = w \oplus wI \oplus (T,T_{n-1})_{n\ge2} .\]
\end{eg}

\smallskip
A modification of Peters' argument \cite[Prop.I.8]{Pet} shows that we can extend
completely isometric endomorphisms to automorphisms of a larger algebra.

\begin{prop}\label{P:inductive}
If $\A$ is a unital operator algebra and $\alpha\in\End(\A)$, then there is
a unital operator algebra $\B$ containing $\A$ as a unital subalgebra and
$\beta\in\Aut(\B)$ such that $\beta|_\A = \alpha$.  
Moreover, $\B$ is the closure of $\bigcup_{n\ge0} \beta^{-n}(\A)$.
\end{prop}

\begin{proof}
First observe that the standard orbit representation makes sense for $\A$.
Let $\sigma$ be a completely isometric representation of $\A$ on a Hilbert space
$\H$ so that $\ca(\sigma(\A)) \simeq \cenv(\A)$.
Form the Hilbert space $\tilde\H = \H \otimes \ltwo$ and define 
\[
 \pi(a) = \sum_{n=0}^\infty\oplus \sigma(\alpha^n(a))
 \qand V = I \otimes S ,
\]
where $S$ is the unilateral shift.
Then it is evident that $(\pi,V)$ is an isometric covariant representation of $(\A,\alpha)$.
In particular, $\pi(\A)$ is completely isometric to $\A$.
Define the corresponding endomorphims $\tilde\alpha$ on $\pi(\A)$ by
\[
 \tilde\alpha(\pi(a)) 
 =\tilde\alpha \big( \sum_{n=0}^\infty\oplus \sigma(\alpha^n(a) \big)
 = \sum_{n=0}^\infty\oplus \sigma(\alpha^{n+1}(a)) 
 = \pi(\alpha(a)) 
\]
for $a \in \A$.

Form the injective system
\[
  \xymatrix{
  \pi(\A) \ar[r]^(.45){\tilde\alpha} \ar[d]^{\tilde\alpha}  &
  \pi(\A) \ar[r]^(.45){\tilde\alpha} \ar[d]^{\tilde\alpha} &
  \pi(\A) \ar[r]^(.45){\tilde\alpha} \ar[d]^{\tilde\alpha} &
  \dots \ar[r] &
  \B \ar[d]^{\beta}\\
  \pi(\A) \ar[r]^(.45){\tilde\alpha} &
  \pi(\A) \ar[r]^(.45){\tilde\alpha} &
  \pi(\A) \ar[r]^(.45){\tilde\alpha} &
  \dots  \ar[r] &
  \B 
  }
\]
Then $\B$ is a unital operator algebra containing $\A$ as a unital subalgebra
and $\beta$ is a completely isometric endomorphism.  
However it is evident that $\beta$ is surjective, so $\beta$ is an automorphism.

Finally, observe that $\B$ is the closure of $\bigcup_{n\ge0} \beta^{-n}(\A)$.
\end{proof}

Now we can use this to lift endomorphisms.

\begin{thm}\label{T:liftendo}
If $\A$ is a unital operator algebra and $\alpha\in\End(\A)$, then there is
a C*-cover $\fA$ of $\A$ and an endomorphisms $\tilde\alpha$ of $\fA$
such that $\tilde\alpha|_\A = \alpha$.
\end{thm}

\begin{proof}
Use Proposition~\ref{P:inductive} to lift $\alpha$ to an automorphism $\beta$
of a larger algebra $\B$.  Then apply Proposition~\ref{P:liftauto} to lift
it again to an automorphism $\tilde\beta$ of the C*-algebra $\fB = \cenv(\B)$.
Let $\fA$ be the C*-subalgebra of $\fB$ generated by $\A$.
Since $\tilde\beta|_\A = \alpha$, we see that $\A$ is invariant under $\tilde\beta$.
Hence so is $\A^*$.  Since $\fA$ is generated by $\A$ and $\A^*$, it is also
invariant under $\tilde\beta$.  So $\tilde\alpha = \tilde\beta|_\fA$ is the
desired $*$-endomorphism.
\end{proof}

While not all endomorphisms of $\A$ lift to $\cenv(\A)$, those that do lift behave 
well when lifted to larger algebras.  
This simplifies the hypotheses in some of the results in \cite{KK} as
explained in the next section.

\begin{prop} \label{P:fix Shilov ideal}
Let $\A$ be a unital operator algebra and let $\fA$ be a C*-cover.
Suppose that $\alpha \in \End_\A(\cenv(\A))$ and that $\beta \in \End_\A(\fA)$
such that $\beta|_\A = \alpha|_\A$.  
Then $\alpha q = q \beta$, and hence  $\beta$ fixes the Shilov ideal $\fJ_\A$, 
where $q$ is the canonical quotient map of $\fA$ onto $\cenv(\A)$.
\end{prop}

\begin{proof}
Observe that $\alpha q$ and $q \beta$ are $*$-homomorphisms 
of $\fA$ into $\cenv(\A)$ which agree on the generating set $\A$.  
Thus they are equal.
Hence 
\[ \fJ_\A = \ker \alpha q = \ker q \beta .  \]
It follows that
\[ \fJ_\A = \{ a \in \fA : \beta(a) \in \fJ_\A \} = \beta^{-1}(\fJ_\A) .\]
In particular, 
\[ \beta(\fJ_\A) \subset \fJ_\A .\qedhere \]
\end{proof}

Extremal and fully extremal coextensions behave well under automorphisms,
but not for endomorphisms.

\begin{prop} \label{P:extremal of auto}
Let $\rho$ be a representation of a unital operator algebra $\A$ 
with $($fully$)$ extremal coextension $\sigma$.
If $\alpha \in \Aut(\A)$, then $\sigma\circ\alpha$ is a
$($fully$)$ extremal coextension of $\rho\circ\alpha$.
\end{prop}

\begin{proof}
Let $\rho$ act on $\H$ and $\sigma$ act on $\K \supset\H$.
Suppose first that $\sigma$ is extremal.
Since $\H$ is coinvariant for $\sigma(\A)$, it is also coinvariant for
$\sigma(\alpha(\A))$.  Thus  $\sigma\circ\alpha$ is a
coextension of $\rho\circ\alpha$.
Suppose that $\tau$ is a coextension of $\sigma\circ\alpha$.
Then $\tau\circ\alpha^{-1}$ is a coextension of $\sigma$.
Hence it splits as 
\[ \tau\circ\alpha^{-1} = \sigma \oplus \phi .\]
Thus 
\[ \tau = \sigma\circ\alpha \oplus \phi\circ\alpha .\]
So $\sigma\circ\alpha$ is extremal.

A similar proof works for fully extremal coextensions of $\rho\circ\alpha$.
\end{proof}

\begin{eg} \label{Eg:extremal of endo}
Let 
\[ \A = \AD \oplus (c \otimes \rCD) \]
with elements $(f,(g_n)_{n\ge1})$
where $f\in\AD$, $g_n\in\rCD$ for $n\ge1$ and 
$\lim_{n\to\infty}g_n =: g_0$ exists.
Define 
\[ \alpha(f,(g_n)_{n\ge1}) = (f(0), (f,g_{n-1})_{n\ge2}) .\]
This is a completely isometric endomorphism.

A representation of $\A$ has the form
\[ \rho(f,(g_n)_{n\ge1}) = f(T) \oplus \sum_{n\ge1} \!\oplus\, \rho_n(g_n) \]
where $T$ is a contraction and $\rho_n$ are $*$-representations of $\rCD$.
It is straightforward to show that a representation of $\A$ is an extremal
coextension if and only if $T$ is an isometry.

Taking $T$ to be an isometry and $\rho_n$ all vacuous (the zero representation
on zero dimensional space), we have an extremal coextension $\rho$
such that 
\[ \rho\circ\alpha(f,(g_n)_{n\ge1}) = f(0) I .  \]
Since $0$ is not an isometry, this is not extremal.
\end{eg}

\section{A Review of Semicrossed Products} \label{S:scp}

Nonself-adjoint crossed products were introduced by Arveson \cite{Arv_meas, ArvJ} as
certain concretely represented operator algebras that encoded the action of a subsemigroup
of a group acting on a measure space, and later a topological space. 
McAsey, Muhly and Sato \cite{MMS, MM}  further analyzed such algebras, again relying on a concrete
representation to define the algebra. 
In \cite{Pet}, Peters gave a more abstract and universal definition of the \textit{semicrossed product}
of a C*-algebra by a single endomorphism. Actually he provides a concrete definition, but
then proves that it has the universal property which has become the de facto definition
of a semicrossed product.

One can readily extend Peter's definition  of the semicrossed product of a C*-algebra 
by a $*$-endomorphism to unital operator algebras and unital completely
isometric endomorphisms.  

\begin{defn}
Let $\A$ be a unital operator algebra and $\alpha \in \End(\A)$.
A \textit{covariant representation} of $(\A,\alpha)$ is a pair $(\rho,T)$ consisting
of a completely contractive representation $\rho : \A \to \B(\H)$ 
and a contraction $T\in \B(\H)$ such that
\[  \rho(a) T =  T \rho(\alpha(a)) \qforal a \in \A .  \]
The \textit{semicrossed product} $\A \times_\alpha \bZ_+$
is the universal operator algebra for covariant representations.
This is the closure of the algebra $\P(\A,\ft)$ of formal polynomials 
of the form $p=\sum_{i=0}^n \ft^i a_i$, where $a_i\in\A$, 
with multiplication determined by 
\[ a\ft = \ft \alpha(a) \]
and the norm
\[ \| p \| = \sup_{(\rho,T) \text{ covariant}} \big\| \sum_{i=0}^n T^i \rho(a_i) \big\| .\]
\end{defn}

This supremum is clearly dominated by $\sum_{i=0}^n \|a_i\|$; so this norm
is well defined.  Since this is the supremum of operator algebra norms, it is also
easily seen to be an operator algebra norm. 
By construction, for each covariant representation $(\rho,T)$, 
there is a unique completely contractive representation 
$\rho \times T$ of $\A \times_\alpha \bZ_+$ into $\B(\H)$ given by
\[ \rho \times T(p) = \sum_{i=0}^nT^i \rho(a_i) .\]
This is the defining property of the semicrossed product.

When $\fA$ is a C*-algebra, the completely isometric endomorphisms are the
faithful $*$-endomorphisms. 
Peters \cite{Pet} shows that when $\alpha$ is a faithful $*$-endomorphism of $\fA$,
the norm of $\fA \times_\alpha \bZ_+$ can be computed by using orbit representations.  
Let $\sigma$ be a faithful $*$-representation of $\fA$ on $\H$.  
Form the $*$-representation $\pi$ on $\H \otimes \ltwo$ by
\[ \pi(a) = \sum_{n=0}^\infty \oplus\, \sigma(\alpha^n(a))\]
and let $V = I \otimes S$,
where $S$ is the unilateral shift on $\ltwo$. 
It is evident that $(\pi,V)$ is a covariant representation of $(\fA,\alpha)$.
The corresponding representation $\pi\times V$ of $\fA \times_\alpha \bZ_+$ 
is known as the \textit{orbit representation} of $\sigma$.

\begin{thm}[Peters]
If $\alpha$ is a faithful $*$-endomorphism of a C*-algebra $\fA$, and
$\sigma$ is a faithful $*$-representation of $\fA$, then the 
orbit representation $\sigma\times V$ provides a completely isometric representation of 
\mbox{$\fA \times_\alpha \bZ_+$.} 
\end{thm}

Moreover, Peters \cite[Prop.I.8]{Pet} establishes Proposition~\ref{P:fix Shilov ideal}
in the case where $\A$ is a C*-algebra.  
It follows \cite[Prop.II.4]{Pet} that $\fA \times_\alpha \bZ_+$ is 
completely isometrically isomorphic to the subalgebra of the crossed product algebra
$\fB \times_\beta \bZ$ generated as a non-self-adjoint algebra by $j(\fA)$ and the
unitary $\fu$ implementing $\beta$ in the crossed product.  Kakariadis and the
second author \cite[Thm.2.5]{KK} show that this crossed product is the C*-envelope:

\begin{thm}[Kakariadis-Katsoulis] \label{T:KK_C*}
Let $\alpha$ be a faithful $*$-endo\-mor\-phism of a C*-algebra $\fA$
and let $(\fB,\beta)$ be the lifting of $\alpha$ to an automorphism described above.
Then
\[ \cenv(\fA \times_\alpha \bZ_+) = \fB \times_\beta \bZ .\]
\end{thm}

Since dilation theorems fail in many classical cases, such as commuting triples
of contractions \cite{Var,Par}, one can circumvent the issue by considering only isometric
covariant relations.  This semicrossed product was introduced by 
Kakariadis and the second author \cite{KK}.  
The results there have immediate application for us.

\begin{defn}
Let $\A$ be an operator algebra and let $\alpha \in \End(\A)$.
A covariant representation $(\rho,T)$ of $(\A,\alpha)$ is called
\textit{isometric} if $\rho$ is a complete isometry and $T$ is an isometry.
The \textit{isometric semicrossed product} $\A \times^{is}_\alpha \bZ_+$
is the universal operator algebra for isometric covariant representations.
This is the closure of the algebra $\P(\A,\ft)$ of formal polynomials 
of the form $p=\sum_{i=0}^n \ft^i a_i$, where $a_i\in\A$, under the norm
\[
 \| p \| = \sup_{\substack{(\rho,T) \text{ isometric}\\[-.5ex] \phantom{(\rho,T) }\text{ covariant}}}
 \big\| \sum_{i=0}^n T^i \rho(a_i) \big\| .
\]
\end{defn}

\begin{thm}[Kakariadis-Katsoulis] \label{T:KKis}
If $\A$ is a unital operator algebra and $\alpha \in \End_\A(\cenv(\A))$, then
$\A \times^{is}_\alpha \bZ_+$ is $($completely isometrically isomorphic to$)$
a subalgebra of $\cenv(\A) \times_\alpha \bZ_+$. 
Moreover, 
\[ \cenv(\A \times^{is}_\alpha \bZ_+) = \cenv(\cenv(\A) \times_\alpha \bZ_+) .\]
\end{thm}

More generally, they consider an arbitrary C*-cover $\fA$ of $\A$.
Let $\J_\A$ denote the Shilov boundary, 
i.e.\ the kernel of the unique homomorphism of $\fA$
onto $\cenv(\A)$ which is the identity on $\A$.
Suppose that $\alpha \in \End_\A(\fA)$ also leaves $\J_\A$ invariant.
Then it is easy to see that this induces an endomorphism 
$\dot\alpha \in \End_\A(\cenv(\A))$.  Hence $\A \times^{is}_{\dot\alpha} \bZ_+$
is (canonically completely isometrically isomorphic to)
a subalgebra of $\cenv(\A) \times_{\dot\alpha} \bZ_+$.  They show that
the same norm is also induced on $\A \times_\alpha \bZ_+$
as a subalgebra of $\fA \times_\alpha \bZ_+$.
This result should be considered in conjunction with 
Proposition~\ref{P:fix Shilov ideal}.

\begin{thm}[Kakariadis-Katsoulis] \label{T:KKcover}
If $\A$ is a unital operator algebra with C*-cover $\fA$ and 
$\alpha \in \End_\A(\fA)$ fixes $\J_\A$, then
$\A \times^{is}_\alpha \bZ_+$ is also 
$($canonically completely isometrically isomorphic to$)$
a subalgebra of $\fA \times_\alpha \bZ_+$. 
\end{thm}

In conjunction with Proposition~\ref{P:fix Shilov ideal}, we obtain:

\begin{cor} \label{C:KKcover}
If $\alpha \in \End_\A(\cenv(\A))$, $\fA$ is a C*-cover of $\A$ and
$\beta \in \End_\A(\fA)$ such that $\beta|_\A = \alpha|_\A$, 
then $\A \times_\beta \bZ_+$ is completely isometrically isomorphic
to $\A \times_\alpha \bZ_+$.
\end{cor}

\section{Dilating Covariance Relations} \label{S:covariance}

We consider the following problem: suppose that $\alpha \in \End(\A)$ 
lifts to a $*$-endomorphism of $\cenv(\A)$.
When is $\A \times_\alpha \bZ_+$ canonically completely
isometrically isomorphic to the subalgebra of $\cenv(\A) \times_\alpha \bZ_+$ 
generated by $\A$ and the element $\ft$ inducing the action $\alpha$.
To simplify statements, we will just say, in this case, that $\A \times_\alpha \bZ_+$ is a 
subalgebra of $\fA \times_\alpha \bZ_+$.
We are seeking general properties of $\A$ which make this true.

Commutant lifting properties can be seen as special cases of dilation theorems
for semicrossed products in the case of the identity automorphism.
The goal of this section is to provide several theorems which allow one to
obtain results about general semicrossed products from the various commutant
lifting properties.

The literature contains a number of dilation theorems for endomorphisms of operator algebras. 
Ling and Muhly \cite{LM} establish an automorphic version of Ando's theorem, which
is a lifting theorem for an action of $\bZ_+^2$. 
Peters \cite{Pet} and Muhly and Solel \cite{MSlift, MSextn} dilate actions of $\bZ_+$ on C*-algebras.
Our first  result in this section uses SCLT and models our dilation theorem for the 
non-commutative disk algebras \cite{DavKatsdilation}.
We wish to contrast the explicit dilation obtained here with the
more inferential one obtained in Theorem~\ref{T:Ando_covariant}
which requires only CLT.

\begin{thm} \label{T:dilate covariant}
Suppose that $\A$ is a unital operator algebra satisfying SCLT and ICLT.
Let $\alpha\in\Aut(\A)$.
Then every covariant representation $(\rho,T)$ of $(\A,\alpha)$
has a coextension $(\sigma,V)$ such that $\sigma$ is a fully extremal 
coextension of $\rho$ and $V$ is an isometry.
\end{thm}

\begin{proof}
Let $\sigma$ be a fully extremal coextension of $\rho$.
Then by Proposition~\ref{P:extremal of auto}, 
$\sigma\circ\alpha$ is also fully extremal.
By Corollary~\ref{C:extremal sums}, 
\[ \sigma \oplus(\sigma\circ\alpha) \] 
is a fully extremal coextension of $\rho \oplus (\rho\circ\alpha)$.
Now the covariance relation implies that for $a \in \A$,
\[
 \begin{bmatrix} \rho(a) & 0 \\ 0 & \rho(\alpha(a)) \end{bmatrix}
 \begin{bmatrix} 0 & T \\ 0 & 0 \end{bmatrix} =
 \begin{bmatrix} 0 & \rho(a) T \\ 0 & 0 \end{bmatrix} =
 \begin{bmatrix} 0 & T \\ 0 & 0 \end{bmatrix} 
 \begin{bmatrix} \rho(a) & 0 \\ 0 & \rho(\alpha(a)) \end{bmatrix}
\]
Using SCLT, we obtain a contractive coextension of 
$\begin{sbmatrix} 0 & T \\ 0 & 0 \end{sbmatrix}$
which commutes with $\sigma \oplus(\sigma\circ\alpha)(\A)$.
The $1,2$ entry is a contractive coextension $T_1$ of $T$ such that
\[ \sigma(a) T_1 = T_1 \sigma(\alpha(a)) \qforal a \in \A .\]
So 
\[ \begin{bmatrix} 0 & T_1 \\ 0 & 0 \end{bmatrix} \]
also commutes with $\sigma \oplus(\sigma\circ\alpha)(\A)$.

Now use ICLT to coextend this operator to an isometry $V$ commuting with
a Shilov coextension $\tau$ of $\sigma \oplus(\sigma\circ\alpha)$.
As $\sigma \oplus(\sigma\circ\alpha)$ is an extremal coextension, 
$\tau$ decomposes as 
\[ \tau = \sigma \oplus(\sigma\circ\alpha) \oplus \tau_1 .\]
With respect to this decomposition $\K \oplus \K \oplus \P$, we can write
\[ V = \begin{bmatrix} 0 & T_1 &0\\ 0 & 0 &0 \\ * & D & *\end{bmatrix} .\]
In particular, the commutation relation shows that
\[ \tau_1(a) D = D \,\sigma(\alpha(a)) \qforal a \in \A .\]
A direct summand of a Shilov extension is Shilov, so $\tau_1$ is Shilov.
Let $\pi_1$ be a maximal representation on $\L \supset \P$ so that
$\P$ is invariant, and $\pi_1|_\P = \tau_1$.
Then consider $P_\P D$ as an operator in $\B(\K,\L)$, and note that
\[ \pi_1(a) P_\P D = P_\P D \,\sigma(\alpha(a))  \qforal a \in \A .\]

Define a coextension of $\rho$ by
\[ \pi = \sigma \oplus \sum_{n\ge0}\strut^\oplus \pi_1\circ\alpha^n \]
acting on $\M = \K \oplus \L^{(\infty)}$.
Since $\sigma$ is a fully extremal coextension of $\rho$ 
and each $\pi_1\circ\alpha^n$ is maximal,
it follows that $\pi$ is a fully extremal coextension of $\rho$.

Now define an isometry $W$ on $\M$ by
\[
 W = \begin{bmatrix} T_1 & 0 & 0 & 0 & \dots\\
                               P_\P D & 0 & 0 & 0 & \dots\\
                               0 & I & 0 & 0 & \dots\\
                               0 & 0 & I & 0 & \dots\\
                               \vdots & & & \ddots &\ddots
        \end{bmatrix} .
\]
Then one readily verifies that 
\[ \pi(a) W = W \pi(\alpha(a)) \qforal a \in \A .  \]
This is the desired isometric coextension of the covariance relations.
\end{proof}

Observe that Theorem~\ref{T:dilate covariant} shows that if $\A$ has SCLT and ICLT,
then 
\[ \A\times_\alpha\bZ_+ = \A \times^{is}_\alpha \bZ_+ .\]
Thus by applying Theorem~\ref{T:KKis} to see that it imbeds into $\cenv(\A) \times_\alpha \bZ_+$.
Since $\alpha$ is an automorphism, the C*-envelope of this is just the full
crossed product $\cenv(\A) \times_\alpha \bZ$.

\begin{thm} \label{T:ISCLT_cenv}
Suppose that $\A$ has properties SCLT and ICLT, and $\alpha\in \Aut(\A)$.
Then $\A \times_\alpha \bZ_+$ is $($canonically completely isometrically isomorphic to$)$
a subalgebra of $\cenv(\A) \times_\alpha \bZ_+$.  Moreover,
\[ \cenv(\A \times_\alpha \bZ_+) = \cenv(\A) \times_\alpha \bZ .\]
\end{thm}

We wish to improve this theorem so that we require only CLT, not SCLT.
This requires a different approach, and does not yield a direct construction
of the isometric coextension of a covariant representation.

\begin{thm} \label{T:Ando_covariant}
Suppose that an operator algebra $\A$ has the Ando property, 
and $\alpha\in \Aut(\A)$.
Then 
\[ \A\times_\alpha\bZ_+ = \A \times^{is}_\alpha \bZ_+ .\]
Hence $\A \times_\alpha \bZ_+$ is $($canonically completely 
isometrically isomorphic to$)$a subalgebra of $\cenv(\A) \times_\alpha \bZ_+$.  
Moreover,
\[ \cenv(\A \times_\alpha \bZ_+) = \cenv(\A) \times_\alpha \bZ .\]
\end{thm}

\begin{proof}
Suppose that  $(\rho,T)$ is a covariant representation of $(\A,\alpha)$.
Let $\sigma_0$ be a fully extremal coextension of $\rho$ on a Hilbert space $\K_0$.
Then by Proposition~\ref{P:extremal of auto}, $\sigma_0\circ\alpha$ is also fully extremal.
Theorem~\ref{T:Ando intertwine} yields a fully extremal coextension $\sigma_1$
of $\rho_1$ on $\K_1$, an isometry  $V_0$ and a maximal dilation $\pi_0$ so that
\[ \sigma_1(a) V_0 = V_0 ((\sigma_0 \circ\alpha) \oplus \pi_0)(a) .\]
Recursively, we obtain fully extremal coextension $\sigma_{n+1}$
of $\rho_1$ on $\K_n$, an isometry  $V_n$ and a maximal dilation $\pi_n$ so that
\[ \sigma_{n+1}(a) V_n = V_n ((\sigma_n \circ\alpha) \oplus \pi_n)(a) .\]
Let
\[ \pi = \sum_{n\ge0} \sum_{k\in\bZ} \oplus (\pi_n \circ \alpha^k )^{(\infty)} .\]
Then set $\tilde\sigma_n = \sigma_n \oplus \pi$ acting on $\tilde\K_n$. 
Extending $V_n$ appropriately to an isometry $\tilde V_n$, we obtain
\[ \tilde\sigma_{n+1}(a) \tilde V_n = \tilde V_n (\tilde\sigma_n\circ\alpha)(a) .\]

Now define a representation on 
\[ \tilde\K = \sum_{n\ge0} \oplus \tilde\K_n \]
by
\[
 \tilde\sigma(a) = \begin{bmatrix}
  \tilde\sigma_0(a) & 0 & 0 & 0 & \dots\\
  0 & \tilde\sigma_1(a) & 0 &0 & \dots\\
  0 & 0 & \tilde\sigma_2(a) & 0 & \dots\\
  0 & 0 & 0 & \tilde\sigma_3(a) & \ddots\\
  \vdots & \vdots & \ddots & \ddots & \ddots
\end{bmatrix}
\]
and
\[
 \tilde V = \begin{bmatrix}
  0 & 0 & 0 & 0 & \dots\\
  \tilde V_0 & 0 & 0 & 0 & \dots\\
  0 & \tilde V_1 & 0 & 0 & \dots\\
  0 & 0 & \tilde V_2 & 0 & \ddots\\
  \vdots & \vdots & \ddots & \ddots & \ddots
\end{bmatrix}  .
\]
This is an isometric covariant representation which coextends the
contractive covariant representation on $\H \otimes \ltwo$:
\[
 \rho^{(\infty)}(a) = \begin{bmatrix}
   \rho(a) & 0 & 0 & 0 & \dots\\
  0 &  \rho(a) & 0 &0 & \dots\\
  0 & 0 &  \rho(a) & 0 & \dots\\
  0 & 0 & 0 &  \rho(a) & \ddots\\
  \vdots & \vdots & \ddots & \ddots & \ddots
\end{bmatrix}
\]
and
\[
 T \otimes S = \begin{bmatrix}
  0 & 0 & 0 & 0 & \dots\\
  T & 0 & 0 & 0 & \dots\\
  0 & T & 0 & 0 & \dots\\
  0 & 0 & T & 0 & \ddots\\
  \vdots & \vdots & \ddots & \ddots & \ddots
\end{bmatrix}  .
\]
This latter representation induces the same seminorm on $\A \times_\alpha \bZ_+$
as the covariant pair $(\rho^{(\infty)}, T\otimes U)$ on $\H \otimes \ltwo(\bZ)$,
where $U$ is the bilateral shift, because this representation is an inductive
limit of copies of \mbox{$(\rho^{(\infty)}, T\otimes S)$.} 
However $(\rho^{(\infty)}, T\otimes U)$ has a seminorm which clearly dominates
the seminorm arising from $(\rho,T)$.

It follows that 
\[ \A\times_\alpha\bZ_+ = \A \times^{is}_\alpha \bZ_+ .\]
The rest follows as in Theorem~\ref{T:ISCLT_cenv}.
\end{proof}

In particular, one gets a dilation theorem that we cannot find
by a direct construction. Indeed, even the isometric dilation is not explicitly
obtained, unlike the proof of Theorem~\ref{T:dilate covariant}.
This result considerably expands the class of operator algebras for
which we can obtain these results.

\begin{cor}\label{C:Ando_covariant}
Suppose that an operator algebra $\A$ has the Ando property, 
and $\alpha\in \Aut(\A)$.  
Then every covariant representation $(\rho,T)$ of $(\A,\alpha)$ dilates to
a covariant representation $(\pi,U)$ of $(\cenv(\A),\alpha)$
where $\pi$ is a $*$-representation of $\cenv(\A)$ and $U$ is unitary.
\end{cor}

The following result applies to endomorphisms, not just automorphisms.
This result was only recently established for the disk algebra \cite{DKdisk}.
The hypotheses are quite strong.

\begin{thm} \label{T:FS_dilate} 
Suppose that $\A$ is a unital operator algebra with SMCLT and FP.
Let $\alpha\in \End_\A(\cenv(\A))$.
Then $\A \times_\alpha \bZ_+$ is $($canonically completely isometrically 
isomorphic to$)$ a subalgebra of $\cenv(\A) \times_\alpha \bZ_+$.
\end{thm}

\begin{proof}
To establish that $\A \times_\alpha \bZ_+$ is completely isometric to a 
subalgebra of $\cenv(\A) \times_\alpha \bZ_+$, 
it suffice to show that if $(\rho,T)$ is a covariant representation
of $(\A,\alpha)$, then $\rho$ has a $*$-dilation $\pi$ of $\cenv(\A)$ on a Hilbert
space $\K \supset \H$ and a contraction $S$ dilating $T$ to $\K$ such that
$(\pi,S)$ is a covariant representation of $(\cenv(\A),\alpha)$ with
$\H$ as a coinvariant subspace.
For this then shows that 
\[
 \| \rho\times T(p) \| \le \| \pi \times S(p) \| \le \|p\|_{\cenv(\A) \times_\alpha \bZ_+} .
\]
The reverse inequality is evident.
Hence the norm on $\A \times_\alpha \bZ_+$ will coincide with the norm as a
subalgebra of $\cenv(\A) \times_\alpha \bZ_+$.

First dilate $\rho$ to a maximal dilation $\pi$.
This extends to a $*$-repre\-sent\-ation of $\cenv(\A)$ which we also denote by $\pi$.  
We may write:
\[
 \pi(a) = \begin{bmatrix} *&0&0\\ *&\rho(a) & 0 \\ * & * & *\end{bmatrix}
 \qfor a \in \A .
\]
Observe that the covariance condition is equivalent to
\[ 
 \begin{bmatrix} \rho(a) & 0 \\ 0 & \rho(\alpha(a)) \end{bmatrix}
 \begin{bmatrix} 0 & T \\ 0 & 0 \end{bmatrix} =
 \begin{bmatrix} 0 & T \\ 0 & 0 \end{bmatrix} 
 \begin{bmatrix} \rho(a) & 0 \\ 0 & \rho(\alpha(a)) \end{bmatrix} .
\]
Now $\rho \oplus (\rho\circ\alpha)$ dilates to a $*$-dilation 
\[ \pi\oplus(\pi\circ\alpha) .\]
So by SMCLT, 
there is a contraction dilating 
\[ \begin{bmatrix}0&T\\0&0\end{bmatrix} \]
which commutes with $(\pi\oplus(\pi\circ\alpha))(\A)$ of the form
\[ \tilde S = \begin{bmatrix}*&S\\ *&*\end{bmatrix} \]
so that  $\H\oplus\H$ is simultaneously semi-invariant for $(\pi\oplus\pi\circ\alpha)(\A)$ 
and $\tilde S$.
There is no loss in assuming that the $*$ entries are all $0$.
Commutation again means that $(\pi,S)$ is a covariant representation of $\A$.

Now by the Fuglede property, the adjoint 
\[ \begin{bmatrix}0&0\\S^*&0\end{bmatrix} \]
also commutes with $(\pi\oplus(\pi\circ\alpha))(\A)$.  
That means that
\[ \pi(\alpha(a)) S^* = S^* \pi(a) \qforal a \in \A . \]
Equivalently, since $\pi$ is a $*$-representation, 
\[  \pi(a^*) S = S \pi(\alpha(a^*)) \qforal a \in \A .\]
But the set of operators in $\cenv(\A)$ satisfying the covariance relations
is a closed algebra, and contains $\A$ and $\A^*$, whence it is all of $\cenv(\A)$.
Thus we have obtained the desired dilation to covariance relations for
$(\cenv(\A),\alpha)$.
Hence $\A \times_\alpha \bZ_+$ is (canonically completely isometrically isomorphic to) 
a  subalgebra of $\cenv(\A) \times_\alpha \bZ_+$.
\end{proof}

The following is immediate from the dilation theory for $\cenv(\A) \times_\alpha \bZ_+$
and Theorem~\ref{T:KKis}.
Combining this with Theorem~\ref{T:KK_C*}, one obtains an explicit description
of this C*-envelope.

\begin{cor}\label{C:FS_dilate}
Suppose that a unital operator algebra $\A$ has FP and SMCLT, 
and $\alpha\in \End_\A(\cenv(\A))$.
Then every covariant representation $(\rho,T)$ of $(\A,\alpha)$ dilates to
a covariant representation $(\pi,U)$ of $(\cenv(\A),\alpha)$
where $\pi$ is a $*$-representation of $\cenv(\A)$ and $U$ is unitary.
Moreover,
\[ \cenv(\A \times_\alpha \bZ_+) = \cenv(\cenv(\A) \times_\alpha \bZ_+) .\]
\end{cor}

\section{Further Examples} \label{S:examples}

\subsection*{The disk algebra.}
The first application yields a recent result about semicrossed products by
completely isometric endomorphisms for the disk algebra \cite{DKdisk}.
We note that endomorphisms which are not completely isometric are also 
treated there, but our results do not apply in that case.

The C*-envelope of the disk algebra $\AD$ is $\CT$, 
which is generated by the unitary element $z$.
The classical Fuglede Theorem shows that $\AD$ has FP.
Also the classical Sz.Nagy--Foia\c{s} Commutant Lifting Theorem yields
the properties SCLT and SMCLT. 
Ando's property is Corollary~\ref{C:ando}, which was
a strengthening of Ando's theorem.
As $\AD$ is Dirichlet, we have uniqueness of extremal coextensions and
of extremal extensions, which are also consequences of the original
Sz.Nagy theory.  As $\AD \simeq \AD^*$, we have SCLT* as well.
Indeed, $\AD$ has all of the properties studied in this paper.

Suppose that $\alpha \in \End_{\AD}(\CT)$.
Then $b = \alpha(z) \in \AD$; and has spectrum 
\[
 \sigma_{\AD}(b) = \sigma_{\AD}(z) = \ol{\bD} 
 \qand \sigma_{\CT}(b) = \sigma_{\CT}(z) = \bT.
\]
Thus $\ran(b)= \ol{\bD}$ and $\ran(b|_\bT)=\bT$.
It follows that $b$ is a non-constant finite Blaschke product.
We have $\alpha(f) = f \circ b$ for all $f \in \CT$.

\begin{thm}
Let $b$ be a non-constant finite Blaschke product, 
and let $\alpha(f) = f \circ b$ in $\End_{\AD}(\CT)$.
Then $\AD \times_{\alpha} \bZ_+$ is $($canonically completely isometrically 
isomorphic to$)$ a subalgebra of $\CT \times_\alpha \bZ_+$; and
\[
 \cenv(\AD \times_{\alpha} \bZ_+) =\cenv(\CT \times_\alpha \bZ_+) .
\]
This is explicitly described as $\rC(\S_{\alpha}) \times_{\beta} \bZ$
where $\S_{\alpha}$ is the solenoid
\[ \S_{\alpha} = \{ (z_n)_{n\ge1 } : z_n = b(z_{n+1}), z_n \in \bT, n \ge1 \} \]
and $\beta$ is the backword shift on $\S_{\alpha}$.
\end{thm}

\begin{proof}
The first statement follows from Theorems~\ref{T:FS_dilate}.  
The detailed description of the C*-envelope comes from the
Kakariadis-Katsoulis Theorem~\ref{T:KK_C*}.
\end{proof}

It is worth restating this theorem as a dilation result.

\begin{cor}
Let $b$ be a non-constant finite Blaschke product and 
suppose that $S$ and $T$ are contractions satisfying $ST=T\, b(S)$. 
Then there exist unitary operators $U$ and $V$, dilating $S$ and $T$ respectively, 
so that $UV=V\, b(U)$.
\end{cor}

\bigbreak
\subsection*{The non-commutative disk algebras.}
For $n\ge2$ finite, the non-commutative disk algebra $\fA_n$ has
the Cuntz algebra $\O_n$ as its C*-envelope. 
The Frazho-Bunce-Popescu dilation theorem \cite{Fra,Bun,Pop1} shows that
the minimal extemal coextension of a representation is unique.
This also follows because $\fA_n$ is semi-Dirichlet.
Popescu \cite{Pop3} proves the SCLT property in a similar manner to the
original proof of Sz.Nagy and Foia\c{s}.
The FP property follows from Proposition~\ref{P:FP_matrix}.

There are many distinct ways to extend the left regular representation to
a maximal representation (see \cite[\secsymb 3]{DP1}).
In particular, the minimal fully extremal extensions are not unique.
Nevertheless, $\fA_n$ has ICLT and MCLT.  
This follows from our paper \cite{DavKatsdilation} specialized to the identity automorphism.

The completely isometric automorphisms of $\fA_n$ are the analogues
of the conformal automorphisms of the ball $\bB_n$ of $\bC^n$.
These were first described by Voiculescu \cite{Voic} as $*$-automorphisms
of $\O_n$ which fix the analytic part.
These are the only such automorphisms of $\fA_n$ \cite{DP2}.
See also \cite{Pop4}.
Thus we recover our results on semicrossed products of $\fA_n$ in 
\cite{DavKatsdilation} as a consequence of Theorem~\ref{T:ISCLT_cenv}.

\begin{thm}[Davidson-Katsoulis]
If $\alpha\in\Aut(\fA_n) = \Aut_{\fA_n}(\O_n)$, then 
\[ \cenv(\fA_n\times_\alpha \bZ_+) =  \O_n\times_\alpha\bZ .\]
\end{thm}

It is also easy to determine $\End(\O_n)$.  
Every $n$-tuple of isometries $\ft_i \in \O_n$ such that $\sum_{i=1}^n \ft_i \ft_i^* = 1$
determines an endomorphism with $\alpha(\fs_i) = \ft_i$ by the universal property
of the Cuntz algebra.  
For the endomorphism $\alpha$ to leave $\fA_n$ invariant, 
it is then necessary and sufficient that $\ft_i$ belong to $\fA_n$.
Given that $\End(\O_n)$ is so rich, the following result seems surprising.

\begin{thm}\label{T:ncdiskendo}
For $n\ge2$ finite,
\[ \End_{\fA_n}(\O_n) = \Aut_{\fA_n}(\O_n) .\]
\end{thm}

\begin{proof}
We represent $\fA_n$ on $\Fock$ by the left regular representation $\lambda$
with generators $L_i = \lambda(\fs_i)$, where 
\[ L_i \xi_w = \xi_{iw} .\]
Note that 
\[ \ca(\lambda(\fA_n)) = \E_n \]
is the Cuntz-Toeplitz algebra,
and that 
\[ q: \E_n \to \E_n/\fK = \O_n \]
is the quotient by the compact operators.
Let $\fR_n$ denote the \wot-closed right regular representation algebra generated by
$R_i$, $1 \le i \le n$, where 
\[ R_i \xi_w = \xi_{wi} .\]
Then $\lambda(\fA_n)' = \fR_n$ \cite{AP, DP1}.
We use the notation $R_v\xi_w = \xi_{wv}$ for words $v\in\Fn$.

Suppose that $\ft_i \in \fA_n$ are isometries such that 
\[ \sum_{i=1}^n \ft_i \ft_i^* = 1 .\]
Let 
\[ T_i = \lambda(\ft_i) \qand  T = \big[ T_1 \ \dots \ T_n \big] .\]
Since $q$ is a complete isometry on $\fA_n$, we have $\|T\| = 1$.
Also each $T_i$ is an isometry: because if $\|\zeta\|=1$ and $\|T\zeta\| \ne 1$,
then 
\[ \|T(R_v\zeta)\| = \|R_v T\zeta\| = \|T\zeta\| .  \]
Since $R_1^n$ tends to $0$ weakly, we see that $T$ is not an 
essential isometry, contrary to assumption.
So $T_1,\dots,T_n$ are isometries in $\lambda(\fA_n)$ with pairwise orthogonal range.
Since 
\[ q(\sum_{i=1}^n T_i T_i^*) = 1 , \]
we deduce that 
\[ P=I-\sum_{i=1}^n T_i T_i^* \]
is a finite rank projection.

The range of each $T_i$ is a cyclic invariant subspace for $\fR_n$, with cyclic vector
$\zeta_i = T_i \xi_\mt$.  
Let $N = \ran P$. 
Then $N^\perp$ is the sum of the ranges of the $T_i$, and so 
it is invariant for $\fR_n$ with wandering space 
\[ W = \spn\{ \zeta_i : 1 \le i \le n\} .\]
Thus $N$ is coinvariant.  Let 
\[ A_i = P_N R_i|_N  \qand A = \big[ A_1 \ \dots \ A_n \big].\]
Then $A$ is a row contraction with a row isometric dilation 
\[ R = \big[ R_1 \ \dots \ R_n \big] .\]
The minimal row isometric dilation is unique \cite{Pop1}, and any other is the direct sum
of the minimal dilation with another row isometry.
Since $R$ is irreducible, this is the minimal dilation of $A$.

By \cite{DKS}, the wandering space $W$ of $N^\perp$ is given by
\[ W = (N + \sum_{i=1}^n R_i N ) \ominus N .\]
Note that 
\[
 I_N - \sum_{i=1}^n A_iA_i^* 
 = P_N( I - \sum_{i=1}^n R_iR_i^*)|_N 
 = (P_N\xi_\mt)(P_N\xi_\mt)^*  .
\]
This is non-zero because if $N$ were orthogonal to $\xi_\mt$, 
then $\xi_\mt$ would also be orthogonal to the invariant subspace it generates, 
which is the whole space.
Thus $N$ is not contained in $\sum_{i=1}^n R_i N$ because this space is
orthogonal to $\xi_\mt$.
So now we compute
\begin{align*}
 n = \dim W &= \dim(N + \sum_{i=1}^n R_i N ) - \dim N \\
 &\ge (1 + n\dim N) - \dim N \\
 &= 1 + (n-1) \dim N .
\end{align*}
Therefore $\dim N \le 1$; whence $\dim N = 1$ because no $n$-tuple of 
isometries $T_i$ in $\fL_n$ is of Cuntz type.

The only coinvariant subspaces of dimension one are $\bC \nu_\lambda$,
where $\nu_\lambda$ is an eigenvector of $\fR_n^*$ \cite{AP, DP1}.
These are indexed by points $\lambda$ in the open unit ball $\bB_n$ of $\bC^n$.
It now follows from the analysis in \cite{DP2} that $\alpha$ is an automorphism.
Briefly, one can compose $\alpha$ with an automorphism $\theta_\lambda$ 
so that $\lambda = 0$ and so $N = \bC \xi_\mt$.  Then 
\[ W = \spn\{\xi_i:1 \le i \le n\} .\]
So $\{\zeta_i\}$ form an orthonormal basis for $W$.
The unitary $U \in \U_n$ which takes $\xi_i$ to $\zeta_i$ induces a gauge unitary
$\tilde U$ which takes each $L_i$ to $T_i$, as this is the unique element of
$\fA_n$ with $T_i\xi_\mt = \zeta_i$.  
Hence $\theta_\lambda \alpha = \ad \tilde U$ is an automorphism;
whence so is $\alpha$.
\end{proof}

\bigbreak
\subsection*{Finite dimensional nest algebras}
A finite dimensional nest algebra can be described as the block upper
triangular matrices with respect to a decomposition 
$\H = \H_1 \oplus \dots\oplus\H_k$ of a finite dimensional Hilbert space into
a direct sum of subspaces.
These are the incidence algebras which are Dirichlet.
They have SCLT, SCLT*, ICLT, ICLT*, MCLT and the Ando property.
By Example~\ref{E:FPnest}, finite dimensional nest algebras have FP.

The only isometric endomorphisms are isometric automorphisms. 
These are unitarily implemented, and the unitary preserves the nest.
(Ringrose \cite{Ring} characterizes the isomorphisms between nest
algebras in infinite dimensions, and includes the more elementary
finite dimensional case.  See \cite{DavNest}.)
Hence the unitary has the form $U = U_1 \oplus \dots \oplus U_k$
with respect to the decomposition of $\H$.  Clearly $\ad U$ extends to
a $*$-automorphism of the C*-envelope $\B(\H) \simeq \fM_n$,
where $n = \dim \H$.

\bigbreak
\subsection*{Graph Algebras and Tensor algebras of C*-correspondences}

The tensor algebra $\T^+(E)$ of a C*-correspondence $E$ over a C*-algebra $\fA$ 
is semi-Dirichlet.  
Thus every Shilov coextension of a representation $\rho$, 
and in particular every extremal coextension of $\rho$, is fully extremal; 
and the minimal extremal coextension of $\rho$ unique.
So in particular, the minimal fully extremal coextension is a cyclic coextension. 
Muhly and Solel \cite{MScenv} show that the tensor algebra of a C*-correspondence
has SCLT.  
The C*-envelope is the Cuntz-Pimsner algebra $\O(E)$ \cite{MScenv, FMR, KatKribs3}.
Kakariadis and Katsoulis \cite{KK} establish that for every $\alpha\in\Aut_{\T^+(E)}(\O(E))$
such that $\alpha|_{\fA} = \id$,
the semi-crossed product $\T^+(E) \otimes_\alpha \bZ_+$ imbeds canonically, completely
isometically as a subalgebra of $\O(E)\otimes_\alpha \bZ$; and this is its C*-envelope.  
In particular, taking $\alpha = \id$, one obtains the Ando property, so it has ICLT and SCLT.

Thus, by circular reasoning, Theorems~\ref{T:dilate covariant} and 
\ref{T:ISCLT_cenv} apply.
The point however is that the dilation theorems for automorphisms follow
immediately once one has the appropriate commutant lifting theorems,
which basically deal with the identity automorphism.
In principle, and often in practice, this is much easier.

An important special case of a tensor algebra is the tensor algebra $\T^+(G)$ of a 
directed graph $G$.  Some of the properties are somewhat easier to see here.
In addition, by Example~\ref{E:FPgraph}, finite graph algebras have FP.

\bigbreak
\subsection*{Bilateral Tree Algebras}

In the case of a bilateral tree algebra $\A$, one readily sees that
$\cenv(\A)$ is a direct sum of full matrix algebras corresponding to the
connected components of the graph.
The automorphisms of finite dimensional C*-algebras are well understood.
Modulo inner automorphisms, one can only permute subalgebras of the same size.
Automorphisms of the tree algebra are more restrictive, and modulo those
inner automorphisms from unitaries in $\A \cap \A^*$, they come from automorphisms
of the associated directed graph.

Bilateral tree incidence algebras have the Ando property by Theorem~\ref{T:treeAndo}.
Hence by Theorem~\ref{T:Ando_covariant}, we obtain:

\begin{cor} \label{C:tree_cenv}
Let $\A$ be a bilateral tree algebra, and let $\alpha\in\Aut(\A)$.
Then 
\[ \cenv(\A \times_\alpha \bZ_+) = \cenv(\cenv(\A) \times_\alpha \bZ_+). \]
Hence if $(\rho,T)$ is a covariant representation of $(\A,\alpha)$, 
there is a maximal dilation $\pi$ of $\rho$ and a unitary dilation $U$ of $T$ 
so that $(\pi,U)$ is a covariant representation of $(\cenv(\A)<\alpha)$.
\end{cor}

It is an interesting question to look at the infinite dimensional WOT-closed versions.
A commutative subspace lattice (CSL) is a strongly closed lattice of 
commuting projections. A CSL algebra is a reflexive algebra whose
invariant subspace lattice is a CSL. Since every CSL is contained in
a masa, one can instead define a CSL algebra to be a reflexive algebra
containing a masa. The seminal paper, which provides a detailed structure
theory for these algebras, is due to Arveson \cite{Arv_CSL}.
See also \cite{DavNest}.

When dealing with weak-$*$ closed operator algebras, the class of all representations
is generally too large. Instead one restrict attention to weak-$*$ continuous
(completely contractive) representations. To apply the results from this paper, a 
weak-$*$ version needs to be developed.

A CSL algebra is a bilateral tree algebra if the lattice satisfies an measure
theoretic version of the discrete bilateral tree property.
We will not define this precisely here, but refer
the reader to \cite{DPP_tree} for the full story.
The approximation results from \cite{DPP_tree} show that every bilateral tree algebra
can be approximated in two very strong ways by a sequence of finite dimensional
subalgebras which are completely isometrically isomorphic to bilateral tree
incidence algebras.  These results should be a crucial step towards deducing similar 
dilation results for semicrossed products of these infinite dimensional
bilateral tree algebras by weak-$*$ continuous endomorphisms.


\end{document}